\documentclass{ims9x6}
\usepackage{hyperref,array}
\newcommand{\f}{\frac}
\newcommand{\1}{{\mathchoice {\rm 1\mskip-4mu l} {\rm 1\mskip-4mu l}{\rm 1\mskip-4.5mu l} {\rm 1\mskip-5mu l}}}

\DeclareMathOperator{\PP}{{\mathbb P}}
\DeclareMathOperator{\R}{{\mathbb R}}
\newcommand{\p}{\partial}
\DeclareMathOperator{\E}{{\mathbb E}}
\newcommand{\ep}{\varepsilon}

\def\t{\tilde}
\newtheorem{assumption}{Assumption}
\DeclareMathOperator{\N}{{\mathbb N}}
\DeclareMathOperator{\Z}{{\mathbb Z}}
\newcommand {\cU} {N}
\def\i{z}
\def\dst{\displaystyle}

\makeindex

\begin{document}

\chapter{Individual and population approaches \\  for calibrating division rates in population dynamics: Application to the bacterial cell cycle}

\markboth{M. Doumic \& M. Hoffmann}{Calibrating division rates}

\author{Marie Doumic and Marc Hoffmann}

\address{Sorbonne Universit\'es, Inria, UPMC Univ Paris 06 \\ Lab. J.L. Lions  UMR CNRS 7598, Paris, France
\\
marie.doumic@inria.fr \\
~
\\
University Paris-Dauphine, CEREMADE, \\ Place du Mar\'echal De Lattre de Tassigny, 75016 Paris, France \\ hoffmann@ceremade.dauphine.fr}

\begin{abstract}
Modelling, analysing and inferring  triggering mechanisms in population reproduction is fundamental in many biological applications. It is also an active and growing research domain in mathematical biology. In this chapter, we review the main results developed over the last decade for the estimation of the division rate in growing and dividing populations in a steady environment. These methods combine tools borrowed from PDE's and stochastic processes, with a certain view that emerges from mathematical statistics. A focus on the application to the bacterial cell division cycle provides  a concrete presentation, and may help the reader to identify  major new challenges in the field.
\end{abstract}

\vspace*{12pt}

\noindent {\bf Keywords:} cell division cycle, bacterial growth, inverse problem, nonparametric statistical inference, kernel density estimation, growth-fragmentation equation, growth-fragmentation process, renewal equation, renewal process, adder model, incremental model, asymptotic behaviour, long-term dynamics, eigenvalue problem, Malthusian parameter

\

\noindent {\bf Mathematics Subject Classification:}  35R30, 92B05, 35Q62, 62G05

\chaptercontents  

\section{Introduction}

\subsection{Biological motivation}

The study of stochastic or deterministic population dynamics, their qualitative behaviour and the inference of their characteristics is an increasingly important research field, which gathers various mathematical approaches as well as application fields. It benefits from the huge advances in gathering data, so that it is today possible not only to write and study qualitative models but also to calibrate them and assess their relevance in a quantitative manner. This chapter aims at contributing to review some recent advances and remaining challenges in the field, through the lense of a specific application, namely the bacterial cell division cycle. Guided by this application, we propose here a kind of roadmap for the mathematician in order to tackle a genuinely applied problem, coming from contemporary biology.

\subsubsection*{How does a population grow?}

Let us begin by describing the growth of a microcolony of bacteria, illustrated in the snapshots of Figure~\ref{fig:microcolony}
: out of one rod-shaped {\it E. coli} bacterium, a  colony rapidly emerges by the growth of each bacterium and its splitting into two daughter cells.

Nutrient being in large excess at this development stage, we can assume a steady environment. We also ignore the many fascinating questions arising from spatial considerations~\cite{duvernoy2018asymmetric,dell2018growing,doumic2020purely}, and focus on the two fundamental mechanisms at stake: growth and division. We can go back and forth between the population and the individual view: how does the knowledge of the growth and division laws of the individual lead to the knowledge of the population growth law, and conversely, to which extent observations on the population can help us infer the individual laws?

\begin{figure}
\begin{center}
{\includegraphics[width=\textwidth]{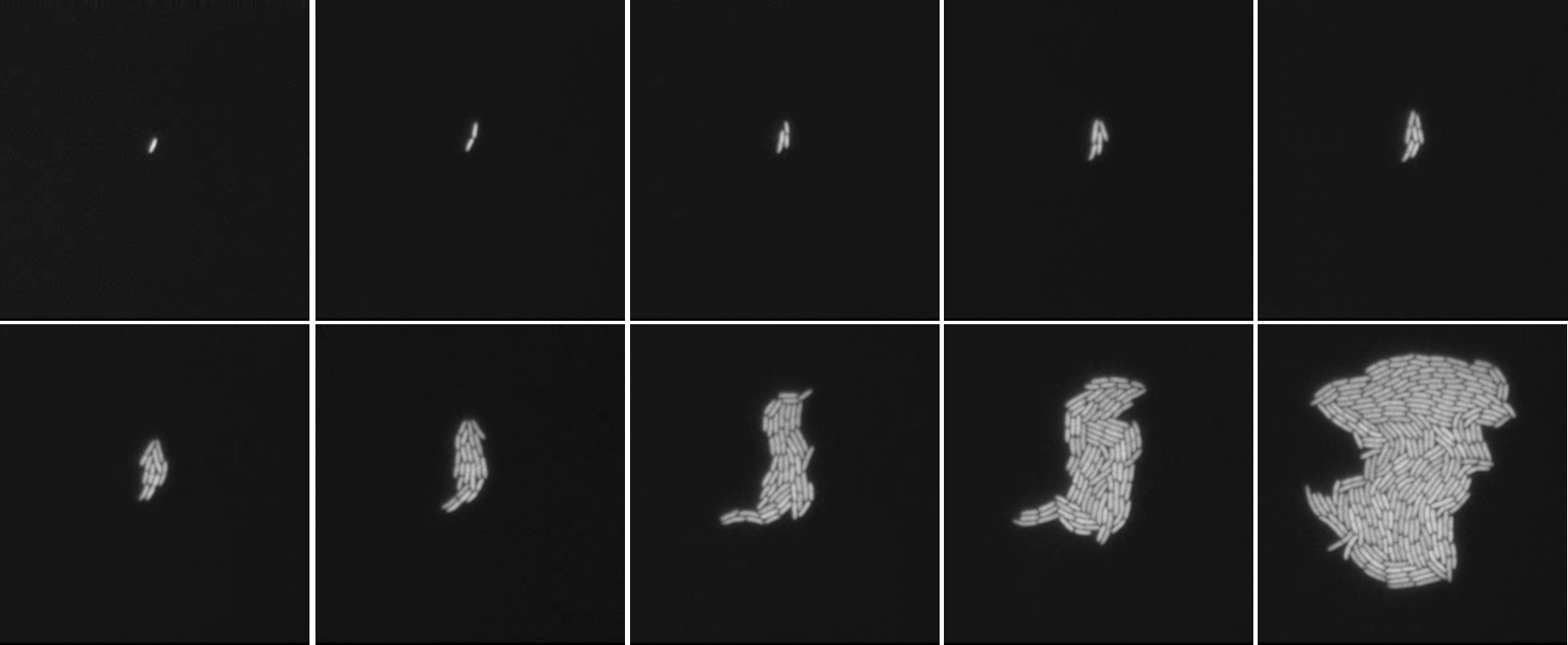}}
\caption{\label{fig:microcolony}\it
From left to right and top to bottom: Successive snapshots of the video~\url{https://doi.org/10.1371/journal.pbio.0030045.sv001}}
\end{center}
\end{figure}

\subsubsection*{How does a cell divide?}

To follow more easily individual characteristics of the cells over many generations, a microfluidic liquid-culture device called "the mother machine" has been developed in the last decade with an increasing success~\cite{balaban2004bacterial,Wang}. This is illustrated in Figure~\ref{fig:microflu}, a snapshot taken from the illustrating Movie S1, from~\cite{Wang}. The population is then reduced to independent lineages, but the two main mechanisms remain the same: growth and division. How do these two mechanisms coordinate each other? What triggers the bacterial division? To answer such questions, many studies have deciphered complex intracellular mechanisms, see for instance the recent review~\cite{meunier2021bacterial}, while others aim at inferring laws of growth and division out of the observation of population (as in Fig.~\ref{fig:microcolony}) or lineages (as in Fig.~\ref{fig:microflu}) dynamics. This last approach, which could be named phenomenological rather than mechanistic, constitutes the guideline of this chapter, and could be summed-up by the problematic: How much information on triggering mechanisms of growth and division can be extracted from  such data as in Figs.~\ref{fig:microcolony} and~\ref{fig:microflu}?

\begin{figure}
\begin{center}
{\includegraphics[width=\textwidth]{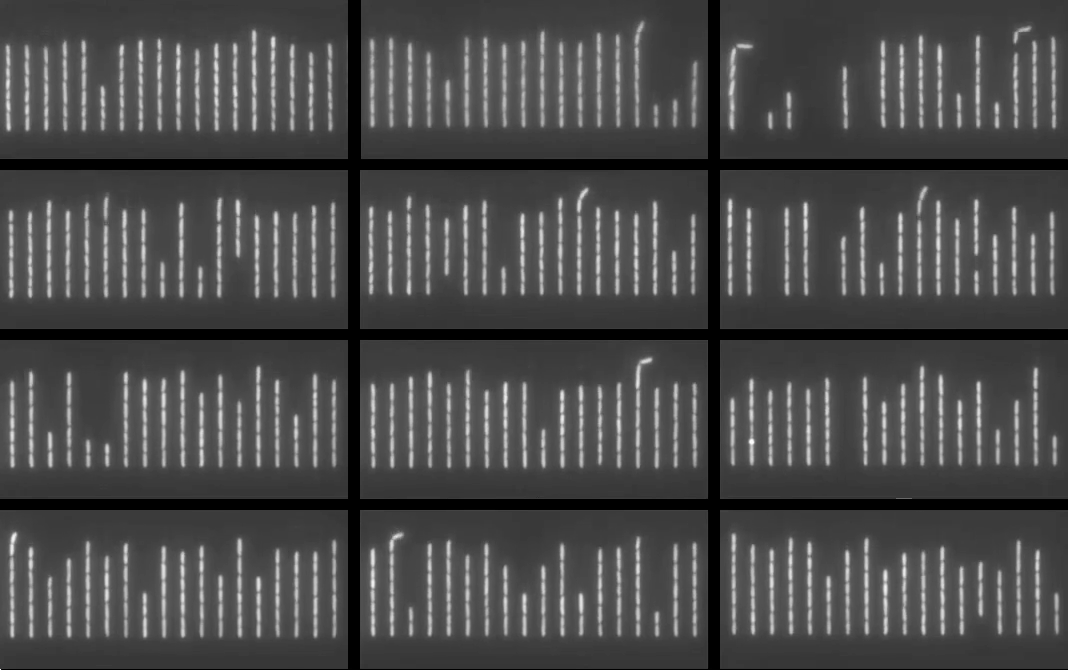}}
\caption{{\it From left to right and top to bottom: Snapshot of the video S1 of~\protect\cite{Wang}. In each channel bacteria grow and divide while remaining aligned, pushing new born cells towards the exit.\label{fig:microflu}}}
\end{center}
\end{figure}

%
\subsubsection*{Other applications}

We follow here the application to cell division; however, many other are possible, such as polymer fragmentation~\cite{XueRadford2013,beal2020division,tournus2021insights}, or mineral crushing~\cite{hoffmann2011statistical}, or yet other types of cell division cycles~\cite{basse}. This would lead us too far for this chapter, but we believe that many ideas gathered here for bacteria may apply to other fields.

\subsection{Outline of the chapter}

To serve as an outline of the chapter, let us enumerate the main steps towards the formulation of laws of growth and division in a process which is rather circular than linear in practice. This -- admittedly subjective -- methodological guideline is quite general, and many readers should recognize their own approach to their own problem; we then explain and specify them when applied to growing and dividing populations, and, more precisely, to the bacterial cell division cycle.

\begin{itemize}
\item {\bf Step 1) Preliminaries} - Sections~\ref{subsec:datanal} and~\ref{subsec:ass}

\begin{itemize}
\item {\bf Analyse data:} (or make the most of direct observations)

Biological data are often extremely rich, and only part of this richness is effectively analysed by experimenters, for instance through the use of averaged quantities rather than individual measurements. At the same time, the noise level is often very high, requiring an appropriate noise modelling approach. Mathematical methods at this first step are a combination of statistics, image analysis and interdisciplinary discussions between modellers and experimenters. This is carried out in our application case to bacterial growth and division in Section~\ref{subsec:datanal}.

\item {\bf Specify assumptions:} (or ``model and simplify'')

The data analysis carried out at Step 1 should lead to two types of hypothesis: simplifying ones, {\it a priori} justified by statistical quantifiers and which should be {\it a posteriori} verified by sensitivity analysis; and modelling ones, guiding conjectures on the underlying laws, which should be challenged at the end of the procedure. This is carried out in Subsection~\ref{subsec:ass}.

\end{itemize}

\item {\bf Step 2) Build a model} - Section~\ref{sec:model}

The goal of Section~\ref{sec:model} is to translate mathematically the assumptions done at Step 1. As suggested by the illustrations of Figs.~\ref{fig:microcolony} and~\ref{fig:microflu}, we may distinguish two types of models: individual-based and population models, leading to stochastic processes, branching trees, or integro-partial differential equations (PDEs).

\

\item {\bf Step 3) Analyse the model} - Section~\ref{sec:anal}

The models analyses, carried out in Section~\ref{sec:anal}, could seemingly be skipped to go directly to the conclusion by simulating the models as best as possible, with available simulation packages,  and comparing them to the data, using here again available fitting tools. We believe however that such approaches not only lack rigour but also risk missing enlightening information. Many success stories in many application domains could illustrate this general comment; in our very field, the cornerstone and foundation of the inverse problem solutions lies in the long-time asymptotics of the population models, as first described by B. Perthame and J. Zubelli~\cite{PZ}. For the polymer breakage application, not developed in this chapter, the use of time asymptotics as proved in~\cite{MischlerRicard} drastically simplified and justified the calibration of the fragmentation model~\cite{beal2020division,tournus2021insights}. In Section~\ref{sec:anal}, we thus review some of the main theoretical results which may prove useful for the following steps.

\

\item {\bf Step 4) Calibrate the model} - Section~\ref{sec:inverse}

We can do here the same remark as for Section~\ref{sec:anal}: standard calibration methods are often applicable, reducing this step to the choice of up-to-date softwares. A main drawback would be the difficulty to assess the confidence we may have in the results obtained; above all, theoretical error estimation inform us on the quantity of information available in the data collected, thus able to inspire design of new experiments. We thus develop in Section~\ref{sec:inverse} the inverse problem analyses carried out in the past decade for the different observation schemes, detailing some of the proofs in the illuminating example of the renewal model and in the "ideal mitosis" case, and reviewing the results obtained for more complex cases.

\

\item {\bf Step 5) Conclusion} - Section~\ref{sec:back}

At this stage, we have all the necessary tools on hand to confront a model to the data, and conclude on the validity on the modelling assumptions made at Step 2. In Section~\ref{sec:back}, a protocol to confront model to data is proposed and applied to experimental data for bacteria. It is however a step where many questions remain open, concerning the design and analysis of statistical tests as well as the formulation of new or more detailed models. It usually lead to boostrapping the methodology by circulate another round from Steps 1 to 5.

\end{itemize}


\subsection{First step: data analysis}
\label{subsec:datanal}
To build models that can be compared to the data, a preliminary step consists in having clear ideas on what can be measured and how. We first distinguish two types of data collection and two observation schemes, and then give some examples of what can be directly obtained from the data.

\subsubsection*{Data description: two types of datasets}



We have already shown in Figs~\ref{fig:microcolony} and~\ref{fig:microflu} two modern experimental settings to observe bacterial growth. In these two settings, a picture is taken at given time intervals - typically here, every $1$ to $5$ min - giving access, through image analysis, not only to time-dependent (noisy) samples of sizes but also to genealogical data and to the knowledge of the time elapsed since birth, of size dimensions at birth and of size dimensions at division, and of the ratio between the mother cell size at division and the offspring sizes. In the sequel, we call such cases  {\it individual dynamics data} , meaning that  some knowledge about the growth and division processes may be directly inferred from the data, where individual dynamics are collected.

 However, there are other situations, for instance when  {\it ex vivo} samples are collected, when only cheaper devices are available, or when the quantities of interest are not dynamically measured, or  yet for other applications such as protein fibrils. In all these cases, the experimenter can only observe size distribution of particles of interest, taken at one or several time points, but without being able to follow each one so that no individual observation of growth or division may be done. In the sequel, we call such cases {\it population point data}, meaning that we have information on the population dynamics or on point individual data, but no access to individual dynamics.


As detailed in Section~\ref{sec:inverse}, each type of data collection raises different problems. Schematically, in the population point data, one has to choose a given model, and the estimation questions  at stake are to determine which parameters may be inferred from the data and how precise this inference is. In the individual dynamics data, data are much richer and so are the questions to tackle:  first, as for the population point data,  estimate model parameters, and second, assess quantitatively how accurate the model is,  evaluate to which extent it can be enriched without over-estimation, and compare it with other models.

\subsubsection*{Data description: genealogical observation vs. population observation }

In the individual dynamics data cases, we have seen two distinct  experimental settings: either we follow the overall population until a certain time, as in Fig.~\ref{fig:microcolony} - we call this case the {\it population case}, corresponding to $k=2$ in the following, $k$ being the number of children at division - or we follow only one given lineage, since at each division we keep observing one out of the two daughter cells - we call this case the {\it genealogical case}, corresponding to $k=1$ in the following. As explained in Section~\ref{sec:model}, the mathematical model needs to adapt to these two cases, and so need the mathematical analysis and the model calibration.

At first sight, the population point data collection seems to apply only to the the population observation scheme ($k=2$), taking sparse pictures of a population state at some times. However, it could be imagined that in a given genealogical observation experiment ($k=1$) we are not able to determine when or where a cell divides; this will be the case for instance if the timesteps of observation are too large. It is thus also interesting to develop models and calibration methods for this seemingly strange but not irrealistic case.

For the population observation case $k=2$, if the cells are in constant growth conditions (unlimited nutrient and space),  the so-called Malthusian parameter characterises the exponential growth of the population. We denote it here $\lambda,$ meaning that the population grows like $e^{\lambda t}$ - rigorous meanings of this statement are provided in Section~\ref{sec:anal}. The Malthusian parameter can be measured in various ways for many different experimental conditions - as the total biomass increase for instance. Equivalently, biologists often refer to the doubling time $T_2$ of the population, with the immediate relation $T_2=\f{\ln(2)}{\lambda}$.


\subsubsection*{Data analysis: size distributions}

To each observation scheme correspond different types of data and specific measurement noise. Let us here gather some of the most frequent information we can extract.\\

It is possible to extract size dimension distributions from the two types of data collection described above. For instance, for a given sample of $n$  cells in a {\it E. coli} population, we measure their lengths at a given time $t$ to obtain
$$\big(x_1(t),\cdots, x_n(t)\big),$$
the width being considered roughly constant among cells.
 As a first approximation, we may assume that the observation is a $n$-drawn of a random vector
 $$\big(X_1(t),\cdots,X_n(t)\big),$$
 where the random variables $X_i(t)$ are independent, with common distribution $\mu(t)$; this assumption is obviously not valid, since we ignore the underlying dependent structure. Yet, it may well happen -- and this will be extensively discussed later -- that the sample behaves approximately for certain linear statistics like an $n$-drawn of a common distribution~\cite{hoffmann2016nonparametric}. In particular,  we may at least expect the convergence of the empirical measure
\begin{equation}
\label{eq:muN}
\mu_n=\f{1}{n}\sum\limits_{k=1}^n \delta_{X_k(t)}
\end{equation}
to $\mu(t)$ in a weak sense ~\cite{van1996weak}, possibly quantified by a distance like Wasserstein~\cite{fournier2015rate,weed2019sharp} that metrizes weak convergence. Assuming that $\mu$ has a density, and using for instance histograms or adaptive kernel density estimators
 we may  obtain a smooth estimate for the size distribution $\mu(t)$. This is treated at length in Section~\ref{sec:inverse}
. Assuming moreover that $\mu(t) = \mu$ does not depend on time, an assumption that can (and  will) be justified in some cases by the asymptotic analysis carried out in Section~\ref{sec:anal}, also known in biology as cell size homeostasis, we may concatenate all data taken at all time to get larger samples. This is illustrated in Fig.~\ref{fig:distrib:all} Right, where we have shown the length-distribution of cells taken at any time out of genealogical observation ($k=1,$ data from~\cite{Wang}) in blue and population observation in green ($k=2,$ data from~\cite{Eric}).

\begin{figure}
\label{fig:distrib}
\begin{center}
\includegraphics[width=\textwidth]{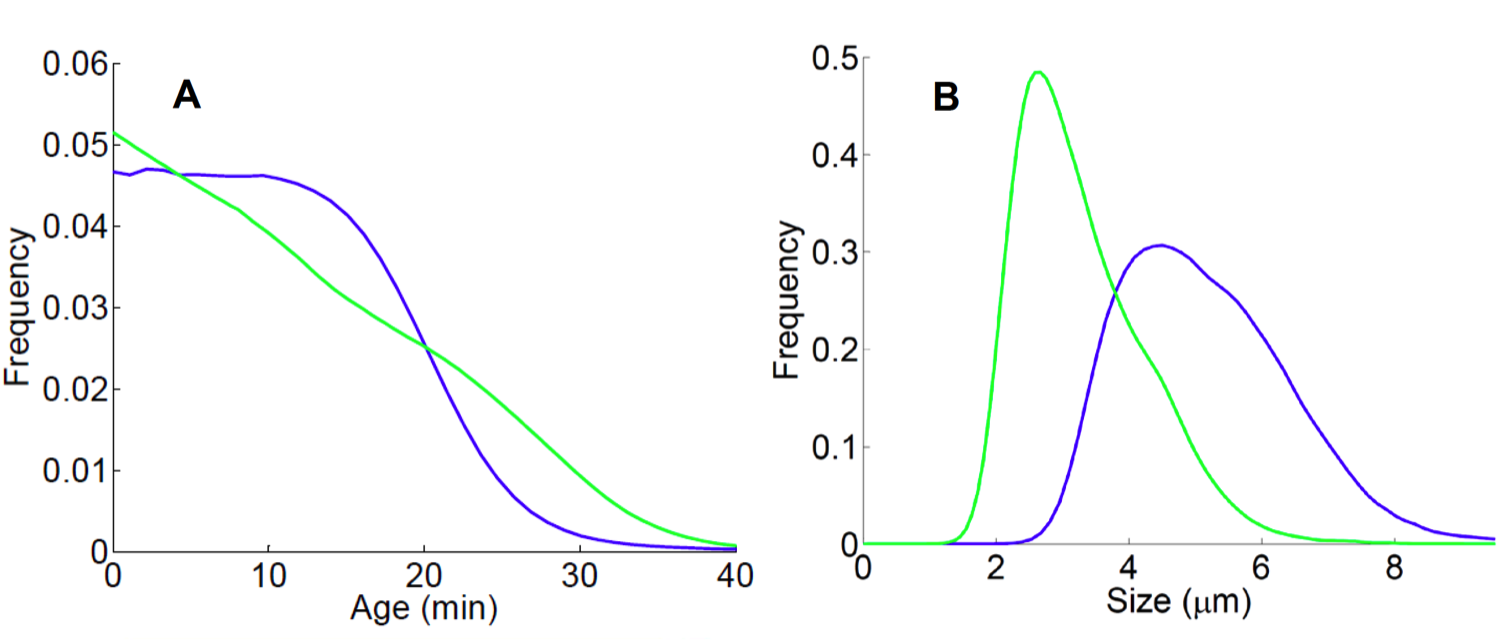}
\caption{ {\bf \emph{ All cells distributions:}} \emph{Kernel density estimation of age (Left) and length (Right) distributions obtained from sample images of genealogical observation (blue curves, data from~\protect\cite{Wang}) and population observation (green curves, data from~\protect\cite{Eric}), data taken at all time points. Figure taken from~\protect\cite{robert:hal-00981312}, Fig.1.\label{fig:distrib:all}} }
\end{center}
\end{figure}

\subsubsection*{Data analysis: Individual dynamics data analysis}

In the case of individual dynamics data collection, it is not only possible to extract size distribution but also much more. Let us list some of the information we can extract from such rich data.

To begin with, we can measure the cell age, {\it i.e.} the time elapsed since birth; or yet - let us mention it due to its recent importance in the field~\cite{ArielAmir,Jun_2015,si2019mechanistic} - size increment, {\it i.e.} the difference between the  size of the bacteria at the time considered and their size at birth. A similar process as seen above for size leads us to age or size increment distributions as illustrated in Fig.~\ref{fig:distrib:all} Left.
In the same vein, we can also select only dividing cells, measure their size, age, size increment at division, and estimate these distributions: this is illustrated in Fig.~\ref{fig:distrib:div}. Finally, we can measure joint distributions, such as age-size or size-increment of size distributions: this is done in Fig.~\ref{fig:distrib:2D}.

\begin{figure}
\begin{center}
\includegraphics[width=0.45\textwidth]{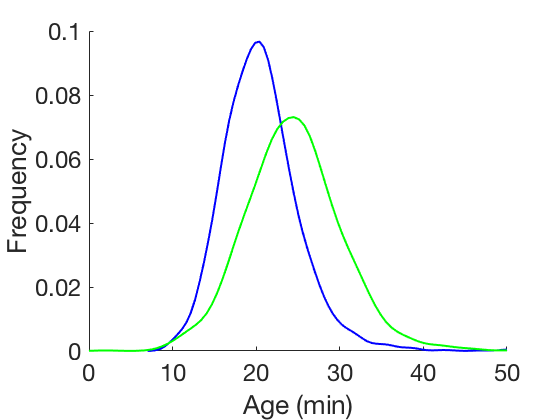}
\includegraphics[width=0.45\textwidth]{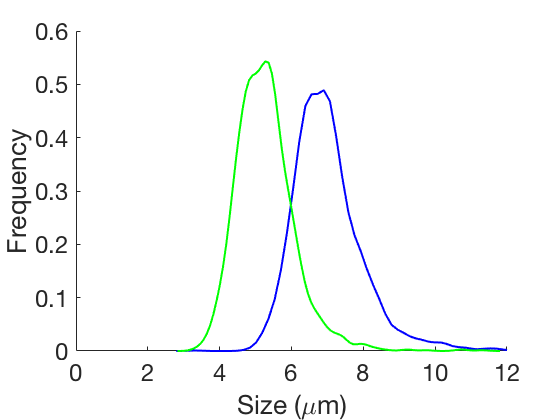}
\caption{\label{fig:distrib:div} {\it Dividing cells distributions: Kernel density estimation of age (Left) and length (Right) distributions of dividing cells obtained from sample images of genealogical observation (blue curves, data from~\protect\cite{Wang}) and population observation (green curves, data from~\protect\cite{Eric})}.}
\end{center}
\end{figure}

\begin{figure}
\begin{center}
\includegraphics[width=0.45\textwidth]{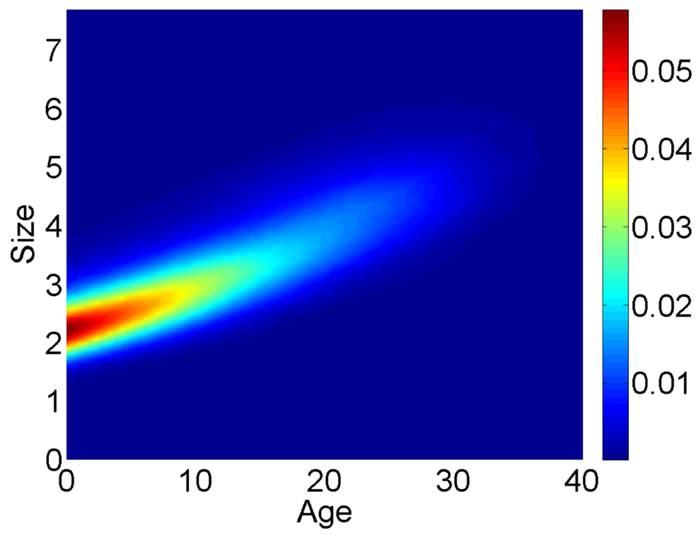}
\includegraphics[width=0.48\textwidth]{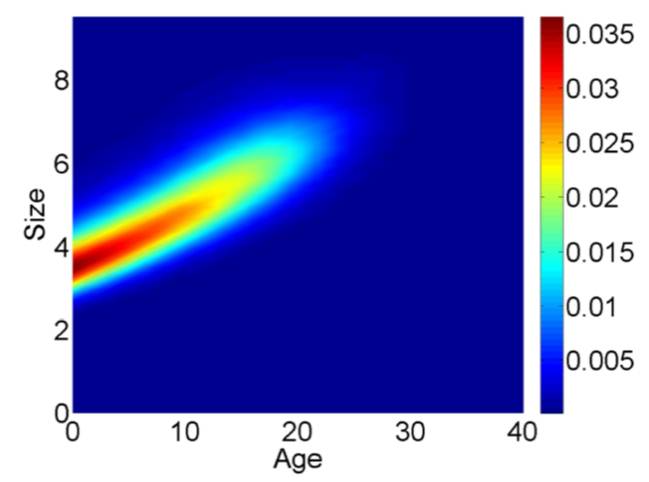}
\caption{\label{fig:distrib:2D} {\it Age-Size distributions for all cells. Left: population observation, data from~\protect\cite{Eric}, Right: genealogical observation, data from~\protect\cite{Wang}.}}
\end{center}
\end{figure}

Being able to follow each cell means that we are also able to estimate individual growth rates. This is illustrated for one given cell in Fig.~\ref{fig:growthrate}: measuring cell length and cell width every two minutes allows one to estimate its growth rate by curve fitting tools. On the right, we see width measurements: it appears to remain constant up to measurement noise. On the left, we see the fit of the data with an exponential curve in red and with a linear curve in black: as studied in~\cite{robert:hal-00981312}, and in accordance with previous studies, the exponential growth model, where we assume that the cell length evolves according to the law
$$\f{dx}{dt}=\kappa x$$ for a certain rate $\kappa >0$, fits the data very well - at least for the lengths where we have data, {\it i.e.} here for a range of sizes typical size between $0.3$ to $5\mu m$.

\begin{figure}
\begin{center}
\includegraphics[width=\textwidth]{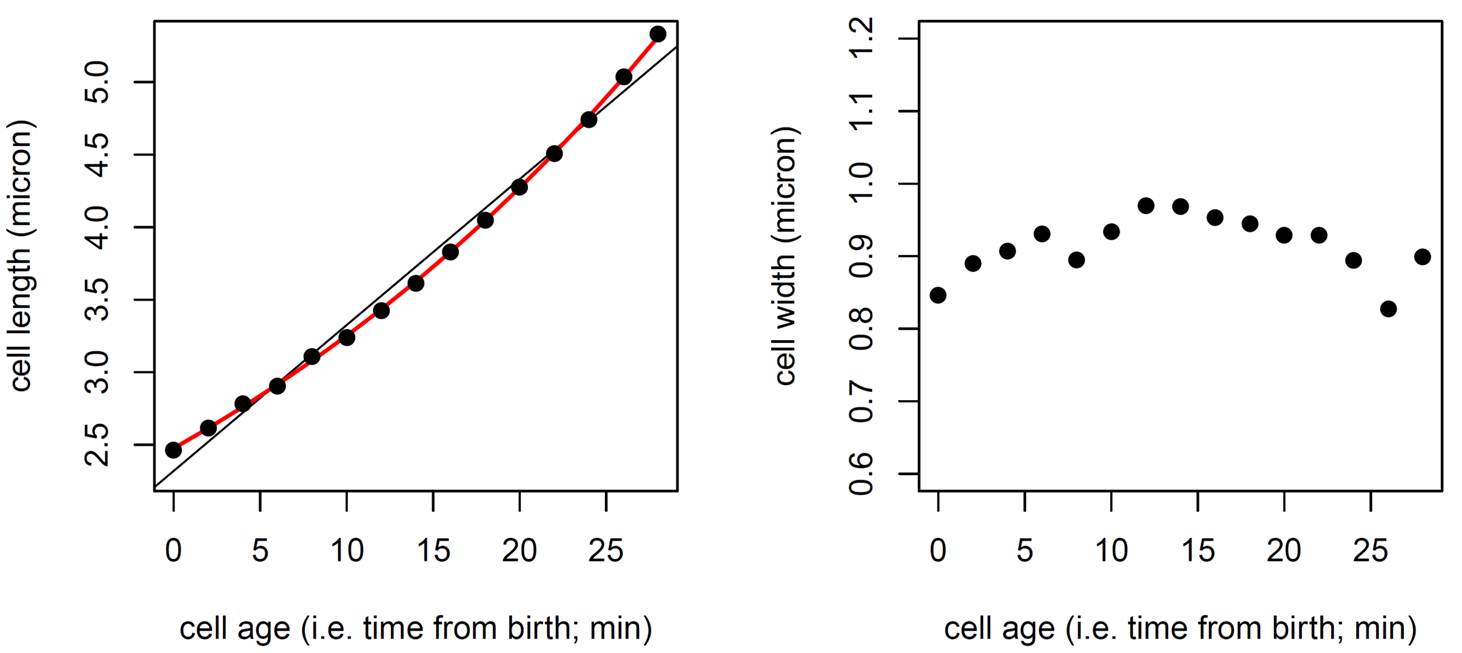}
\caption{\label{fig:growthrate} {\it Single-cell growth rate analysis. Figure taken from~\protect\cite{robert:hal-00981312} Figure~2.  For a given cell, we measure its size dimensions - length on the right, width on the left - through time, and conclude to a good agreement of the exponential growth model $\f{dx}{dt}=\kappa x$ for the length, and a constant width.}}
\end{center}
\end{figure}

\

But is this growth rate constant among all cells? As for the age or size distributions, once estimated for each cell, and assuming them independent (no heritability) and time-independent (no slowing down in growth due to lack of nutrient for instance), it is possible to study the distribution of growth rates. This is illustrated in Fig.~\ref{fig:distrib:growthrate}, Left.

Concerning division, we already mentioned that it is possible to extract dividing cells distribution. But we have more information: for instance, it is possible to measure the ratio between daughter and mother cells. 
In our application setting, this is called the septum position, the {\it septum} being the boundary formed between the two dividing cells.

\begin{figure}
\begin{center}
\includegraphics[width=0.45\textwidth]{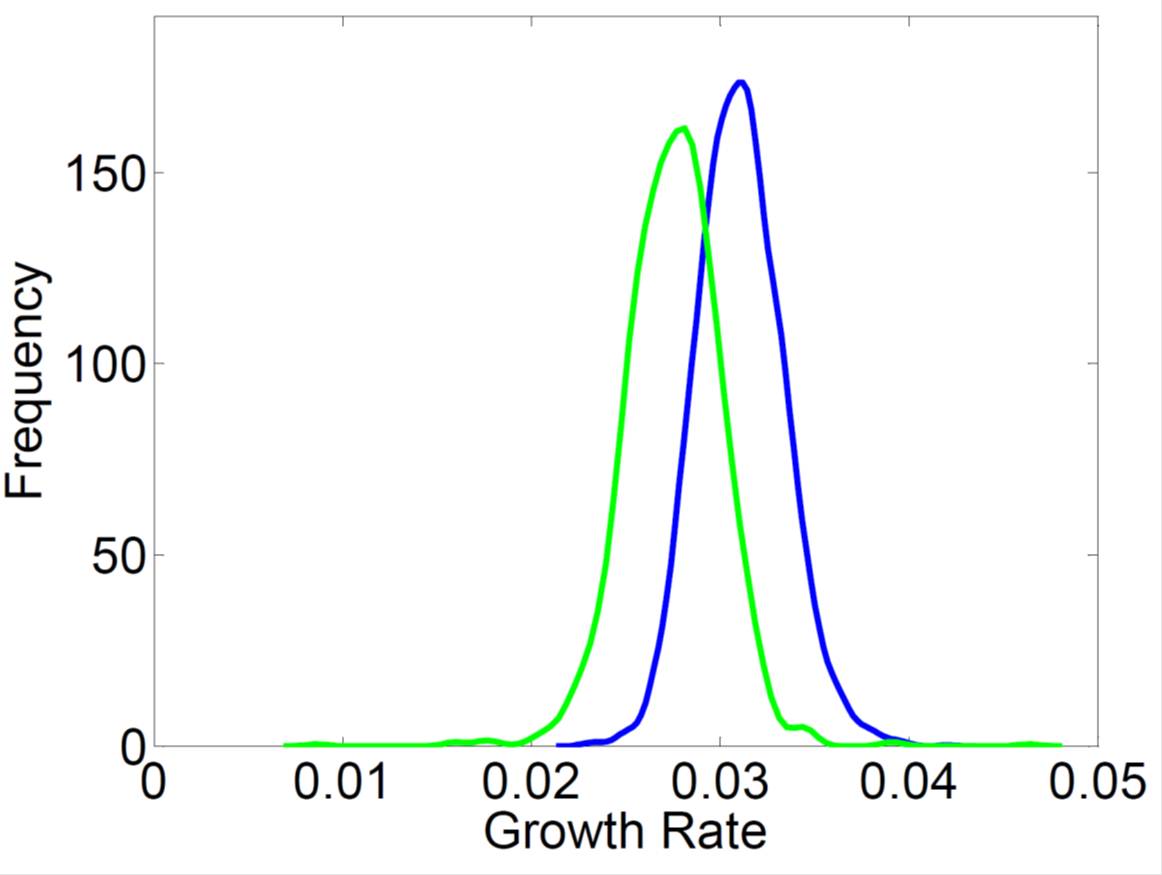}
\includegraphics[width=0.43\textwidth]{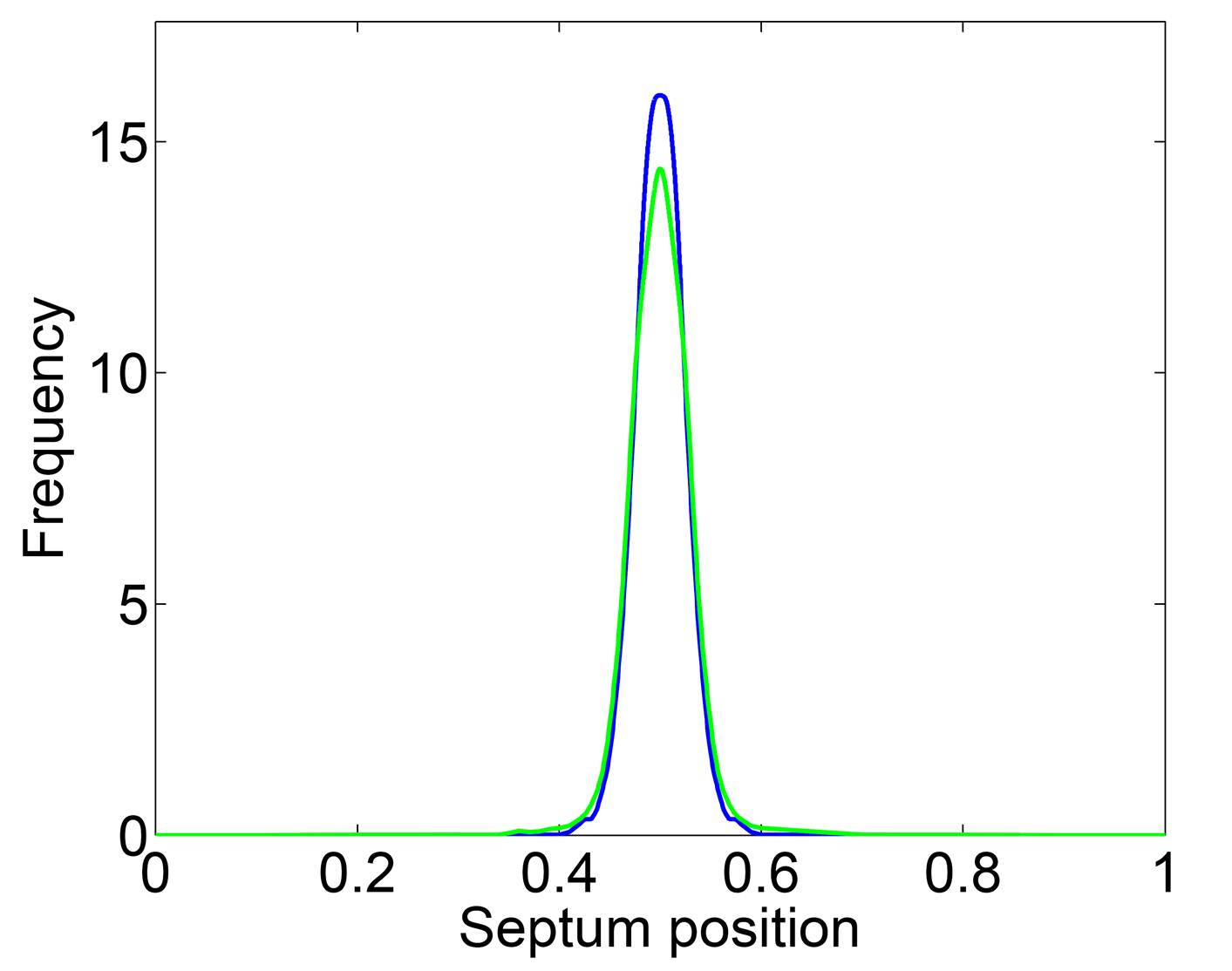}
\caption{{\it Growth rate in min$^{-1}$ (Left) and daughter/mother ratio (Right) distributions. Blue curves: genealogical data from~\protect\cite{Wang}, Green curves: population data from~\protect\cite{Eric}. Figure taken from Fig. S4 of~\protect\cite{robert:hal-00981312} \label{fig:distrib:growthrate}}.}
\end{center}
\end{figure}



\subsection{Second step: making assumptions}
\label{subsec:ass}
In the first step we have started to manipulate the data. It is clear that we could still get a lot of information from them: inheritance between mother and daughter cells, distributions over time and not just aggregating all the time data, etc.  Here and there, we already made two types of assumptions: model assumptions and simplifying assumptions.  Let us gather them here, so that we will be ready to design mathematical models.

Simplifying assumptions can be made out of the direct observations done during the first step. In our application case, we list the following assumptions which will be used as a departure point for the calibration step, see Section~\ref{sec:inverse}.
\begin{itemize}
\item The daughter cell size at birth is half its mother cell size (Fig.~\ref{fig:distrib:growthrate} Right shows very little variability: we may neglect it first).
\item All cells grow exponentially with the same growth rate $\kappa$ (Fig.~\ref{fig:distrib:growthrate} Left shows some variability, that can be neglected in a first approximation).
\item Space and nutrient consumption are infinite and do not influence growth and division (due to the experimental setting, this assumption may be verified by statistical analysis of time dynamics).
\end{itemize}
Model assumptions of course  depend on the underlying application. They are the gateway to the third step and a way to formulate the vague question of the introduction: how to determine laws for growth and division? In our case, a central assumption, linked to the Markov property of our model, is the absence of memory between mother and daughter cells, {\it i.e.} we assume no heritability of the growth rate - see~\cite{delyon2018investigation} for a thorough study of this question - and no heritability of the division features.
Concerning growth, in the case of individual dynamics data collection, we have seen that we can  formulate laws directly inferred from the data. This is not true concerning the law of division: the question "what triggers bacterial division?" remains unanswered by our data analysis step.  We thus formulate the following modelling assumptions. They provide guidelines for all our subsequent  models:

\begin{itemize}
 \item  A particle of age $a$ and size $x$ may divide with a division rate $B$ depending on its age,
\item  a particle of age $a$ and size $x$ may divide with a division rate $B$ depending on its  size,
\item a particle of size $x$ and size at birth $x_b$ may divide with a division rate $B$ depending on its increment of size $x-x_b$,
\item   a particle of size $x$ may divide with a division rate $B$ depending on an auxiliary (and latent, {\it i.e.} unobserved) variable.
\end{itemize}

\section{Building models}
\label{sec:model}

The preliminary steps sketched in Sections~\ref{subsec:datanal} and~\ref{subsec:ass}  allow one to have a clear idea on the necessary ingredients to translate mathematically the biological mechanisms  to study. We list three apparently different approaches:

\renewcommand{\labelenumi}{\arabic{enumi}.}
\begin{enumerate}
\item
{\bf Continuous-time branching processes:} the most direct and intuitive way is to model each cell of the population inside its genealogical tree, linking the parent to its offspring by a tree branch representing its lifetime, each node representing a cell taken at birth or at division. We explain this model below in Section~\ref{sec:model:branching}.

\item {\bf Stochastic differential equations (SDE) via Poisson random measures:} this formalism is strictly equivalent to the building of the branching tree and consists in writing a stochastic differential equation satisfied by a random measure representing all the cells alive at a certain time. The advantage of writing this equation is that it is a very convenient way to link the stochastic model to the "deterministic" - or, more adequately, average - approach described in Section~\ref{sec:model:PDE}. This limit is rigorously proved in Section~\ref{sec:model:random} in the pedagogical case of the renewal process, and we review results of the literature for other models.

\item {\bf Integro-partial differential equations (PDE)}: looking at a {\it large} population, or yet at the {\it average} behaviour of one or of a small number of individuals, we can write a balance equation satisfied by the concentration distribution of cells at time $t$, with given characteristics such as age, size, increment of size, etc: this gives rise to what is called a {\it structured population equation}, the term {\it structured} referring to this characteristic {\it trait} triggering growth, division, or more generally speaking evolution/change. This type of models is reviewed in Section~\ref{sec:model:PDE}.
\end{enumerate}

Depending on the context or on the specific questions to be solved, one of the three above points of view may seem more advantageous, either from a technical or an interpretation point of view. The three approaches are closely related and sometimes equivalent. We next describe somehow their minimal mathematical features and their links.

\subsection{Continuous-time branching processes}
\label{sec:model:branching}

Continuous-time branching processes are classical objects, well documented in numerous textbooks and papers, see {\it e.g.} \cite{lamperti1967continuous,kendall1966branching,harris1948branching,sevast1957limit} or \cite{haccou2005branching,meleard2009trait,champagnat2006unifying,champagnat2008individual}. Ref.~\cite{bansaye2011limit} is recommended for an efficient presentation of the topic.


The models we present here belong to the wide family of so-called "continuous-time branching trees", whose history dates back to 1873~\cite{kendall1966branching} and knows continuous interest from ecologists as well as mathematicians, see {\it e.g.} ~\cite{haccou2005branching} or for a specific and very recent example~\cite{bansaye2019scaling}. For the sake of simplicity, we stick here to the modelling and simplifying assumptions sketched above, so that our branching process is encoded in a binary tree: each node splits into exactly two branches. Using the Ulam-Neveu notation, we define
\begin{equation}\label{def:U}{\mathcal U} := \bigcup_{n=0}^\infty \{0,1\}^n\;\;\;\text{with}\;\;\{0,1\}^0:=\emptyset.
\end{equation}
\begin{figure}
\begin{center}
\includegraphics[width=0.45\textwidth]{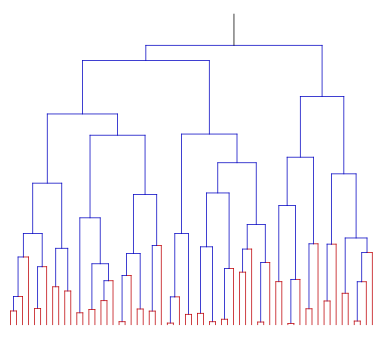}
\includegraphics[width=0.43\textwidth]{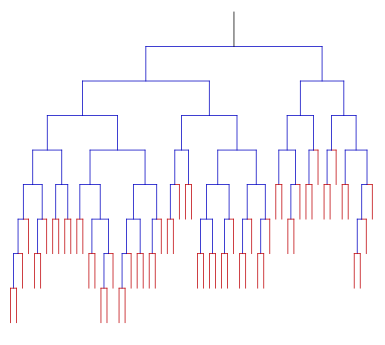}
\caption{{\it  Simulation of a binary tree with.  Left: the size of each segment represents the lifetime of an individual.  Individuals alive at  time $t$ are  represented  in  red.  Right:  genealogical  representation  of  the  same realisation of the tree. Figure taken from Fig. 1 of~\protect\cite{hoffmann2016nonparametric} \label{fig:UlamNeveu}.}}
\end{center}
\end{figure}
To each node $u \in {\mathcal U}$, we associate a cell with birth time $b_u$, size at birth  $\xi_u$, lifetime $\zeta_u$ and increment of size since birth $\eta_u$. Assuming that the size growth is given by the growth rate $\tau(x)$ (in the sequel, we will often specify  $\tau(x)=\kappa x$ for some common growth rate $\kappa >0$),
if $u^- \in \mathcal U$ denotes the parent of $u \in \mathcal U$, then the size at division is given by $\chi_u=X(\zeta_u,\xi_u)$ where $X(t,x)$ denotes the characteristic curve solution to the differential equation
$$\f{dX (t,x)}{dt}=\tau\left(X(t,x)\right), \qquad X(0,x)=x,$$
and we have, for a division probability kernel $b(\chi_u,dx)$
$$\xi_u \sim b(\chi_{u^-},dx), \qquad  \eta_{u}=\chi_u-\xi_u=X(\zeta_u,\xi_{u}) -\xi_u.$$
In the specific case of exponential growth and diagonal kernel (equal mitosis, for which $b(\chi_u,dx)=\delta_{\f{\chi_u}{2}}(dx)$), this gives
$$\xi_u =\f{\chi_{u^-}}{2}=\frac{ \xi_{u^-}}{2}\exp\big(\kappa \zeta_{u^-}\big),\qquad \eta_{u}=\xi_{u} \left(\exp(\kappa \zeta_u) -1\right).$$
We see that within such a formalism, it is straightforward to generalise our assumptions: for instance, the growth rate $\kappa$ could be selected randomly at birth, according to some probability law which could depend on the parent growth rate $\kappa_{u^-}$ (or on some other trait of the mother). Another way to model unequal division is to write
$$\xi_{(u^-,0)} =\alpha \xi_{u^-}\exp(\kappa \zeta_{u^-}),\qquad
\xi_{(u^-,1)}=(1-\alpha)\xi_{u^-}\exp(\kappa \zeta_{u^-}),$$
with $\alpha\in(0,1)$ chosen randomly according to some probability distribution $b_0 (d\alpha)$, symetric in $\f{1}{2}$. In such a case, the probability kernel $b$ has a specific form, which is sometimes called {\it self-similar}, namely
$$b(y,x)=\f{1}{y}b_0(\f{x}{y}).$$

\

One last ingredient is still missing to have fully determined the tree: the distribution of the random time at which division occurs. Let us list the three examples discussed in the assumptions section above:

\begin{enumerate}
\item {\bf The division depends on age:} this translates into the fact that given a division rate function
$$B:(0,\infty)\to [0,\infty),\;\;\int^\infty B(s)ds = \infty,$$
the lifetime $\zeta_u$ is a random variable with distribution
$$\PP(\zeta_u \in da) = B(a)e^{-\int_0^a B(s)ds}da.$$ Other said, we have
\begin{align*}
 \PP(\zeta_u \in (a,a+da) \vert \zeta_u\geq a)=B(a)da, \,
 \PP(\zeta_u \geq a) = \exp\Big(-\int_0^a B(s)ds\Big).
\end{align*}
With this construction, the renewal process is simply embedded into a branching tree representation. See \cite{hoffmann2016nonparametric} for a study of a slight generalisation of this model. We notice  that we can add other variables as the size and the increment of size, but they will have no influence on the tree.
\item {\bf The division depends on size:} when the division is triggered by a size-dependent rate, and when the size grows at a rate $\tau(x)$, the division rate function now translates into the following properties
$$
\begin{array}{ll}
\PP(\chi_u \in (x,x+dx) \vert \chi_u\geq x)=B(x)dx=\tau(x)B(x) dt,
\\ \\
 \PP(\chi_u\geq x \vert \xi_{u})=\1_{\{x\geq \xi_{u}\}} \exp\left(-\int_{\xi_{u}}^x B(y)dy\right).
 \end{array}
$$
Notice that the age variable $\zeta_u$, although well-defined, is a bit irrelevant in this context: this is due to the fact that  we  define  an instantaneous rate $B(x)dx$ instead of defining a rate $B(x)dt.$ Compared to previous mathematical papers~\cite{BP,MMP2,DG,DHKR} etc., we simply substitute  $B$ by $B\tau.$
\item {\bf The division depends on the increment of size since birth:} as for the renewal process, the increment is reset to zero at each division, however if the growth rate $\tau(x)$ is not constant it does not increase linearly with time, so that $\eta_u$ is now characterised by
$$
\begin{array}{ll}
\PP(\eta_u \in (\i,\i+d\i) \vert \eta_u\geq \i,\xi_u)=B(\i)d\i=B(\i)\tau(\xi_u +\i) dt,\\ \\
 \PP(\eta_u\geq \i )=\exp\left(-\int_0^{\i} B(y)dy\right).
 \end{array}
 $$
Indeed, for a size increment $\i$ we have a size since birth equal to $x=\xi_u + \i$, and the size grows  according to
$\f{dx}{dt} = d(\xi_u+\i)= d\i=\tau (x)dt$. When we embed the model in the time-continuous framework, we see that the lifetime does depend on size, contrarily to the renewal process, due to the fact that $d\i\neq dt$ formally.
\end{enumerate}

Given observational data, there are basically two ways of considering the tree: either we look at the process of structured  cells {\it until a certain generation} $n$ -  along one branch chosen at random at each node (genealogical observation), or along chosen branches, or yet along all of them; alternatively,  we look at the process {\it until a certain physical time} (population observation). In the first point of view, the physical time is not intrinsically important, contrarily to the second point of view, that we next adopt to describe the process via random measures.

\subsection{Random measures} \label{sec:model:random}

We consider the random process
$$X(t) = \big(X_1(t),X_2(t),\ldots\big)$$
that describes the (ordered) sizes of the population at time $t$, or
$$A(t) = \big(A_1(t),A_2(t),\ldots\big)$$
the (ordered) ages of the population at time $t$. Equivalently, we can define the random processes with values in finite-point measures on $(0,\infty)$ via
\begin{equation}
{ Z}_t^{(a)} = \sum_{i=1}^{\infty} \delta_{A_i(t)},\;\;
Z_t^{(x)} = \sum_{i=1}^{\infty} \delta_{X_i(t)}, \label{eq: random measure}
\end{equation}
(the sum is finite but the total number of particles may be different for different values of $t$), that describe the population states, characterised by the structuring variables (here size, age, size increment and so on) at any given time $t$.

\

We can then look for a characterisation of the process $(Z_t^{(a)})_{t \geq 0}$ or $(Z_t^{(x)})_{t \geq 0}$ as via a stochastic evolution equation, here a measure-valued stochastic differential equation (SDE). In order to do so, we need a technical tool, namely the use of Poisson random measures.


\subsubsection*{Preliminaries: Poisson random measures}

A convenient way to model scattered points or events through time and a state space is by means of a Poisson random measure. The notion and its properties are well documented in numerous textbooks, from stochastic geometry to stochastic calculus, see {\it e.g.} \cite{baccelli2020random,daley2003introduction,jacod2013limit} and we briefly recall the essential and basic material needed here.\\

We start with a state space $\mathcal X \subset \R^d$ and we let $\mu$ be some sigma-finite measure on $\mathcal X$ equipped with its Borel sigma-field. If $(x_i)_{i \in \mathcal I}$ is a countable family of elements of $\mathcal X$, the point measure $N$ associated with $(x_i)_{i \in \mathcal I}$ is given by
$$N(dx) = \sum_{i \in \mathcal I} \delta_{x_i}(dx),$$
and its  action on test functions $\varphi: \mathcal X \rightarrow \R$ is written as
$$\langle N,\varphi\rangle = \int_{\mathcal X}\varphi(x)N(dx) = \sum_{i \in \mathcal I}\varphi(x_i)$$
whenever the above sum is well-defined.
\begin{definition} A random point measure $N$ is a random Poisson measure on $\mathcal X$ with intensity $\mu$ if
\begin{enumerate}
\item For every Borel set $A$, the random variable $N(A)$ has a Poisson distribution with parameter $\mu(A)$:
$$\mathbb P(N(A) = k) = e^{-\mu(A)}\frac{\mu(A)^k}{k!},\;\;k=0,1,\ldots.$$
\item For every countable family of Borel sets $(A_i)_{i \in \mathcal I}$ with $A_i \cap A_j = \emptyset$, if $i\neq j$, the random variables $\big(N(A_i)\big)_{i \in \mathcal I}$ are independent.
\end{enumerate}
\end{definition}
Consider now a random measure $N$ on $\mathcal X = [0,\infty) \times [0,\infty)$ with intensity $dt \otimes dx$. A first natural process that can be easily constructed from $N$ is a so-called inhomogeneous
Poisson process with intensity function $\lambda : [0,\infty) \rightarrow [0,\infty)$. It models an event based process $(X_t)_{t \geq 0}$ that counts the events that occur between $0$ and $t$ according to the following two rules:
$$\mathbb P(\text{an event occurs in }\; (t,t+dt)\;|\;\text{given the history up to time}\; t) = \lambda(t)dt,$$
and $(X_t)_{t \geq 0}$ has independent increments: the random variables: $(X_{t_i+h_i}-X_{t_i})_{1 \leq i \leq n}$ are independent, for any family of disjoints sets $([t_i, t_i+h_i])_{1 \leq i \leq n}$.
A possible construction  is given by
$$X_t = \int_0^t \int_{[0,\infty)} \1_{\{u \leq \lambda(s)\}} N(ds,du).$$
To understand better the construction from a heuristic point of view, if $\lambda(s)$ is sufficiently smooth, for $s \in [t,t+dt]$, we can approximate $s \mapsto \lambda(s)$ by the constant function $s \mapsto \lambda(t)$ and thus obtain the chain of approximations
\begin{align*}
X_{t+dt}-X_t  & = \int_t^{t+dt} \int_0^{\lambda(s)}N(ds,du)
\\ \\
& \approx \int_t^{t+dt} \int_0^{\lambda(t)} N(ds,du)
 = N\big( [t,t+dt] \times [0,\lambda(t)]\big)
\end{align*}
which has a Poisson distribution of parameter $\lambda(t)dt$ since the intensity of $N$ is $dt \otimes ds$.  In particular, a standard Poisson process with parameter $\lambda >0$ can be represented as $X_t = \int_0^t \int_{[0,\infty)}\1_{\{u \leq \lambda\}}N(ds, du)$. This is the basic ingredient to construct more elaborate jump models.

\subsection*{The simplest age model}

We first prove the link between the deterministic and the stochastic model in the simplest age-dependent model: at time $t$, the age-structured cell population is described by the states (here the ages) of the living cells at time $t$, that we denote $(A_1(t), A_2(t),\ldots, A_n(t),\ldots)$ (given in the increasing order for instance), that we encode into the point measure
$$Z_t^{(a)} = Z_t(da) = \sum_{i = 1}^\infty \delta_{A_i(t)}(da),$$
where the sum ranges from $1$ to $\langle Z_t, \1\rangle$ and which is finite, as in \eqref{eq: random measure}. Assuming that the ages are ordered, we define the evaluation maps
$$a_i: Z_t \mapsto a_i(Z_t) = a_i\Big(\sum_{j = 1}^\infty \delta_{A_j(t)}(da)\Big) = A_i(t).$$
Abusing notation slightly, we may identify $A_i(t)$ and $a_i(Z_t)$. We have a complete description of the stochastic dynamics of $Z_t = Z^{(k)}_t$ by means of a family of independent Poisson random measures $N_i(ds,d\vartheta)$, $i=1, 2\ldots$ with intensity $ds \otimes d\vartheta$ on $[0,\infty) \times [0,\infty)$.
It is given by the stochastic evolution equations, for $k=1,2$:
\begin{equation} \label{eq: EDS age simple}
\begin{array}{ll}
Z_t^{(k)} &= \tau_t Z_0 \\ \\
&+ \int_0^t \sum_{i \leq \langle Z_{s-}^{(k)}, \1\rangle}\int_0^\infty (k\delta_{t-s}-\delta_{a_i(Z_{s-}^{(k)})+t-s})\1_{\{\vartheta \leq B(a_i(Z_{s-}^{(k)}))\}} N_{i}(ds, d\vartheta),
\end{array}
\end{equation}
where
$$Z_{0}^{(k)} = \sum_{i = 1}^{\langle Z_{0}^{(k)},\1\rangle}\delta_{A_{i}(0)}$$
is an initial age distribution and
$$\tau_t Z_{0}^{(k)} = \sum_{i = 1}^{\langle Z_{0}^{(k)},\1\rangle}\delta_{A_{i}(0)+t}.$$
We have $k=1$ for the genealogical observation case, where only one daughter cell is kept at each division, and $k=2$ for the population observation case.\\

We may interpret \eqref{eq: EDS age simple} as follows: at time $t$ the term $\tau_t Z_{0}^{(k)}$ accounts for the initial population that has aged by the amount $t$. This term must be corrected by:
\begin{itemize}
\item adding the ages of newborn cells. This is done as follows: a division event is run according to the rate $B(a_{i}(Z_{s-}^{(k)}))$, where $i$ ranges over the whole population at time $s$, (and then $s$ ranges from $0$ to $t$). When a division occurs, produced by the cell $i$ with age $a_{i}(Z_{s-}^{(k)})$, $k$ newborn cells are created with age $0$ at time $s$, that will have the same age $t-s$ at time $t$. We thus add to the system the term $k\delta_{t-s}$;
\item  removing the ages of the mother cells. This comes from the dying mother cell at time $s$ and age $a_{i}(Z_{s-}^{(k)})$, according to the same division event as for the addition of newborn cells. This cell would have age $a_{i}(Z_{s-}^{(k)})+t-s$ at time $t$, and therefore, we remove to the system the term $\delta_{a_{i}(Z_{s-}^{(k)})+t-s}$.
\end{itemize}
Whenever the integrator is a jump measure, we must be careful to consider integrands that are predictable in the sense of stochastic calculus, hence the left limits in terms involving $a_{i}(Z_{s-})$. This is innocuous at the informal level we keep here, but becomes crucial whenever martingales properties are involved, that are essential to properly define existence and uniqueness of a measure solution $(Z_{t}^{(k)})_{t \geq 0}$ to \eqref{eq: EDS age simple}.

Note that it is not completely obvious that a solution $(Z_{t}^{(k)})_{t \geq 0}$ to \eqref{eq: EDS age simple} exists, as well as its uniqueness. This is obtained by classical arguments and we refer to \cite{Tran2008,fournier2004microscopic} for a comprehensive treatment of the subject.


\subsubsection*{From the stochastic evolution \eqref{eq: EDS age simple} to a deterministic PDE}

Consider the following renewal equation:
\begin{equation} \label{eq: age edp}
\left\{\begin{array}{l}
     \f{\p}{\p t} n_{k} (t,a) + \f{\p}{\p a} n_{k}(t,a)
        +B(a) n_{k}(t,a)=0,\\  \\
    n_{k}(t,0)=k\int_0^\infty B(a)n_{k}(t,a)da, \\ \\
        n_{k}(0,a)=n^{in}(a), \qquad \int_0^\infty n^{in} (a)da=1,
    \end{array}\right.
\end{equation}
This is a classical model that goes back to McKendricks and von Foerster, see in particular the textbooks \cite{MetzDiekmann,BP}. If $n^{in}$ and $B$ are well-behaved  there is existence and weak uniqueness of \eqref{eq: age edp}, see Prop.~\ref{prop:exist:age}. Moreover, the solution $n_k(t,da) = n_k(t,a)da$ is absolutely continuous. Define for $t \geq 0$ the deterministic family of positive measures $(m_k(t))_{t \geq 0}$ via
$$\langle m_k(t,\cdot), \varphi \rangle = \int_0^\infty \varphi(a)m_k(t,da) = \mathbb E\big[\langle Z_t^{(k)},\varphi \rangle\big],$$
with $m_k(0,da) = n^{in}(a)da$.  The interpretation of $m$ is the mean, or macroscopic state of the system. The link between the stochastic evolution \eqref{eq: EDS age simple} and the deterministic  Fokker-Planck equation \eqref{eq: age edp} is given by the following simple equivalence:
\begin{proposition} \label{prop: first equivalence}
Assume that $B$ is continuous and bounded and that the initial condition $n^{in}$ is absolutely continuous, with a bounded continuous density. We then have $m_k=n_k$.
\end{proposition}
These conditions are not minimal. We detail the essential steps of the proof of Proposition \ref{prop: first equivalence}, since it illustrates in a relatively simple framework the interplay between deterministic and stochastic methods. The other equivalence results between random measure evolution equations and structured population equations stated in Section \ref{sec: structured pop eq} are obtained in the same way.\\

\begin{proof}
\noindent {\it Step 1) The action of time and age test functions \eqref{eq: EDS age simple}}.
Let $\varphi_t(a) = \varphi(t,a)$ denote a test function (assumed to be smooth and compactly supported for safety). First, by definition of \eqref{eq: EDS age simple},
\begin{align}
\langle Z_t^{(k)},\varphi_t\rangle & = \langle \tau_t Z_0^{(k)},\varphi_t\rangle \nonumber 
+ \int_0^t \sum_{i \leq \langle Z_{s-}^{(k)}, \1\rangle}\int_0^\infty \\ &\big(k\varphi_t(t-s)-\varphi_t(a_i(Z_{s-}^{(k)})+t-s)\big)\1_{\{\vartheta \leq B(a_i(Z_{s-}^{(k)}))\}} N_i(ds,d\vartheta).  \label{eq: test}
\end{align}
Next, for every $0 \leq s \leq t$, we have
$$\varphi_t(a+t-s) = \varphi_s(a)+\int_s^t  \big(\tfrac{\p}{\p u}\varphi_u(a+s-u) + \tfrac{\p}{\p a}\varphi_u(a+s-u)\big)du.$$
Therefore, successively
\begin{align*}
& \langle \tau_t Z_0^{(k)},\varphi_t\rangle  = \sum_{i = 1}^{\langle Z_0^{(k)},\1\rangle} \big(\varphi_0(a_i(Z_0^{(k)}))
\\
& + \int_0^t  \big(\tfrac{\p}{\p s}\varphi_s(a_i(Z_0^{()k})+s) + \tfrac{\p}{\p a}\varphi_s(a_i(Z_0^{(k)})+s)\big)ds\big), \\
& \varphi_t(t-s)  = \varphi_s(0)+\int_s^t  \big(\tfrac{\p}{\p u}\varphi_u(u-s) + \tfrac{\p}{\p a}\varphi_u(u-s)\big)du, \\
 &\varphi_t(a_i(Z_{s-}^{(k)})+t-s)  = \varphi_s(a_i(Z_{s-}^{(k)}))+\\
 & \hspace{3cm}\int_s^t  \big(\tfrac{\p}{\p u}\varphi_u(a_i(Z_{s-}^{(k)})+u-s)  \tfrac{\p}{\p a}\varphi_u(a_i(Z_{s-}^{(k)})+u-s)\big)du.
\end{align*}
Plugging-in these three expansions in \eqref{eq: test}, and using the definition of $Z_s^{(k)}$ again, we obtain
\begin{align*}
\langle Z_t^{(k)},\varphi_t\rangle  & = \langle Z_0^{(k)},\varphi_0\rangle 
+ \int_0^t \big\langle \tfrac{\p}{\p s}\varphi_s + \tfrac{\p}{\p a}\varphi_s, Z_s^{(k)}\big\rangle ds
\\
&+\int_0^t \sum_{i \leq \langle Z_{s-}^{(k)}, \1\rangle}\int_0^\infty\big(k\varphi_s(0)+ \varphi_s(a_i(Z_{s-}^{(k)}))\big)\1_{\{\vartheta \leq B(a_i(Z_{s-}^{(k)}))\}} N_i(ds,d\vartheta) .
\end{align*}

\noindent {\it Step 2) A martingale-oriented representation of \eqref{eq: EDS age simple}}. We compensate the Poisson random measures $N_i()ds,d\vartheta$ and define
$$\widetilde N_i(ds,d\vartheta) = N(ds,d\vartheta)-ds \otimes d\vartheta.$$
For fixed $\varphi$, by construction, the random process
$$M_t(\varphi) = \int_0^t \sum_{i \leq \langle Z_{s-}^{(k)}, \1\rangle}\int_0^\infty\big(k\varphi_s(0)+ \varphi_s(a_i(Z_{s-}^{(k)}))\big)\1_{\{\vartheta \leq B(a_i(Z_{s-}^{(k)}))\}} \widetilde N_i(ds,d\vartheta)$$
is a centred martingale, and in particular $\mathbb E[M_t(\varphi)] = 0$. Noticing that
\begin{align*}\sum_{i \leq \langle Z_{s-}^{(k)}, \1\rangle}\int_0^\infty\big(k\varphi_s(0)+ \varphi_s(a_i(Z_{s-}^{(k)}))\big)\1_{\{\vartheta \leq B(a_i(Z_{s-}^{(k)}))\}}ds d\vartheta 
\\ = \big\langle  Z_s^{(k)}, (k\varphi_s(0)-\varphi_s)B\big\rangle,
\end{align*}
we obtain from Step 1 the representation
$$\langle Z_t^{(k)},\varphi_t\rangle =  \langle Z_0^{(k)},\varphi_0\rangle +\int_0^t \big\langle \tfrac{\p}{\p s}\varphi_s + \tfrac{\p}{\p a}\varphi_s + (k\varphi_s(0)-\varphi_s)B, Z_s^{(k)}\big\rangle ds + M_t(\varphi).$$
Taking expectation and applying Fubini's theorem, we derive
\begin{equation} \label{eq: rep after mg}
\langle m_k(t,\cdot),\varphi_t\rangle =  \langle m_k(0,\cdot),\varphi_0\rangle +\int_0^t \big\langle \tfrac{\p}{\p s}\varphi_s + \tfrac{\p}{\p a}\varphi_s + (k\varphi_s(0)-\varphi_s)B, m_k(s,\cdot)\big\rangle ds.
\end{equation}
\noindent {\it Step 3: From \eqref{eq: rep after mg} to \eqref{eq: age edp}}. We finally prove that \eqref{eq: rep after mg} and \eqref{eq: age edp} are equivalent, along a classical line of arguments: we take for granted that $m_k(t,da) = m_k(t,a)da$ is absolutely continuous. This is a consequence of the smoothness of the initial condition and of $B$, see Prop.~\ref{prop:exist:age}. We refer to \cite{BP} for transport equations or to \cite{Tran2008} for a probabilist point of view. We evaluate its time derivative against a test function $\phi$ that depends on age only using \eqref{eq: rep after mg}. We obtain, for a nice enough $\phi$ so that we can interchange integral and derivative:
$$\tfrac{d}{dt}\int_0^\infty \phi(a)m_k(t,a)da = \int_0^\infty \big(\tfrac{\p}{\p a}\phi(a)+(k\phi(0)-\phi(a))B(a)\big)m_k(t,a)da.$$
Assume now that $\phi$ is compactly supported and smooth, and that moreover $\phi(0)=0$. Integrating by part
$$\int_0^\infty  \big(\tfrac{\p}{\p a}\phi(a)m_k(t,a)da
=  -\int_0^\infty \phi(a)\tfrac{\p}{\p a}m_k(t,a)da.$$
We infer that for every such $\phi$:
$\tfrac{d}{dt}\langle m_k(t,\cdot),\phi\rangle = - \int_0^\infty \big(\f{\p}{\p a}m_k(t,a)+B(a)\big)m_k(t,a)\big)\phi(a)da$ and therefore
$$\tfrac{\p}{\p t}m_k(t,a)+\tfrac{\p}{\p a}m_k(t,a)+B(a)m(t,a) = 0\;\;da-{\text{almost everywhere}}$$
using one more time that $\phi(0)=0$ to eliminate the term involving $k\phi(0)$.
Now, consider an arbitrary test function $\phi$ vanishing at infinity: from the previous computation, we have
\begin{align*}0 &= \int_0^\infty  \big(\tfrac{\p}{\p t}m_k(t,a)+\f{\p}{\p a}m_k(t,a)\big)\phi(a)da 
\\
&= -\phi(0)m_k(t,0)+ k\phi(0)\int_0^\infty m_k(t,a)B(a)da
\end{align*}
and simplifying, we obtain the boundary condition $m_k(t,0) = k\int_0^\infty B(a)m_k(t,a)$, as expected. This completes the proof.
\end{proof}

\subsection{Structured population equations} \label{sec: structured pop eq}
\label{sec:model:PDE}

In the previous subsection, we have seen a rigorous derivation of the renewal equation from the stochastic differential equation. The same proof can be done in more intricate situations. For specific examples, we now give below the stochastic equation followed by the corresponding PDE, and the references for their rigorous derivation as well as some examples of possible generalisations.

\subsubsection*{Size-structured model: the growth-fragmentation process \& equation}
We structure the cell population according to the individual sizes $X_i(t)$ of each individual present at time $t$, encoded into the random point measure
$$Z_t^{(k)}(dx) = \sum_{i = 1}^{\infty}\delta_{X_i(t)}(dx),$$
with values in $(0,\infty)$, the state space of sizes, and $k=1$ or $2$ refers to the number of cells that are kept into the system at division. The sum ranges from $1$ to $\langle Z_t^{(k)}, \1\rangle $. Assuming that the sizes are ordered, the evaluation maps are defined as
$$x_i(Z_t^{(k)}) = x_i\Big(\sum_{j = 1}^\infty  \delta_{X_j(t)}(dx)\Big) = X_i(t).$$
In analogy to the age model described in the previous section, we have a stochastic evolution equation for the process $(Z_t^{(k)})_{t \geq 0}$:
\begin{align} \label{eq: EDS size simple}
&Z_t^{(k)} = \exp(\kappa t) Z_0 + \int_0^t \sum_{i \leq \langle Z_{s-}^{(k)}, \1\rangle}\int_0^\infty  \nonumber \\
&(k\delta_{\tfrac{1}{2}x_i(Z_{s-}^{(k)})\exp(\kappa (t-s))}-\delta_{x_i(Z_{s-}^{(k)})\exp(\kappa (t-s))})\1_{\{\vartheta \leq B(x_i(Z_{s-}^{(k)}))\}} N_i(ds,d\vartheta),
\end{align}
where, abusing notation slightly, for a random point measure $Z(dx) = \sum_{i = 1}^\infty\delta_{z_i}(dx)$, and a real valued function $\phi$, we set
$\phi(Z)(dx) = \sum_{i  = 1}^{\infty} \delta_{\phi(z_i)}(dx)$, so that $\exp(\kappa t) Z_0=\sum \delta_{X_i(0)e^{\kappa t}}$.
Set, for (regular compactly supported) $\varphi$
$$\langle n_k(t,\cdot), \varphi \rangle = \E\big[\langle Z_t,\varphi\rangle\big],$$
We have (in a weak sense)
\begin{equation}\label{eq: EDP size simple}
\frac{\partial}{\partial t} n_k(t,x) + \frac{\partial}{\partial x}\big(\kappa x\, n_k(t,x)\big) + B(x)n_k(t,x) = 2kB(2x)n_k(t,2x).
\end{equation}

The fact that the measure $\E\big[\langle Z_t,\cdot\rangle\big]$ solves \eqref{eq: EDP size simple} can be obtained along the same line or arguments as in Proposition \ref{prop: first equivalence}. An alternative proof related on fragmentation processes, following the tagged fragment technique developed for instance in \cite{bertoin2006random,haas2003loss}, can be found in \cite{DHKR}, for a more general model allowing variability in the growth rate $\kappa$.

\subsubsection*{Adder model: the incremental process \& structured equation}

We structure the model in the pair of traits $(x,z)$ where $x$ denotes the size of a cell and $z$ its size increment since its birth. We obtain a random measure
$$Z_t^{\iota}(dx,dz) = \sum_{i = 1}^\infty \delta_{(X_i(t), Z_i(t))}(dx,dz),$$
where $(X_1(t), X_2(t) \ldots) $ denotes the (ordered) size of each cell present in the system at time $t$ (being born before $t$ or course) and $Z_i(t)$ denotes the size increment of the cell with size $X_i^{b}$ at birth. With these notation, the size $X_i(t)$ of a cell present in the system at time $t$ is simply
$$X_i(t) = X_i^{b}+Z_i(t).$$
With the evaluation mappings
$$(x_i,z_i)(Z_t^\iota) = (x_i,z_i)\Big(\sum_{j = 1}^\infty \delta_{(X_j(t), Z_j(t))}(dx,dx)\Big) = (X_i(t), Z_i(t)),$$
the stochastic evolution for the measure-valued process $(Z_t^\iota)_{t \geq 0}$ is given by
\begin{align}
Z_t^\iota & = Z_0^\iota \exp(\kappa t) \nonumber \\
&+ \int_0^t \sum_{i \leq \langle Z_{s-}^\iota, \1\rangle}\int_0^\infty \Big(k\delta_{\big(\tfrac{1}{2}x_i(Z_{s-}^\iota)\exp(\kappa(t-s)), \tfrac{1}{2}x_i(Z_{s-}^\iota)(\mathrm{e}^{\kappa(t-s)}-1)\big)}  \nonumber\\
&\hspace{0.3cm}-\delta_{\big(x_i(Z_{s-}^\iota)\mathrm{e}^{\kappa(t-s)},x_i(Z_{s-}^\iota)\mathrm{e}^{\kappa (t-s)}-(x_i(Z_{s-}^\iota)-z_i(Z_{s-}^\iota))\big)} \Big) \nonumber \\
&\hspace{0.8cm} \times 
\1_{\{\vartheta \leq \kappa x_i(Z_{s-}^\iota) B(z_i(Z_{s-}^\iota))\}} N_i(ds,d\vartheta), \label{eq: EDS adder simple}
\end{align}
where 
$$Z_0^\iota \exp(\kappa t)  = \sum_{i = 1}^\infty \delta_{\big(x_i(Z_0^\iota)\exp(\kappa t), x_i(Z_0^\iota)\exp(\kappa t)-(x_i(Z_0^\iota)-z_i(Z_0^\iota))\big)}(dx,dz),$$
and similarly, set for (regular compactly supported) $\varphi$ valued in $(0,\infty)\times (0,\infty)$
$$\langle n_k(t,\cdot,\cdot), \varphi \rangle := \E\big[\langle Z_t,\varphi\rangle\big],$$
We have (in a weak sense)
\begin{eqnarray}\label{eq: EDP adder simple}\f{\p}{\p t}  n_k + \f{\p}{\partial \i} (\kappa x n_k)+ \f{\p}{\partial x} \big(\kappa x n_k\big)  = - \kappa x B(\i)n_k(t,\i,x) ,
\\
\nonumber \\
\kappa x n_k(t,0,x)= 4k \kappa x \int_0^\infty B(\i) n_k(t,\i,2x)d\i, \label{eq: EDP adder simple bound}
\end{eqnarray}
 with $n_k(0,\i,x)= n^{(0)}(\i,x)$, $0\leq \i\leq x$.

\subsubsection*{Discussion on some generalisations}
We can follow two variables, one behaving like a physiological age, {\it i.e.} reset at zero at each division, but which evolves with a non constant rate $\tau_z(\i,x)$; the other behaving like a size, {\it i.e.} which is conserved by division through a fragmentation kernel $b(y,dx)$ such that $\int_0^y b(y,dx)=1,$ and growing at a rate $\tau_x(x,\kappa),$ with a parameter $\kappa$ chosen at birth according to the rate $\kappa'$ of the mother along a probability kernel $\theta(\kappa',d\kappa)$. We may assume that the division rate depends on size, growth rate and physiological age, and denote it $\beta(\i,x,\kappa)$; we write it as a time instantaneous rate, contrarily to the previous notations where $B$ is an age or a size instantaneous rate, thus multiplied by the corresponding age or size growth rate. 
%
We expect the mean measure $n_k = \E[\langle Z_t,\cdot\rangle]$ of the corresponding population process to solve a PDE of the form
\begin{align}
\label{eq: EDP gene}
&\f{\p}{\p t}  n_k + \f{\p}{\partial \i} \left(\tau_{\i}(\i,x) n_{k}\right)+ \f{\p}{\partial x} \left(\tau_{x}(x,\kappa) n_k\right)   = -\beta(\i,x,\kappa)n_{k}(t,\i,x,\kappa),
\\ \nonumber 
\\  \label{eq: EDP gene bound}
&\tau_x (0,\kappa)n_k(z,0,\kappa)=0,
\qquad \tau_{\i}(0,\kappa) n_{k}(t,0,x,\kappa)  =
\\ \nonumber \\
&\qquad 2k \int_0^\infty \int_0^\infty \int_0^\infty \theta(\kappa',\kappa) b(y,x)\beta(\i,y,\kappa')  n_{k}(t,\i,y,\kappa')d\i dyd\kappa'.\nonumber
\end{align}
We recover the age model by taking $\beta(\i,x,\kappa)=B(\i),$ $\tau_z\equiv 1,$ and integrating in size and growth rate; the size model with variable growth rate and unequal fragmentation by taking $\beta(\i,x,\kappa)=\tau_{x}(x,\kappa)B(x)$ and integrating in $\i$; the increment model with variable growth rate by taking $\tau_{\i}=\tau_{x}$, $\beta(\i,x,\kappa)=\tau_{x} (x,\kappa)B(\i).$
A detailed study  lies beyond the level intended in these notes but the interest is to embed all the models considered here in a common framework.

\section{Model analysis: long-time behaviour}
\label{sec:anal}
Having  built kinetic models leads naturally mathematicians towards the question of their long-time behaviour. In this matter, this is far from being a pure mathematical question: it reveals the cornerstone of our calibration strategy, developed in Section~\ref{sec:inverse}. It is however a whole field in itself: as in Section~\ref{sec:model}, we provide the main ingredients in the simplest case of the renewal equation and process - which already reveals  not so simple when we study the stochastic population model - and review - not exhaustively - the extremely rich literature for more involved cases.

 Importantly, we only consider here linear models, \emph{i.e.} we neglect feedbacks or exchanges with the environment or between the cells; we only mention a few nonlinear results. When linear, the study of the asymptotic behaviour of the equations are closely related to the spectral analysis of the related semigroup operator, and leads to three main types of behaviours:
 \begin{itemize}
 \item convergence to a steady state (exponentially fast in case of a spectral gap), which happens for the conservative equations (case $k=1$: genealogical observation);
 \item steady exponential growth, {\it i.e.} there is a decoupling between an exponential growth at the rate of the dominant eigenvalue, called the {\it Malthusian parameter} or {\it fitness} of the population, and a steady distribution in the structuring variables. This happens in the case $k=2$, where all the population is followed. As for the conservative case, an exponentially-fast trend to this steady behaviour is linked to a spectral gap.
 \item Other non physical behaviours, for instance spreading in the distributions, or yet trend to a permanently oscillating system, together with exponential growth in the case $k=2.$ This last case is a non-robust behaviour, linked to some degeneracy in the coefficients, so that the dissipation of entropy is not sufficient for a trend to a steady behaviour to emerge; it can also be seen as another type of convergence towards the dominant eigenvector, when this one becomes non unique. Whatsoever, this last behaviour, in the case of linear equations, is a mathematical curiosity rather than biologically informative. The interest of proving such non physical results consists in excluding some models or assumptions as non realistic if they lead to such non physical results.
 \end{itemize}
These considerations drive our assumptions: they need to ensure the convergence of both the conservative ($k=1$) and the nonconservative ($k=2$) equations towards steady behaviours.

\subsection{The renewal equation: a pedagogical example}
\label{subsec:anal:age}
This is historically the first structured-population model to be studied~\cite{Kermack1,MetzDiekmann}; another key reference is \cite{harris1948branching} around 1950 and later the celebrated textbook book by Harris \cite{harris1963theory}. There are strong links with the classical renewal theory of random walks in probability, that dates back to classics: W. Feller, J. Doob, A. Lotka \cite{feller2015integral,doob1948renewal,lotka1948application}.  A recent account of the link between fragmentation processes and renewal theory can be found in the textbook by J. Bertoin \cite{bertoin2006random}.



Let us denote the renewal equation in one of its simplest form:
\begin{equation}
\left\{
\begin{array}{ll}
\f{\p}{\p t} n(t,a) + \f{\p}{\p a} n(t,a) = - B(a) n(t,a),
\\ \\
n(t,0)=k \int_0^\infty B(a)n(t,a)da, \qquad n(0,a)=n^0(a).
\end{array}\right.
\label{eq:renewal}
\end{equation}
This equation has to be understood first in a weak sense, {\it i.e.} as an equivalent way to write~\eqref{eq: age edp}
\subsubsection*{Functional spaces through the lense of modelling}

To give a rigorous meaning to this equation, either in a weak or strong formulation, one first needs to decide in which functional space - in the structuring variable (age here), then time - the solution $n$ should be considered.

Closer to the stochastic process and to non-asymptotic or not-averaged populations are measure-valued solutions, see the recent and very pedagogical approach of P. Gabriel~\cite{gabriel2018measure} for the case $k=1$, or~\cite{gwiazda2016generalized}: this allows one to consider measure-valued initial conditions for $n^0.$ 

Convenient to handle the inverse problem or to use Fourier or Mellin transforms are $L^2-$ types spaces, see e.g.~\cite{PZ}, but they lack physical interpretation.  Assuming however that biological populations are expected to be in both $L^1$ and $L^\infty,$ this is acceptable, for mathematical reasons, to choose such spaces.

Finally, weighted $L^1$ spaces have been widely used, for  semi-group approaches~\cite{MetzDiekmann} as well as for general relative entropy~\cite{BP} or still others~\cite{gwiazda2006invariants}. They  have the advantage of immediate physical interpretation, since $\int n(t,a)da$ and $\int B(a)n(t,a)da$ represent respectively (the expectation of) the total number of individuals and the total number of dividing individuals at time $t$, $\int a n(t,a)da$ the average age of individuals at time $t$, etc. Moreover, assuming $n(t,\cdot) \in L^1((0,\infty))$ means that the expectation of the stochastic measure has a density and, up to a renormalization if $k=2,$ is a probability density. At least for large times, we consider that this is relevant in a modelling perspective, and thus priviledge this last family of functional spaces.

Let us briefly recall  how the main results may be found using the entropy approach developed in~\cite{mischler2002stability}, and discuss the links between the cases $k=0$ and $k=1$ - we refer to~\cite{BP}, chapter 2 for a complete presentation.

We assume
\begin{equation}
\label{as:B}
\begin{array}{c}
B\in L^\infty_{loc}\left((0,\infty),[0,\infty)\right),
\\ \\ \int^\infty B(x)dx =\infty, \qquad \int_0^\infty x B(x)e^{-\int_0^x B(a)da} dx <\infty.
\end{array}
\end{equation}
These assumptions have a stochastic interpretation: if $A$ is a random variable representing the age at division of a cell, the fact that $\int_0^\infty B(x)dx=\infty$ ensures that all cells divide. The probability density of $A$ being the function $x\mapsto B(x)e^{-\int_0^x B(a)da},$ the last assumption means that $\E[A] < \infty.$
\subsubsection*{Eigenelements}
As said in the introduction, the asymptotic behaviour is closely linked to the study of the spectrum of the linear operator under consideration: we may formally write
$$\f{d}{dt} n(t,\cdot)={\cal A}_k n(t,\cdot)$$
where ${\cal A}_k$ is defined by~\eqref{eq:renewal}. Eigenelements $(\lambda_k,N_k,\varphi_k)$ are solutions to the equations
$$ {\cal A}_k N_k =\lambda_k N_k , \qquad
{\cal A}_k^* \varphi_k  =\lambda_k \varphi_k ,$$
where ${\cal A}_k^*$ is the adjoint operator to ${\cal A}_k,$ defined by the following adjoint equation
$${\cal A}_k^* \varphi (a): = \f{\p}{\p a} \varphi (a) - B(a) \varphi(a) + k B(a) \varphi(0).$$
\begin{proposition}\label{prop:eigen:age}
Under Assumption~\eqref{as:B}, for $k=1$ and $k=2,$ there exists a unique solution $(\lambda_k,N_k,\varphi_k)$ to the system
\begin{equation}\label{eq:age:eigen}
\left\{\begin{array}{l}
\lambda_k N_k (a) + \f{\p}{\p a} N_k(a) + B(a)N_k(a)=0,\qquad \int_0^\infty N_k (a)da=1, \\ \\
N_k(0)=k \int_0^\infty B(a)N_k (a) da, \\ \\
\lambda_k \varphi_k (a) - \f{\p}{\p a} \varphi_k (a) + B(a)\varphi_k(a)=k B(a)\varphi_k (0),\qquad \int_0^\infty N_k (a)\varphi_k (a) da=1.
\end{array}
\right.
\end{equation}
Moreover, we have $\lambda_1=0,$ $\varphi_1\equiv 1$,  $\lambda_2>0,$ $N_k, \;\varphi_2 >0$, $\lambda_2$ is uniquely defined by the relation
\begin{equation}
\label{def:lambda2}
1=2\int_0^\infty B(a)e^{-\lambda_2 a - \int_0^a B(s)ds}da,
\end{equation}
and $\varphi_2 \in L^\infty (0,\infty)$, $\varphi_2(0)>0$ with the uniform bound
$$\Vert \varphi_2 \Vert_{L^\infty} \leq 2 \varphi_2 (0)=2\frac{\int_0^\infty e^{-\lambda_2a-\int_0^a B(s)ds}da}{\int_0^\infty sB(s) e^{-\lambda_2 s -\int_0^s B(\sigma)d\sigma}ds}.$$
\end{proposition}
\begin{proof} The fact that $\varphi_1\equiv 1$ can be seen directly from the fact that for $k=1$ the equation is conservative. In the case of the renewal equation, contrarily to more involved cases where the study of existence and uniqueness of eigenelements is a field in itself (see Section~\ref{sec:anal:growthfrag}), we can immediately compute that solutions must satisfy
\begin{equation}\label{def:eigen:age}
\left\{
\begin{array}{ll}
N_k (a)&= N_{k}(0) e^{-\lambda_{k} a-\int_0^a B(s)ds}, \,  N_{k}(0)=N_{k} (0) k \int_0^\infty B(a)  e^{-\lambda_{k} a-\int_0^a B(s)ds}da,
\\ \\
\varphi_k (a)&=\varphi_k (0)\left(1-k \int_0^a B(s) e^{-\lambda_k - \int_0^s B(\sigma)d\sigma} ds\right)e^{\lambda_k a + \int_0^a B(s)ds}\\ \\
&=k\varphi_k (0)\int_a^\infty B(s)  e^{-\lambda_k (s-a) - \int_a^s B(\sigma)d\sigma} ds
\end{array}
\right.
\end{equation}
From this relation and from Assumption~\eqref{as:B}, we deduce that $\lambda_1=0$ and obtain that $\lambda_2$ satisfies~\eqref{def:lambda2}  from the boundary condition at $a=0.$  Since the right-hand side of~\eqref{def:lambda2} is a continuously decreasing function of $\lambda_2$ that equals $2$ for $\lambda_2=0$ (thanks to the fact that $\int_0^\infty B dx=\infty$) and vanishes for $\lambda\to\infty$, this defines a unique $\lambda_2>0$. The normalisation condition $\int_0^\infty N_k (a)da=1$ ensures uniqueness of $N_k$ (and the convergence of the integral $\int_0^\infty N_1(a)da=1$ is guaranteed by the assumption $\int x B(x) e^{~\int_0^x Bda}<\infty$); the normalisation condition $\int_0^\infty N_k (a)\varphi_k (a) da=1$ ensures uniqueness of $\varphi_k$. The uniform bound for $\varphi_2$ is obtained by integrating by parts~\eqref{def:eigen:age}:
$$\varphi_2 (a) =2 \varphi_2 (0) (1-\int_a^\infty \lambda_2 e^{-\lambda_2 (s-a)-\int_a^s B(\sigma)d\sigma}ds) \leq 2\varphi_2 (0),$$
and we compute
$$\begin{array}{ll}
1&=\int_0^\infty N_2(a)da \implies N_2(0)=\left(\int_0^\infty e^{-\lambda_2a-\int_0^a B(s)ds}da\right)^{-1},
\\ \\
1&=\int_0^\infty N_2(a)\varphi_2(a)da=N_2(0)\varphi_2(0)\int_0^\infty \int_a^\infty B(s) e^{-\lambda s -\int_0^s B(\sigma)d\sigma}dsda
\\ \\
&=N_2(a)\varphi_2(a)\int_0^\infty sB(s) e^{-\lambda s -\int_0^s B(\sigma)d\sigma}ds
\\ \\
&\implies  \varphi_2(0)=\frac{\int_0^\infty e^{-\lambda_2a-\int_0^a B(s)ds}da}{\int_0^\infty sB(s) e^{-\lambda_2 s -\int_0^s B(\sigma)d\sigma}ds}.
\end{array}$$
\end{proof}
\begin{remark}We may find a solution for $k=2$ under the relaxed assumption
$$B\in L^\infty_{loc}\left((0,\infty),[0,\infty)\right), \quad\int_0^\infty B(x)dx >\ln (2),$$
which is weaker than~\eqref{as:B}.
 This would however  lead to $\lambda_1 <0$ in the case where $\int_0^\infty B dx<\infty,$ so that the conservative equation would have no conservative eigenvector, {\it i.e.} no possible trend to a steady behaviour, what should be excluded for modelling purpose.
 Similarly, if the assumption $\int x B(x) e^{-\int_0^x Bda} dx <\infty$ is not fulfilled, we cannot normalise the eigenpair by $\int N_1\varphi_1 dx=1$ because $\int N_1 \varphi_1 dx=\infty$ in this case.
 \end{remark}
\subsubsection*{Existence of solutions}
Many methods are available to prove existence and uniqueness of solutions in various spaces; we specially refer the interested reader to~\cite{MetzDiekmann,BP} and for measure solutions to~\cite{gabriel2018measure}. We only mention here an easy proof obtained by the Banach-Picard fixed point theorem and Duhamel's formula, inspired by~\cite{BP}, chapter 3,~\cite{bansaye2019ergodic} Appendix~B and~\cite{gabriel2018measure} (solution for the dual equation). We take an $L^1$ space which is natural for solutions, namely $L^1\left((1+B(a))da\right),$ so that the assumptions on $B$ are minimal.
\begin{proposition} For $B\in L^\infty_{loc}\left((0,\infty);(0,\infty)\right)$ such that $\int_0^\infty B(a)da=\infty,$ for $n^0 \in L^1\left((1+B(a))da\right),$
there exists a unique solution $n\in L^\infty_{loc}\left(0,\infty;L^1\left((1+B(a))da\right)\right)$ to~\eqref{eq:renewal} and we have the comparison principle:
$$\forall x,\; n^0_1(a)\leq n^0_2 (a) \implies n_1(t,a) \leq n_2 (t,a) \qquad \forall \;t,\,a\geq 0.$$\label{prop:exist:age}
\end{proposition}
\begin{proof}
Let $\lambda>0$ a given constant. We look for solutions $\tilde n (t,a)=e^{-\lambda t} n(t,a)$ to the  equation
$$
\left\{
\begin{array}{ll}
\f{\p}{\p t} \t n(t,a) + \f{\p}{\p a} \t n(t,a) +\lambda \t n= - B(a) \t n(t,a),
\\ \\
\t n(t,0)=k \int_0^\infty B(a)\t n(t,a)da, \qquad \t n(0,a)=n^0(a).
\end{array}\right.
$$We define $\psi(t)=\t n(t,t+a)$ which is solution to
$$\psi'(t) + \left(B(t+a)+\lambda\right)\psi =0,\qquad \psi(0)=n^0(a),$$
and similarly $\t \psi (t)=\t n(t+a,t)$ is solution to
$$\t \psi ' + \left(B(t) +\lambda\right)\t \psi =0,\qquad \t\psi(0)=\t n(a,0)=k \int_0^\infty B(s) \t n(a,s)ds$$
so that we find the Duhamel's formula
\begin{equation}\label{eq:duhamel}\begin{array}{ll}\t n(t,a)&=n^0(a-t)e^{-\int_0^t B(s+a)ds -\lambda t}\1_{\{a\geq t\}} 
\\ \\
&+ k e^{-\int_0^a B(\sigma)d\sigma -\lambda a}\int_0^\infty B(s)\t n(t-a,s) ds \1_{\{t\geq a\}}\\ \\
&:=\Gamma_{n^0}[\t n](t,a)
\end{array}.
\end{equation}
We consider the mapping $\Gamma_{n^0}$ taken on the weighted space $X=L^\infty\left(0,\infty;L^1\left((\f{\lambda}{2k}+B(a)\right)da\right):$ we find, for any $t\geq 0,$
$$\begin{array}{ll}\int_0^\infty \left\vert \left(\f{\lambda}{2k}+B(a)\right)\Gamma[n_1-n_2](t,a)\right\vert da
\\ \\
\qquad \leq \int_0^\infty k \left(\f{\lambda}{2 k}+B(a)\right)e^{-\int_0^a B(\sigma)d\sigma -\lambda a} da \left\Vert n_1-n_2\right\Vert_X
\\ \\
\qquad \leq  k \left(1+\int_0^\infty (\f{\lambda}{2k} -\lambda) e^{-\lambda a}da\right)
\left\Vert n_1-n_2\right\Vert_X
\\ \\
\qquad\leq \f{1}{2} \left\Vert n_1-n_2\right\Vert_X.
\end{array}
$$
This proves that $\Gamma$ is a strict contraction on $X,$ hence applying the Banach-Picard fixed point theorem we find a unique fixed point, which is solution to~\eqref{eq:renewal}. The comparison principle comes from the fact that if $n_0^1\leq n_0^2$ then $\Gamma_{n^0_1}[n]\leq \Gamma_{n^0_2}[n]$; we can iterate each of the operators $\Gamma_{n^0_i}$ and at the limit we find $\t n_1\leq \t n_2,$ hence $n_1\leq n_2.$
\end{proof}

\subsubsection*{General relative entropy}
A general fact shared by many linear population dynamics equations is a wide class of time-decreasing functionals, which may be used as a key ingredient to study the equation, in particular to prove long-time asymptotics~\cite{MMP2} or uniqueness of eigensolutions~\cite{DG}. We refer to~\cite{BP} for a thorough presentation in many situations.

\begin{lemma}
Let $H:[0,\infty) \to [0,\infty)$ a convex differentiable function,  and $n_1,$ $n_2$ two solutions of~\eqref{eq:renewal} in $L^\infty_{loc}([0,\infty); L^1((1+B(a))da)$ such that $n_2(t,x)>0$ for all $t,x\geq 0,$ and such that $\int \varphi_k (a) n_2(0,a) H\left(\f{n_1(0,a)}{n_2(0,a)}\right)da <\infty.$ 
We have the following inequality:
 \begin{equation}
 \label{ineq:GRE:age}
{\cal H}(t):=\int \varphi_k(a) e^{-\lambda_k t} n_2(t,a) H\left(\f{n_1(t,a)}{n_2(t,a)}\right) da\qquad \implies \qquad \f{d}{dt} {\cal H}(t) \leq 0.
 \end{equation}
\end{lemma}
\begin{proof}
We compute, using the equations and integrating by parts:
$$\begin{array}{lll}
 \f{d}{dt}{\cal H}(t)&=&
 \int_0^\infty e^{-\lambda_k t}\left\{ \left(-\lambda \varphi_k (a) n_2 (t,a) +\varphi_k(a) \p_t n_2 (t,a) \right)H\left(\f{n_1}{n_2}(t,a)\right) \right.
  \\ \\
& &\left. + \varphi_k(a) \left(\p_t n_1(t,a)-\f{n_1}{n_2}(t,a) \p_t n_2 (t,a)\right) H'\left(\f{n_1}{n_2}(t,a)\right)
 \right\}da
 \\ \\
 &=&\int_0^\infty e^{-\lambda_k t}\left\{ \left(-\lambda \varphi_k (a) n_2 (t,a) +\varphi_k(a)(-\p_a n_2 (t,a) - B(a)n_2(t,a) \right)H(\f{n_1}{n_2}(t,a)) \right.
  \\ \\
  &&\left. +\varphi_k(a) \left(-\p_a n_1-Bn_1
-\f{n_1}{n_2} (-\p_a n_2 -B n_2)\right)(t,a) H'\left(\f{n_1}{n_2}(t,a)\right)
 \right\}da
 \\ \\
& =&\int_0^\infty e^{-\lambda_k t} \left(-\lambda \varphi_k  n_2  +\p_a \varphi_k n_2  - B\varphi_k n_2 \right)H\left(\f{n_1}{n_2}\right)(t,a) da
 \\ \\
   && +\int_0^\infty e^{-\lambda_k t} \varphi_k(a)  \left(\p_a n_1(t,a) - \f{n_1}{n_2} (t,a)\p_a n_2(t,a)\right)H'\left(\f{n_1}{n_2}(t,a)\right)da
  \\ \\&&
  +e^{-\lambda_k t}\varphi_k n_2 H\left(\f{n_1}{n_2}\right)(t,0)+ \int_0^\infty e^{-\lambda_k t}  \varphi_k(a) \left(-\p_a n_1
+\f{n_1}{n_2} \p_a n_2 \right) H'\left(\f{n_1}{n_2}\right) (t,a)
 da
 \\ \\
 &=&e^{-\lambda_k t}\left\{ \int_0^\infty - k \varphi_k(0)B (a)n_2(t,a) H\left(\f{n_1}{n_2}(t,a)\right)da \right.
 \\ \\
 && \left. +k \varphi_k(0) H \left(\f{n_1}{n_2}(t,0)\right)\int_0^\infty B(a)n_2(t,a)da\right\}
  \\ \\
 &=&k\varphi_k(0)e^{-\lambda_k t} \int_0^\infty   B(a)n_2(t,a)\left\{- H\left(\f{n_1}{n_2}(t,a)\right) + H \left(\int_0^\infty B(s)\f{n_1(t,s)}{n_2(t,0)}ds\right)\right\}da
 \\ \\
 &\leq& 0,
  \end{array}
 $$
 the last inequality following by Jensen's inequality applied to the convex function $H$ applied to the function  $f(a)=\f{n_1}{n_2}(t,a)$, with respect to the probability measure $B(a)\f{n_2(t,a)}{n_2(t,0)}da.$
\end{proof}
\begin{remark}In full generality, we could replace $\varphi_k(x)e^{-\lambda_k t}$ by any solution of the adjoint equation of~\eqref{eq:renewal}. We can also relax the assumption on $H$ differentiable, by taking a regularising sequence, since at the limit the terms involving $H'$ cancel each other. For $k=1,$ we have $\varphi_k e^{-\lambda_k t}\equiv 1:$ the "general relative entropy" is a standard relative entropy between $n_1$ and $n_2.$
\end{remark}
\subsubsection*{Long-time behaviour}
The long-time behaviour of the solution $n$ may be obtained by several methods: the semi-group theory~\cite{MetzDiekmann,mischler2016spectral}, the use of a Laplace transform~\cite{Feller}, the general relative entropy~\cite{BP} extended recently to measure solutions~\cite{gwiazda2016generalized}, use of invariants~\cite{Gw}, and very recently for measure-valued solutions Harris theorem and Doeblin's conditions~\cite{Bansaye2,Bansaye}. Whatever the method, the general fact is that under suitable assumptions, we have  the convergence
 $$n_k(t,a)e^{-\lambda_k t} \to  N_k(a) \int_0^\infty n_k^0(a) \varphi_k(a)da$$ with $\lambda_k$ and $N_k$ defined in Proposition~\ref{prop:eigen:age}, and this convergence is exponentially fast in functional spaces where a spectral gap may be proved. More specifically, let us cite the following result, adapted from~\cite{bansaye2019ergodic}, theorem 3.8.
 \begin{theorem}{\bf{(From~\cite{bansaye2019ergodic})}}
 \label{theo:asymp:age}Assume that there exists $a_0>0,$ $\underline{B}>0$, $p>0$ and $\ell \in (p/2,p]$
 such that $\forall a\in [a_0+\N p, a_0\ell +\N p], B(a)\geq \underline{B}.$ Then there exist $C>0,$ $\rho>0$ (which can be explicitely computed) such that for any positive finite measure $\mu^0,$ there exists a unique measure-valued solution $\mu_t$ to~\eqref{eq:renewal} in a weak sense, and the following inequality holds:
 $$\Vert e^{-\lambda t} \mu_t - \langle \mu^0,\varphi_k \rangle N_k\Vert_{TV} \leq C \Vert \mu^0\Vert_{TV} e^{-\rho t},$$
 where $\Vert\cdot\Vert_{TV}$ denotes the total variation norm.
 \end{theorem}
 Note that the assumptions are more restrictive than for the existence or the eigenproblem results, and still more restrictive to prove an exponential rate of convergence. To give an elementary intuition of this convergence, let us look at the case $B>0$ constant: then, Proposition~\ref{prop:eigen:age} shows that $\lambda_k=(k-1)B$, $\varphi_k \equiv 1$, $N_k(a)=k Be^{-k Ba}$ and a simple integration of the equation implies
 $$ \int_0^\infty n(t,a) da=e^{(k-1) Bt} \int_0^\infty n^0(a)da,$$
 we now  go back to the Duhamel formula~\eqref{eq:duhamel} and find
$$ \begin{array}{ll}n(t,a)&=n^0(a-t)e^{-\int_0^t B(s+a)ds}\1_{\{a\geq t\}} + k e^{-\int_0^a B(\sigma)d\sigma}\int_0^\infty B(s) n(t-a,s) ds \1_{\{t\geq a\}}
\\ \\
&=n^0(a-t)e^{- Bt}\1_{\{a\geq t\}}+N_k (a) e^{\lambda_k t}\1_{\{t\geq a\}}\int_0^\infty  n^0(a)\varphi_k (a) da.
\end{array}
$$
The first term of the right-hand side vanishes exponentially fast at rate $B,$ and the second term of the right-hand side is the expected limit.

\subsection{The renewal process} \label{sec: renewal}

Let us first discuss some heuristics for the convergence of empirical measures for further statistical applications. In order to extract information about $a \mapsto B(a)$, we consider the empirical distribution for a test function $g$ defined by
$$\mathcal E^T(g) = |\mathcal V_T|^{-1}\sum_{u \in \mathcal V_T}g(\zeta_u^{T}),$$
where
\begin{equation}\label{def:VT}
\mathcal V_T = \{u \in \mathcal U,\;b_u \leq T,\;b_u+\zeta_u > T\},
\end{equation}
 {\it i.e.}  the population of cells that are alive at time $T$, and $\zeta_u^T = T-b_u$ is the value of the age trait at time $T$.
We expect a law of large number as $T \rightarrow \infty$.\\

Heuristically, we postulate for large $T$ the approximation
$$\mathcal E^T(g) \sim
\frac{1}{\E[|{\mathcal V}_T|]}\E\Big[\sum_{u \in \mathcal V_T}g(\zeta_u^{T})\Big].$$
Then, a classical result  based on renewal theory (see Theorem 17.1 pp 142-143 of \cite{harris1963theory}) gives the estimate
$$
\E\big[|\mathcal V_T|\big] \sim \kappa_{B}e^{\lambda_{2}T},
$$
where $\lambda_{2}>0$ is the Malthusian parameter of the model, defined as the unique solution to
$$
\int_0^\infty B(x)e^{-\lambda_2x-\int_0^xB(u)du}dx=\frac{1}{2},
$$
as in \eqref{def:lambda2},
and $\kappa_{B}>0$ is an explicitly computable constant.
As for the numerator,
call $\chi_t$ the size of a particle at time $t$ along a branch of the tree picked at random. The process $(\chi_t)_{t\geq 0}$ is  Markov process with values in $[0,\infty)$ with infinitesimal generator
\begin{equation} \label{def generator}
\mathcal A g(x) = g'(x)+B(x)\big(g(0)-g(x)\big)
\end{equation}
densely defined on bounded continuous functions.
It is then relatively straightforward to obtain the identity
\begin{equation} \label{firstMtO}
\E\big[\sum_{u \in {\mathcal V}_T}g(\zeta_u^{T})\big] = \E\big[2^{N_T}g(\chi_T)\big],
\end{equation}
where $N_t = \sum_{s \leq t}{\bf 1}_{\{\chi_s-\chi_{s_-}>0\}}$ is the counting process associated to $(\chi_t)_{t \geq 0}$.
Putting together $\E\big[|\mathcal {V}_T|\big] \sim \kappa_{B}e^{\lambda_2T}$ and \eqref{firstMtO}, we thus expect
$$\mathcal E^T(g) \sim \kappa_{B}^{-1}e^{-\lambda_2T}\E\big[2^{N_T}g(\chi_T)\big],$$
and we anticipate that the term $e^{-\lambda_2 T}$ should somehow be compensated by the term $2^{N_T}$ within the expectation. To that end, following \cite{cloez2017limit} (and also \cite{bansaye2011limit} when $B$ is constant) one introduces an auxiliary ``biased" Markov process $(\widetilde \chi_t)_{t\geq 0}$, with generator $\mathcal A_{H_B}g(x) = \mathcal A g(x) = g'(x)+H_B(x)\big(g(0)-g(x)\big)$ for a biasing rate $H_B(x)$ characterised by
$$f_2(x)=H_B(x)\exp(-\int_0^x H_B(y)dy),$$
with
\begin{equation} \label{characterisation H_B}
f_2(x)=2e^{-\lambda_2x}f_1(x),\;x \geq 0,
\end{equation}
where
$$f_1(a) = B(a)\exp\Big(-\int_0^a B(s)ds\Big)$$
is the typical lifetime of a cell (without observation bias), or equivalently the density distribution of $\zeta_u$, see \eqref{eq: distrib classical} below.
This choice (and actually this choice only) enables us to obtain
\begin{equation} \label{first HB}
e^{-\lambda_2T}\E\big[2^{N_T}g(\chi_T)\big] = 2^{-1}\E\big[g(\widetilde \chi_T)B(\widetilde \chi_T)^{-1} H_B(\widetilde \chi_T)\big]
\end{equation}
with $\widetilde \chi_0=0$. Moreover $(\widetilde \chi_t)_{t \geq 0}$ is geometrically ergodic, with invariant probability $c_B\exp(-\int_0^xH_B(y)dy)dx$. We further anticipate
\begin{align*}
\E\Big[g(\widetilde \chi_T) B(\widetilde \chi_T)^{-1} H_B(\widetilde \chi_T)\Big] & \sim c_B\int_0^\infty g(x)B(x)^{-1}H_B(x)e^{-\int_0^xH_B(y)dy}dx \\
& = 2c_B\int_0^\infty g(x)e^{-\lambda_2 x}B(x)^{-1}f_1(x)dx
\end{align*}
assuming everthing is well-defined,
 by \eqref{characterisation H_B}.
Finally, we have $\kappa_B^{-1}c_B = 2\lambda_2$ which enables us to conclude
\begin{equation} \label{def limite bord}
\mathcal E^T(g) \rightarrow \mathcal E(g) := 2\lambda_2\int_0^\infty g(x)e^{-\lambda_2x} e^{-\int_0^x B(y)dy} dx
\end{equation}
as $T \rightarrow \infty$. The convergence is in probability, with some explicit rate linked to $\lambda_2$, as given below.



\begin{definition} The family of (real-valued) random variables $(\Upsilon_T)_{T>0}$ is asymptotically bounded in probability if
$$\limsup_{T \rightarrow \infty} \mathbb P\big(|\Upsilon_T| \geq K\big) \rightarrow 0\;\;\text{as}\;\;K\rightarrow \infty.$$
\end{definition}
In other words, the family of distributions of the random variables $\Upsilon_T$ is tight, or weakly relatively compact on a neighbourhood of $T=\infty$.
\begin{proposition}{\bf{Rate of convergence for particles living at time $T$ (Theorem 3 in \cite{hoffmann2016nonparametric}}} \label{rate en T}
Assume $\lambda_2 \leq 2\inf_x H_B(x)$. If $B$ is differentiable and satisfies $B'(x) \leq B(x)^2$ and $0 < c \leq B(x) \leq 2c$ for every $x \geq 0$ and some $c>0$, then
$$e^{\lambda_2T/2}\big(\mathcal E^T\big(g\big) - {\mathcal E}(g)\big)$$
is asymptotically bounded in probability.
\end{proposition}

\subsection{The growth-fragmentation equation}
\label{sec:anal:growthfrag}
When the structuring variable is the size, the so-called growth-fragmentation equation appears in many applications, from TCP-IP protocol to polymerization reactions.  Let us write~\eqref{eq: EDP size simple} here under a more general form, with a growth rate $\tau(x),$ a division rate $B(x)=\tau(x)B(x),$ an a fragmentation kernel $b(y,dx)$ representing the probability distribution (in $dx$) of the offspring of a dividing individual of size $y$:
\begin{equation}\label{eq:growthfrag}
\left\{\begin{array}{ll}
\f{\p}{\p t} n_k (t,dx)+ \f{\p}{\p x}\big(\tau(x)n_k(t,dx)\big) = -B(x) \tau(x) n_k(t,dx)
\\ \\ \qquad\qquad\qquad+k \int_{y=x}^\infty  B(y)\tau(y) n(t,y)b(dy,x),
\\ \\
\tau(0)n_k(t,0)=0,
\end{array}\right.
\end{equation}
with two conditions ensuring conservation of mass through division and division into two (this second assumption is easily relaxed but this would be meaningless in our application context):
$$\int_{0}^y xb(y,dx)=\f{y}{2},\qquad \int_{0}^y  b(y,dx)=1.$$
An important simplification often done is to restrict the equation to so-called {\it self-similar} division kernel $b,$ meaning that the division place only depends on the ratio between the mother size and the daughter size: in such a case, we have
\begin{equation}\label{as:fragself}
b(y,dx):=\f{1}{y}b_0(\f{dx}{y}),\qquad \int_0^1 b_0(dz)=1,\qquad \int_0^1 zb_0(dz)=\f{1}{2},
\end{equation}
and modelling considerations also lead to $b(y,dx)=b(y,y-dx)$, leading to $b_0$ symetric in $1/2$
(in a still more general way, we may consider a stochastic number of children, see~\cite{bertoin2006random}). In the study of the equation, the moments play a very important role, and the moments of order zero and one have a physical interpretation: integrating the equation, we obtain - formally at this stage
$$\f{d}{d t} \int n_k(t,x) dx = (k-1) \int \tau(x)B(x) n_k(t,x) dx, $$
which means that for $k=1$ the number of individuals is constant - the equation is conservative - whereas for $\ep=2$ it grows due to the division process. Integrating the equation against the weight $x$, we have
$$\f{d}{d t} \int x n_k(t,x) dx =(\f{k}{2}-1)  \int B(x)\tau (x) n_k (t,x) dx + \int \tau(x) n_k(t,x)dx,$$
which is easily interpreted: for $k=1$ the mass increases with growth but decreases with division, whereas for $k=2$ it only increases with growth. More generally, moments of order $p$ show a balance between growth and division:
$$\f{d}{dt} \int_0^\infty x^p n dx=  \int_0^\infty \left(p -B(x)x\bigl(1-k\int_{0}^x \f{y^p}{x^p} b(x,dy)\bigr)\right)  x^{p-1}\tau(x) n dx.$$
This balance between growth and division leads to the main asymptotic behaviour to be expected: the convergence to a steady size-distribution profile, with exponential growth in time for the population case $k=2,$ at en exponential rate of convergence if a spectral gap is proved. This study begins in the 1980's with the work by Diekmann, Heijmans, Thieme and Gyllenberg and Webb~\cite{DiekmannHeijmansThieme1984,diekmann2003steady}, based on the theory of semigroups, and generally carried out under the assumption of a compact support for the size so that general theorems may apply in a more direct way. This has been then generalised by several authors, see~\cite{mischler2016spectral,bansaye2019ergodic}. Explicit solutions, for power law rates and specific fragmentation kernels have also been studied~\cite{escobedo2017short,doumic2018explicit,suebcharoen2011asymmetric}.

\subsubsection*{Eigenelements}
The eigenvalue problem and its adjoint are as follows:
\begin{equation}
\label{eq:eigenproblem}
\left \{ \begin{array}{l}
 \f{\p}{\p x} ((\tau  N_k)(x)) + \lambda_k N_k(x) =- (\tau B N_k)(x) +  k\int_x^\infty (\tau B N_k)(y)b(y,x)  dy,
\\
\\
\tau N_k(x=0)=0 ,\qquad N_k(x)\geq0, \qquad \int_0^\infty N_k(x)dx =1,
\\
\\
 -\tau (x) \f{\p}{\p x} (\varphi_k(x))  + \lambda_k\varphi_k(x) = B(x)\tau(x) (- \varphi_k(x) + k \int_0^x b(x,y) \varphi_k(y) dy),
\\
\\
\varphi_k(x)\geq 0, \qquad \int_0^\infty \varphi_k(x)N_k(x)dx =1.
\end{array} \right.
\end{equation}

These equations are meant in a weak sense, {\it i.e.} we look for $N_k\in L^1((0,\infty),dx)$ such that $\forall \phi\in {\cal C}_c^\infty ([0,\infty))$ we have
$$-\int_0^\infty \tau N_k\p_x\phi dx + \lambda \int_0^\infty \phi dx=\int_0^\infty B \tau N_k \left(k\int_0^x \phi (y)b(x,dy) -\phi(x)\right)dx,$$
and $\varphi\in W^{1\infty}_{loc}$ is solution almost everywhere.
For existence and uniqueness of eigenelements, the following assumptions are among the most general ones - except the fact that rates are at most polynomially growing, which is relaxed in Theorem~\ref{theo:asympCGY} below, but which was useful in the proof of~\cite{DG} based on moment estimates:
\begin{equation}\begin{array}{c}
Supp(b(y,\cdot))\subset [0,y],\qquad \int_0^y b(y,dx)=1,
\\ \\  \int_0^y xb(y,dx)=\f{y}{2},\qquad\int_0^y \f{x^2}{y^2} b(y,dx)\leq c<\f{1}{2},
\\
\\
\tau,\;B \in {\cal P}:=\left\{f\geq 0:\,\exists\,\mu,\,\nu\geq 0,\, \limsup\limits_{x\to\infty}x^{-\mu} f(x) <\infty,\, \liminf\limits_{x\to\infty} x^\nu f(x)>0\right\},
\\ \\
\tau B \in L^1_{loc}((0,\infty)),\quad \exists\,\alpha_0\geq 0,\quad \tau \in L^\infty_{loc} ([0,\infty), x^{\alpha_0} dx),
\\ \\
\forall\,K \,\text{compact on }(0,\infty),\quad \exists\,m_K>0,\quad \tau(x)\geq m_K \,{a.e.} \,x\in K,
\\ \\
\exists \,B_0\geq 0,\,C>0,\,\gamma\geq 0,\; Supp(B)=[b,\infty),\; \int_0^x b(y,dz) \leq \min\left(1,C(\f{x}{y})^\gamma\right),
\\ \\
B,\;\f{x^\gamma}{\tau}\in L^1_0:=L^1_{loc}([0,\cdot)),\qquad \lim\limits_{x\to\infty} xB =\infty.
\end{array}
\label{as:growthfrag}
\end{equation}

\begin{theorem}{\bf{(From~\cite{DG})}} \label{theo:eigenMDPG}   Under the balance assumptions~\eqref{as:growthfrag} on $\tau,$ $B$ and $b$, there exists a unique triplet $(\lambda_k,N_k,\varphi_k)$ with  $\lambda_2 >0,$ $\lambda_1=0$, $\varphi_1\equiv 1$ solution of the  eigenproblem \eqref{eq:eigenproblem}
and $$x^\alpha \tau N_k\in L^p(\R^+),\quad\forall \alpha\geq-\gamma,\quad\forall p\in[1,\infty],\quad x^\alpha \tau N_k\in W^{1,1}(\R^+),$$
$$\exists p>0\ s.t.\ \f{\varphi_2}{1+x^p}\in L^\infty(\R^+), \quad \tau\f{\p}{\p x}\varphi_2\in L_{loc}^\infty(\R^+).$$
\end{theorem}
We notice that the assumptions on $B$ at $\infty$ are very similar to the ones for the renewal equation. If  $B(x)=x^\gamma,$ the assumptions on $B$ are satisfied for $1+\gamma >0$, see~\cite{M1}. The proof is  based on standard theorems for regularized equations (Krein-Rutman or Perron-Frobenius), the compactness obtained by successive moments estimates, and uniqueness by using general relative entropy inequalities.

Finally, a quick computation shows that in the case of exponential growth $\tau(x)\equiv x$ and of division into two equally-sized daughters $b(y,x)=\delta_{x=\f{y}{2}}$, we have $\varphi_2 (x)\equiv C x$ with $C>0$ a normalisation constant, and $N_{1}(x)=C'xN_{2}(x)$ with $C'>0$ another normalisation constant.

\subsubsection*{General Relative Entropy}

\begin{lemma}{\bf{(General Relative Entropy Inequality - \cite{MMP2}, Theorem 2.1)}} \label{lem:GRE}
Let $n_1,$ $n_2$ be two solutions of~\eqref{eq:growthfrag}, with $n_{2}(t,x)>0$ for all time and size, $H:\R\to [0,\infty)$ positive, differentiable and convex, and $\int_0^\infty \varphi(x) n_{2}(0,x)H\left(\f{n_1}{n_2}(0,x)\right) dx<\infty.$ Then we have
$${\cal H} (t):= \int_0^\infty \varphi(x) e^{-\lambda t} n_2(t,x)H\left(\f{n_1}{n_2}(t,x)\right) dx,\; \f{d{\cal H}}{dt}=-D^H[n_1,n_2]\leq 0,
$$
with
\begin{align*}D^H[n_1,n_2](t)=\int_0^\infty\int_0^\infty k\varphi(x)e^{-\lambda t} n_2(t,y)\tau(y)B(y)b(y,x)
\\ \\ \left\{
H\left(\f{n_1}{n_2}(t,y)\right) - H\left(\f{n_1}{n_2}(t,x)\right) \right.\left. - H'\left(\f{n_1}{n_2}(t,x)\right) \{\f{n_1}{n_2}(t,y) -\f{n_1}{n_2}(t,x)
\} \right\}dxdy.
\end{align*}
\end{lemma}
We let the reader check directly this computation or refer to~\cite{MMP1} or~\cite{BP}. We also observe that the entropy for the renewal equation may be viewed as a specific case of this inequality: it suffices to take $b(y,x)=\delta_{x=0}$ and $\tau(x)=1.$  As for the renewal equation, it is possible to prove long-term convergence by means of this entropy inequality~\cite{BP,MMP1}, and exponential rate of convergence through entropy-entropy dissipation inequalities~\cite{caceres2011rate}, {\it i.e.} if we can bound $-D^H$ by a quantity depending on ${\cal H}.$

\subsubsection*{Long time asymptotics: the central case of asynchronous exponential growth}

Many methods have been developed to study the long time asymptotics of growth-fragmentation equations, from semi-group theory, general relative entropy, to methods inspired by stochastic processes such as several very recent studies~\cite{FournierPerthame,canizo2020spectral}. Let us cite here a recent result, carried out only for the two extreme cases of uniform (\emph{i.e.} $b_0= 2\1_{[0,1]}$) or equal mitosis({\it i.e.} $b_0=2\delta_{1/2}$)  fragmentation kernels, but relatively general for the assumption on the fragmentation rate, which may grow faster than polynomially.

\begin{assumption}{\bf{(Assumptions for a spectral gap, see~\cite{canizo2020spectral})}}\label{as:theoCGY}

\

\begin{itemize}
\item $b(y,x)=\f{2}{y}\1_{\{x\leq y\}}$ or $b(y,x)=2\delta_{x=\f{y}{2}}.$
\item $\tau$  is locally Lipschitz, $g(x)=O(x)$ around $\infty$, $g(x)=O(x^{-\xi})$ around $0$ with $\xi\geq 0$,
\item
$B:\; (0,\infty) \to [0,\infty)$ is continuous,
\item
$\int_0 B dx <\infty,\qquad \lim\limits_{x\to 0} xB(x)=0,\qquad \lim\limits_{x\to\infty} xB(x)=\infty. $
\item If $b(y,x)=2\delta_{x=\f{y}{2}},$ the growth rate $\tau$ must moreover satisfy
\begin{itemize}
\item $\omega g(x) < g(\omega x)$ for all $x>0$ and $0<\omega<1,$
\item  $H(z):=\int_0^z \tau^{-1}(z)dz <\infty$ for all $z>0,$
\item $\lim\limits_{z\to \infty} H^{-1}(z+r) / H^{-1}(z) =1.$
\end{itemize}
\end{itemize}
\end{assumption}
Restricted to $\tau(x)=x^\nu$, we see that the assumptions imply $\nu\leq 1$ for the uniform kernel, $\nu<1$ for the equal mitosis kernel, and allow any growth for $B$ at infinity: compared to the assumptions for the existence of eigenelements, the main restriction, apart from the specific shapes of the fragmentation kernel, is that we cannot consider  superlinear growth rates , since then the cell sizes may explode in finite time.
\begin{theorem}{\bf{(Theorem~1.3 from~\cite{canizo2020spectral}}}
\label{theo:asympCGY}Under Assumptions~\ref{as:theoCGY}, there exists a unique eigentriplet $(\lambda_2,N_2,\varphi_2)$ solution to~\eqref{eq:eigenproblem}. Let us denote, for $k\leq 0$ and $K>1,$ the following weighted total variation norm
$$\Vert\mu\Vert_{k,K}:=\int_0^\infty (x^k + x^K) \vert \mu\vert (dx).$$
Then for $n_0$ a nonnegative finite measure satisfying $\Vert n_0\Vert_{k,K}<\infty,$ there exists a unique measure-valued solution $n_2$ to~\eqref{eq:growthfrag} and it satisfies, for some $C>0$ and $\rho>0,$
$$\Vert e^{-\lambda_2 t} n_2(t,\cdot) - <\varphi_2,n_0> N_2 \Vert_{k,K} \leq C e^{-\rho t} \Vert n_0 - <\varphi_2,n_0> N_2\Vert_{k,K},
$$
with the following possible choices for $k$ and $K:$
\begin{enumerate}
\item if $\int_0 \tau(x)^{-1}dx <\infty,$ take $k=0$ and any $K>1+\xi,$
\item if $\int_0 \tau(x)^{-1} dx=\infty,$ take any $k\in (-1,1)$ and $K>1+\xi,$
\item if $\tau(x)=x,$ take any $k\in (-1,1)$ and $K>1.$
\end{enumerate}
\end{theorem}
The assumptions on the state space where the convergence holds are crucial to obtain the exponential speed of convergence, which is linked to a spectral gap. Specifically, P. Michel, S. Mischler and B. Perthame proved convergence - without speed -  in the  weighted space $L^1(\varphi_2 dx)$~\cite{MMP2}, which is the most natural space to prove convergence results through the general relative entropy inequality; but under the assumption of bounded fragmentation rates, E. Bernard and P. Gabriel  proved that there exists no spectral gap in this space: the convergence may hold arbitrarily slowly for well-chosen initial conditions, see Theorems~1.2 and~4.1 in~\cite{bernard2017asymptotic}. Among other important results that we cannot review in detail here, let us cite fine estimates on the eigenvector and adjoint eigenvector~\cite{caceres2011rate,balague:hal-00683148}, semigroup approaches~\cite{mischler2016spectral}, probabilistic approaches~\cite{bertoin2019feynman,bertoin2020strong}.

\subsubsection*{Long time asymptotics: other cases}

When the balance assumptions between growth and division around zero and around infinity fail to be satisfied, other types of asymptotic behaviour may happen, leading to mass escape towards zero (dust formation or shattering) or infinity (gelation). Let us focus on the case of exponential growth $\tau(x)=x$, interesting in several ways: as already said, it is the idealised growth rate for many unicellular organisms, like bacteria; it is also the limit case before characteristic curves grow to infinity in finite time; last but not least, it appears by a change of variables when studying the asymptotic trend to a self-similar profile for the pure fragmentation equation, see~\cite{MischlerRicard} Theorem 3.2. Assuming a power law for the division rate $B(x)=x^\gamma,$ we can classify the anomalous asymptotic behaviours according to the value of $\gamma.$
\begin{itemize}
\item $\gamma <-1:$ in such a case, there is a loss of mass by dust formation in finite time called {\it shattering}~\cite{haas2003loss,haas2004regularity,haas2010asymptotic,goldschmidt2010behavior,dyszewski2021sharp,goldschmidt2016behavior}, non-uniqueness of solutions~\cite{banasiak2006shattering,banasiak2019analytic}.

\item $\gamma=-1:$ this is a case where the {\it time-dependent} division rate $\beta(x)=B(x)\tau(x)$ is constant. It is a limit case, where there is neither loss of mass in finite time nor  convergence to a steady profile and exponential growth, since each moment of the equation grows or decays exponentially with a specific rate, see~\cite{doumic2016time}.
 This is interesting from a modelling perspective because it explains the fact that a model with both exponential growth in size and size-independent division (for instance an age-structured division rate for exponentially growing cells) is irrelevant, leading to a non realistic exponential behaviour since no steady profile in size may be obtained~\cite{robert:hal-00981312}. Generalisations of this limit case to $\tau(x)=x^{1+\gamma}$ and $B(x)\tau(x)=x^\gamma$ with $\gamma>0$, leading to blow-up in finite time, has been done by M. Escobedo~\cite{escobedo2017short,escobedo2020non} and J. Bertoin and A. Watson for the corresponding stochastic processes~\cite{bertoin2016probabilistic}.

\item $\gamma>-1:$ in general, convergence theorems such as Theorem~\ref{theo:asympCGY} are valid, however for very specific division kernel such as the idealised equal mitosis case, the solution may converge to a cyclic behaviour. In such a case, we still have a general relative entropy inequality given by Lemma~\ref{lem:GRE}, but a simple computation shows that the entropy dissipation $D^H$ vanishes not only for $\f{n_1}{n_2}$ constant, but also for any ratio satisfying
$$\f{n_1(x)}{n_2(x)}=\f{n_1(2x)}{n_2(2x)},$$
which is a kind of periodicity condition. The best intuition on what happens here comes from the underlying stochastic branching tree: all descendants of a given cell of size $x_0$ at time $0$ live on the countable set of curves $x_0 e^t 2^{-n},$ due to the very specific relation between growth and division, whereas for other growth or division the times of division account, leading to a kind of dissipativity. This case has been studied by semigroup theory for compact support in size by G. Greiner and R. Nagel~\cite{greiner1988growth}, and extended and revisited in~\cite{bernard2016cyclic} where the following explicit asymptotic result has been proved - see also~\cite{gabriel2019periodic} for extension to measure solutions.
\end{itemize}
%
%
%

\begin{theorem}{\bf{(Theorem~2.3. in~\cite{bernard2016cyclic})}}\label{theo:mitosis}
Let $\tau>0$ and define $\tau(x)= x$, $b(y,x)=2\delta_{x=\f{y}{2}}$, and $B$ such that
\begin{equation}\label{hyp:B}\left\{
\begin{array}{l}
B:\; (0,\infty) \to (0,\infty) \text{ is measurable},\;
B\in L^1_{loc}([0,\infty)),\\ \\
\exists\gamma_0,\gamma_1,K_0,K_1,x_0>0,\quad K_0x^{\gamma_0}\leq xB (x\geq x_0)\leq K_1x^{\gamma_1}.\end{array}\right.
\end{equation}
Theorem~\ref{theo:eigenMDPG} holds, but there also exists a countable set of nonpositive dominant eigentriplets defined, for $m\in\Z,$ by

 $$\lambda_k^m=(k-1)+\f{2im\pi}{\ln 2},\qquad N_k^m(x)=x^{-\f{2im \pi}{\ln 2}} N_k^m(x),\qquad \varphi_k^m (x)= c_m x^{k-1+\f{2im\pi}{\ln 2}},
$$
 with $c_m$ normalisation constants. All the quantities $<n_k(t,\cdot),\varphi^m_ke^{-\lambda_k^m t}>$ are then conserved, and
for any $n_0 \in$ $ L^2([0,\infty),x^{k-1}/\cU_k^0 (x) dx),$ the unique solution $n_k(t,x)$   $\in$ $C\big([0,\infty),L^2([0,\infty),x^{k-1}/\cU_k^0 (x) dx)\big)$ to~\eqref{eq:growthfrag}  satisfies
\[\int_0^\infty\, \bigg|n_k(t,x)e^{-\kappa (k-1)t} - \sum_{m=-\infty}^{\infty}<n_0,\varphi_k^m>\,\cU_k^m(x) e^{\frac{2im\pi}{\ln 2}t} \bigg|^2\frac{x\,dx}{\cU_k^0(x)} \xrightarrow[t\to\infty]{}0.\]
\label{theo:osci}
\end{theorem}
A numerical scheme needs to be non-dissipative to capture the oscillations, for instance by splitting transport and fragmentation and by using a geometric grid, see Section~3 in~\cite{bernard2016cyclic}; another way would be to use a splitting particle method~\cite{carrillo2014splitting}.

\subsection{Structured population equations and processes}
\label{subsec:model:gene}
For more general models and methods, several excellent  books~\cite{MetzDiekmann,BP,bansaye2015stochastic} have been written together with an extensive literature, mainly for linear but also for nonlinear~\cite{mischler2002stability,gwiazda2010nonlinear} cases, from a PDE or a stochastic point of view. Let us mention here only the expected asymptotic behaviour, given by the eigenvector(s) linked to the dominant eigenvalue(s), for the models cited in Section~\ref{sec: structured pop eq}.

\subsubsection*{The incremental/adder model}
The eigenvalue problem linked to the system~\eqref{eq: EDP adder simple}\eqref{eq: EDP adder simple bound} may be written as follows:
\begin{equation}\label{eigen:adder}
\left\{\begin{array}{l}
\lambda_k N_k + \f{\p}{\partial \i} (\kappa x N_k)+ \f{\p}{\partial x} \big(\kappa x N_k\big)  = - \kappa x B(\i)N_k(\i,x) ,
\\ \\
\kappa x N_k(0,x)= 4k \kappa x \int_0^\infty B(\i) N_k(\i,2x)d\i,
\end{array}\right.
\end{equation}
with $\lambda_1=0$ as usual, and $\lambda_2=\kappa$ as for the growth-fragmentation equation, due to the linear growth rate. Existence and uniqueness of a dominant positive eigenvalue and eigenvector has been recently studied in~\cite{gabriel2018steady}, for more general fragmentation kernels than just the diagonal kernel. Note also that in this specific case, we also have a countable set of dominant (not positive) eigenvalues, so that an equivalent of Theorem~\ref{theo:osci} may be obtained.
\subsubsection*{Generalisations}
Several types of generalisations have been studied, for instance with varying growth rates~\cite{DHKR}, or with a maturity variable added to the renewal equation~\cite{mischler2002stability}, size and age structured models~\cite{D,kang2020nonlinear} etc. Let us only write here the eigenvalue problem related to~\eqref{eq: EDP gene}\eqref{eq: EDP gene bound} - whose study once more lies beyond the scope of this chapter:
\begin{align}
\label{eq: EDP gene eigen}
&\lambda_k  N_k + \f{\p}{\partial \i} \left(\tau_\i(\i,x) N_k\right)+ \f{\p}{\partial x} \left(\tau_x(x,\kappa) N_k\right)   = - \beta(\i,x,\kappa)N_k(\i,x,\kappa),
\\
&\tau_x N_k(\i,0,\kappa)=0,\qquad N_k\geq 0, \qquad \iiint N_k dzdxd\kappa=1, \nonumber
\\
&\tau_\i(x,\kappa) N_k(0,x,\kappa)  =\label{eq: EDP gene bound eigen}
\\&\qquad \nonumber 2k \int_0^\infty \int_0^\infty \int_0^\infty \theta(\kappa',\kappa) b(y,x) \beta(\i,y,\kappa)  N_k(\i,y,\kappa')d\i dyd\kappa'. 
\end{align}
\section{Model calibration: statistical estimation of the division rate}
\label{sec:inverse}

In Section~\ref{subsec:datanal}, we have noticed that, contrarily to the growth rate or even to the division kernel, the division rate
cannot be inferred from direct measurements, even from individual dynamics data.  We thus face a typical inverse problem: How to estimate the division rate $B$ from data on a population, which follows - we assume - the dynamics given by one of the structured population model described above?

A first and major idea~\cite{PZ} consists in taking
 advantage of the asymptotic analysis carried out in Section~\ref{sec:anal}: we consider that at any time of the experiment, the population has already reached its steady asymptotic regime, {\it i.e.} for the observation scheme $k$ that the population is aligned along the dominant eigenvector $N_k$ of the model under consideration: age, size, increment, or more general model. This is well justified by the theoretical analysis above: the trend being exponentially fast under fairly general assumptions, the not-asymptotic regime concerns in most experimental cases only a negligible part of the data collected.

We recall (see Section~\ref{subsec:datanal}) that there are two types of datasets, each being related to a different inverse problem:
\begin{itemize}
\item {\bf Individual dynamics data collection}: following the trajectory of each individual allows us to measure the dividing and newborn cells. Intuitively, one feels that this allows a relatively direct estimation of the division rate. This type of data may concern genealogical  (through {\it e.g.} microfluidic device) as well as population (microcolony growth) observation. A difficulty in the population dynamics observation is the selection bias~\cite{hoffmann2016nonparametric}.

\item {\bf Population point data}: we can observe, at given timepoints, samples of some structuring variables such as size (or age, fluorescent label~\cite{Banks1} or whatsoever), which are then related to the empirical distribution~\eqref{eq: random measure}, itself related to $f_t(\cdot)=\f{n_k(t,\cdot)}{\int n_k (t,\cdot)d\cdot}$ thanks to a representation like in Proposition~\ref{prop: first equivalence}.
We may moreover assume an approximation of the form $f_t\approx N_k$ by the use of a time-asymptotic  result such as Theorems~\ref{theo:asymp:age} or~\ref{theo:asympCGY}. Alternatively, we can keep up with a stochastic approach, linking directly the stochastic measure to its limit through probabilistic results such as~\cite{DHKR,hoffmann2016nonparametric,bertoin2020strong}. We then address the inverse problem consisting in estimating $B$ from measurements of $N_k;$ one feels immediately that such an approach requires more analysis and, since the available information is less rich, that the inverse problem is more ill-posed.

\end{itemize}

Note that even in the case of individual dynamics collection, it may be more interesting to use the second approach: if the data are more numerous or less noisy, this may compensate the fact that the information they contain is poorer. In our example of {\it E. coli}, both approaches are possible, which allows to compare their accuracy in practice.

\subsection{Estimating an age-dependent division rate}
\label{subsec:estim:age}
As for the previous sections~\ref{sec:model} and~\ref{subsec:anal:age}, the age-structured model is somehow the simplest model along our line of models for which explicit computations can be conducted.
We review here the different types of inverse problems we have to solve, depending on the type of data available; we will find all the same problems for the other models.

\subsubsection{Individual dynamics data} \label{sec: indiv dyn data}

\paragraph{Individual dynamics data, stochastic viewpoint.}


Let us assume that we have data such as shown in Figures~\ref{fig:microcolony} or~\ref{fig:microflu}: at short time intervals, we observe the age of cells, so that
we observe
$$\big\{\zeta_u,\;\;u \in {\mathcal U}_k\big\},$$
for some subtree ${\cal U}_k \subset {\cal U}$, with $\cal U$ being defined in~\eqref{def:U}.
For the genealogical observation ($k=1$), we define
\begin{equation}
\label{def:U1}
{\cal U}_1=\left\{u_\ell \in \{0,1\}^\ell, \qquad u_{\ell +1}=(u_\ell,u^+), \qquad 0\leq \ell \leq n\, \qquad u^+\in\{0,1\}\right\},
\end{equation}
with $u^+\in \{0,1\}$ chosen uniformly at random, so that $u_{\ell +1}$ is offspring of $u_\ell$, and $n$ is a fixed number given by the experimentalist. In practice, we gather several such trees. For the population observation ($k=2$), the trees are defined by a final time $T>0$ fixed by the experimentalist, so that we observe
\begin{equation}
\label{def:U2}
{\cal U}_2=\left\{u \in {\cal U}, b_u+\zeta_u\leq T
\right\},
\end{equation}
where $b_u$ is the birth time of the cell: we observe all the lifetimes of cells which have divided before $T$. We see that the number of cells is stochastic for this second case, and there is a selection bias: we will observe more descendants of cells which have divided quickly, see Section~\ref{sec:anal}.

\paragraph{Individual dynamics data, stochastic viewpoint, genealogical observation.}
This case is  relatively straightfroward: as already observed, since
$$\PP(\zeta_u \in [a,a+da]\,|\zeta_u \geq a) = B(a)da,$$
 we obtain that the probability distribution of a lifetime $\zeta_u$ is given by
\begin{equation} \label{eq: distrib classical}
\PP(\zeta_u \in da) = f_1(a)da = B(a)\exp\Big(-\int_0^a B(s)ds\Big)da.
\end{equation}
In the case of a genealogical observation, the subtree is deterministic: there is no selection bias and the lifetime of each cell is independent from the others: we observe a sample of $n$ cells having divided at ages which are the
realizations of $\{\zeta_u, u \in \mathcal U_1\}$, as independent random variables with common density $f$. Moreover,  as soon as $\int^\infty B=\infty$, we have the survival analysis representation
\begin{equation}\label{eq:estim:age1}
B(a)=\f{f_1(a)}{\int_a^\infty f_1(s)ds}=\frac{f_1(a)}{S_1(a)},
\end{equation}
where $S_1$ is the survival function, as a simple inversion of the formula \eqref{eq: distrib classical} given above. In other application contexts, $B$ is called a hazard function. In this simple formula, we notice here three important facts, that will be found throughout our study:
\begin{itemize}
\item estimating $B$ has the same complexity as estimating the density $f_1$. This stems from the fact that the survival function $S_1(a)$ can be estimated at rate $\sqrt{n}$ by its empirical counterpart, hence only the numerator in the right-hand side of \eqref{eq:estim:age1} is a genuine nonparametric estimation problem.
\item The estimation of $B(a)$ becomes harder as $a$ increases, since $S(a)$ vanishes when $a$ tends to infinity.
\item We can also interpret our observations directly on the eigenvector equation $N_1:$ the proportion of dividing cells being $B(a)N_1(a),$ we find $B$ by writing simply
$$B(a)=\f{B(a)N_1(a)}{N_1(a)},$$
and using the equation again we find that $N_1(a)=Ce^{-\int_0^a B(s)ds}=CS_1(a),$ with $C>0$ a normalisation constant, so that we are back to~\eqref{eq:estim:age1}.
\end{itemize}

We make the first point rigorous by recalling a standard statistical estimation result. Let $K:[0,\infty)\rightarrow \R$ denote a well-located kernel of order $\ell \geq 0$, namely
\begin{equation}
\label{as:K}
 K \in {\cal C}_c^0(\R), \qquad \int_0^\infty a^k K(a)da = \1_{\{k = 0\}}\;\;\text{for}\;\;k=0,\ldots, \ell.
\end{equation}
The existence of such an oscillating kernel for arbitrary $\ell$ is standard, see {\it e.g.} the textbook \cite{tsybakov2003introduction}.
For $h>0$ the bandwidth, define $K_h(a) = h^{-1}K(h^{-1}a)$ and
\begin{equation}\label{def:Bestim:age1}
\widehat B_{n,h}(a) = \frac{\sum_{u \in \mathcal U_1}K_h(a-\zeta_u)}{\sum_{u \in \mathcal U_1}\1_{\{\zeta_u \geq a\}}},
\end{equation}
(and set $0$ if none of the $\zeta_u$ are above $a$.)
The {\it bias} of $f_1$ at $a$ relative to the approximation kernel $K$ is defined as
$$\mathfrak{b}_h(f_1)(a) = \big|\int_{[0,\infty)}f_1(a')K_h(a-a')da'-f_1(a)\big|.$$
\begin{proposition}\label{prop:estim:age:indiv:gene} We have
\begin{align}
&\E\big[\big||\mathcal U_1|^{-1}\sum_{u \in \mathcal U_1}K_h(a-\zeta_u)-f_1(a)\big|^2\big] \nonumber \\
&\leq \mathfrak{b}_h(f_1)(a)^2+ (nh)^{-1}\sup_{a-a' \in \mathrm{Supp}(K)}B(a')\int_{[0,\infty)}K(a')^2da'  \label{eq: esti kernel}
\end{align}
and
\begin{equation} \label{eq: survival}
\E\big[\big||\mathcal U_1|^{-1}\sum_{u \in \mathcal U_1}\1_{\{\zeta_u \geq a\}}-S_1(a)\big|^2\big] \leq \tfrac{1}{4}n^{-1}.
\end{equation}
\end{proposition}
\begin{proof}
The first part is obtained by noticing that $\int_{[0,\infty)}f_1(a')K_h(a-a')da' = \E[K_h(a-\zeta_u)]$, and using that the variables $K_h(a-\zeta_u)-\E[K_h(a-\zeta_u)]$ are independent and identically distributed, with common variance bounded above by $h^{-1}\int_{[0,\infty)} K_h(a-a')^2f_1(a')da' \leq \sup_{a-a' \in \mathrm{Supp}(K)}B(a')\int_{[0,\infty)}K(a')^2da'$. The result is simply a combination of this observation and the act that the variance of the sum of independent random variables is the sum of its variances. The second part easily follows, noticing now that $\1_{\{\zeta_u \geq a\}}$ is a Bernoulli randiom variable with expectation  $S_1(a)$ and variance (always) bounded by $1/4$.
\end{proof}

Assuming that $B$ has smoothness of order $s>0$ around $a$, in a H\"older sense for instance, we then have $\mathfrak{b}_h(f_1)(a) \lesssim h^s$ as soon as $\ell \geq s-1$, and therefore the estimator $|\mathcal U_1|^{-1}\sum_{u \in \mathcal U_1}K_h(a-\zeta_u)$ has pointwise squared risk of order $h^{2s}+(nh)^{-1}$ in the following sense:
$$\big(\inf_{h>0}(h^{s}+(nh)^{-1/2})\big)^{-1}\big(\widehat B_{n,h}(a)-B(a)\big)$$
is bounded in probability as $n \rightarrow \infty$.\\

Finding the optimal bandwidth $h$ leads to the classical rate $\inf_{h>0}(h^{s}+(nh)^{-1/2}) \approx n^{-s/(2s+1)}$ in nonparametric estimation, which is always a slowlier rate of convergence than $n^{-1/2}$. Combining the two estimates \eqref{eq: esti kernel} ans \eqref{eq: survival} for an optimal bandwidth, we see that $\widehat B_{n,h}(a)$ estimates $B(a)$ with optimal (normalised) order $n^{-s/(2s+1)}$. These are classical results in nonparametric estimation, see {\it e.g.} \cite{tsybakov2003introduction,gine2021mathematical} and the references therein, in particular regarding data driven choices of $h$, since the smoothness $s>0$ is only a mathematical construct that has no real meaning in practice.

\paragraph{Individual dynamics data, stochastic viewpoint, population observation.}

Following in spirit Section \ref{sec: renewal} but now with data extracted from $\mathcal U_2$, we first look for the behaviour of empirical sums of the form
$$\mathcal E^T(g,\mathcal U_2) = \frac{1}{|\mathcal U_2|}\sum_{u \in \mathcal U_2} g(\zeta_u),$$
for nice (say bounded) test functions $g:[0,\infty)\rightarrow \R$.
As in Section \ref{sec: renewal}, we also have a many-to-one formula that now reads
\begin{equation} \label{many-to-one cloez}
\E\Big[\sum_{u \in \mathcal U_2}g(\zeta_u^T)\Big]  = \E\Big[\sum_{u \in \mathcal U_2}g(\zeta_u)\Big] = 2^{-1}\int_0^T e^{\lambda_2 s}\E\big[g(\widetilde \chi_s)H_B(\widetilde \chi_s)\big]ds,
\end{equation}
where $(\widetilde \chi_t)_{t \geq 0}$ is the auxiliary one-dimensional auxiliary Markov process
with generator $\mathcal{A}_{H_B}$, see \eqref{def generator}, where $H_B$ is characterised by \eqref{characterisation H_B} above. Assuming again ergodicity, we approximate the right-hand side of \eqref{many-to-one cloez} and obtain (at least heuristically)
\begin{align*}
\E\Big[\sum_{u \in \mathcal U_2}g(\zeta_u)\Big]  & \sim c_B2^{-1} \frac{e^{\lambda_2T}}{\lambda_2} \int_0^\infty g(x)H_B(x)e^{-\int_0^x H_B(u)du}dx \\
 & = c_B \frac{e^{\lambda_2T}}{\lambda_2} \int_0^\infty g(x)e^{-\lambda_2 x}f_{1}(x)dx. \nonumber
\end{align*}
since $H_B(x)\exp(-\int_0^x H_B(y)dy)=2e^{-\lambda_2x}f_1(x)$ by \eqref{characterisation H_B}.
We again have an approximation of the type  $\E[|{\mathcal U}_2|] \sim \kappa_{B}e^{\lambda_2T}$  with another constant $\kappa_B'$
and we eventually expect
\begin{equation*}
\mathcal E^T(g,\mathcal U_2) \sim \mathring{\mathcal E}\big(g\big)  := \frac{c_B}{\lambda_2\kappa_B'}\int_0^\infty g(x)e^{-\lambda_2 x}f_1(x)dx =  2\int_0^\infty g(x)e^{-\lambda_2 x}f_1(x)dx
\end{equation*}
as $T \rightarrow \infty$, where the last equality stems from the identity $c_B=2\lambda_2\kappa'_B$ that can be readily derived by picking $g=1$ and using \eqref{characterisation H_B} together with the fact that $f_2$ is a density function.\\

\begin{proposition}{\bf{(Rate of convergence for particles living at time $T$ - Theorem 4 in \cite{hoffmann2016nonparametric})}} \label{rate avant T}
Assume $\lambda_2 \leq 2\inf_x H_B(x)$,  and $B$  differentiable satisfying $B'(x) \leq B(x)^2$ and $0 < c \leq B(x) \leq 2c$ for every $x \geq 0$ and some $c>0$. Then
$$e^{\lambda_2T/2}\big(\mathcal E^T\big(g\big) - \mathring{\mathcal E}(g)\big)$$
is asymptotically bounded in probability.
\end{proposition}

\paragraph{Estimation: Step 1). Reconstruction formula for $B(a)$.}
We have
\begin{equation} \label{eq: rep fund}
B(a) = \frac{f_1(a)}{1-\int_0^a f_1(y)dy} = \f{2^{-1}f_2(a) e^{\lambda_2 a}}{1-2^{-1}\int_0^a f_2(y)e^{\lambda_2 y}dy}
\end{equation}
and from the definition
$\mathring{\mathcal E}\big(g\big) = 2\int_0^\infty g(x) e^{-\lambda_2x}f_1(x)dx=$ we obtain the formal reconstruction formula
\begin{equation} \label{formal reconstruction}
B(a) =  \frac{\mathring{\mathcal E}\big(2^{-1}e^{\lambda_2\cdot}\delta_a(\cdot)\big)}{1- \mathring{\mathcal E}\big(2^{-1}e^{\lambda_2\cdot}{\bf 1}_{\{\cdot \leq a\}}\big)}
\end{equation}
where $\delta_a(\cdot)$ denotes the Dirac function at $x$. Therefore, 
taking $g$ as a weak approximation of  $\delta_a$ via a kernel, we obtain a strategy for estimating $B(a)$ replacing $\mathring{\mathcal E}(\cdot)$ by its empirical version ${\mathcal E}^T(\mathcal U_2, \cdot)$.\\

\paragraph{Estimation: Step 2). Construction of a kernel estimator and function spaces}.
Let $K: [0,\infty) \rightarrow \R$ be a kernel function. For $h>0$, set $K_h(x)=h^{-1}K(h^{-1}x)$. In view of \eqref{formal reconstruction}, we define the estimator
\begin{align}\label{def:estimB:age1}
\widehat B_{T,h}(a) & = \frac{{\mathcal E}^T\big(\mathcal U_2, 2^{-1}e^{\lambda_2 \cdot }K_h(a-\cdot)\big)}
{1-{\mathcal E}^T\big(\mathcal U_2, 2^{-1}e^{\lambda_2 \cdot} {\bf 1}_{\{\cdot \leq a\}}\big)}
\end{align}
on the set ${\mathcal E}^T\big(\mathcal U_2, 2^{-1}e^{\lambda_2 \cdot} {\bf 1}_{\{\cdot \leq a\}}\big) \neq 1$ and $0$ otherwise. Thus $\widehat B_{T,h}(a)$ is specified by the choice of the kernel  $K$ and the bandwidth $h>0$.\\



%

\noindent {\it Performances of the Estimator.} We are ready to give the rate of convergence of $\widehat B_T(a)$ for $a$ restricted to a compact interval $\mathcal D$, uniformly over H\" older balls ${\mathcal H}^s_{\mathcal D}$

\begin{proposition}{\bf{(Upper rate of convergence, Theorem 7 in \cite{hoffmann2016nonparametric})}} \label{thm: upper rate}
In the same setting as in Proposition \ref{rate avant T}, specify $\widehat B_{T,h}$ with a kernel satisfying~\eqref{as:K} for some $\ell >1$ and
$$
h=\widehat h_T = \exp\big(-\tfrac{1}{2s+1}\lambda_2 T\big)
$$
for some $s \in (1,\ell+1)$. Then
$$e^{\lambda_2\frac{s}{2s+1}T}\big(\widehat B_{T,h}(a)-B(a)\big)$$
is asymptotically bounded in probability if $B$ is $s$-H\"older in a neigbourhood of $a$.
\end{proposition}
This rate is indeed optimal in a minimax sense, see Theorem 8 in \cite{hoffmann2016nonparametric}, where the problem of estimating $\lambda_2$ is also considered. The proof of Proposition \ref{thm: upper rate} is detailed in \cite{hoffmann2016nonparametric}. In a more condensed way, they can also be found in \cite{hoffmann2018statistical}.

\paragraph{Individual dynamics data, deterministic viewpoint.}

In the field of "deterministic" inverse problems, we model the noise by assuming that we observe data in a certain metric space up to an error $\ep$ according to this metric. In our case, this means that we first assume  that the population has reached its steady asymptotic behaviour given by Theorem~\ref{theo:asymp:age}, and second that there exists a (known) noise level $\ep>0,$ and that we observe the distribution of ages of dividing cells, defined by
$$f_k(a):=\f{B(a)N_k(a)}{\int B(a)N_k(a)da},$$
up to a noise, {\it i.e.} the measurement $H_k^\ep(a)$ is such that
$$\Vert f_k^\ep - f_k \Vert_{W^{-s,p} ([0,\infty))} \leq \ep,$$
with $s\geq 0,$ $1\leq p\leq \infty$ and $W^{-s,p}([0,\infty))$ the corresponding Sobolev space.
Integrating~\eqref{eq:age:eigen} to express $N_k(a)$ in terms of $B(a)N_k(a)$ and $\lambda_k$, chosen with $k=1$ or $k=2$ according to the observation scheme considered, we find the formula
\begin{equation}\label{def:B:direct}
B(a)=\f{B(a)N_k(a)}{N_k (a)}=\f{B(a)N_k(a)}{e^{-\lambda_k a} \int_a^\infty B(s)N_k(s)e^{\lambda_k s}ds}=\f{f_k(a)}{e^{-\lambda_k a} \int_a^\infty f_k(s)e^{\lambda_k s} ds},
\end{equation}
where we recognize~\eqref{eq:estim:age1} if $k=1$ and~\eqref{formal reconstruction} if $k=2.$ This naturally leads us to define an estimate $B_\ep$ by replacing in this formula $f_k$ by $f_k^\ep,$ and add a threshold condition for the denominator, as done above by considering compact intervals.  If $s=0,$ {\it i.e.} if the noise lies in $L^p([0,\infty)),$ we do not need to regularize this estimate: the problem is well-posed, and $B_\ep$ provides directly an estimate for $B$ in a space $L^p([0,\infty))$ weighted by $N_k.$ If either $s<0$ or we want an estimate for $B$ in some $W^{m,p}$ space with $m>0,$ then a regularization is needed: in exactly the same spirit as for kernel density estimation, we can define
\begin{equation}\label{def:fkeph}f_{k}^{\ep,h}=K_h * f_k^\ep,
\end{equation}
and we have the following result, deterministic version of the above Propositions~\ref{prop:estim:age:indiv:gene} and~\ref{rate avant T}. \begin{proposition}\label{prop:estim:age:direct:det}Let $K$ a kernel satisfying~\eqref{as:K}, $K\in C^1_b(\R),$ and $K_h(\cdot)=1/h K(\cdot/h).$ Let $1\leq p\leq \infty$ and $\theta\in [0,1].$ We have the estimate
$$\Vert f_k^{\ep,h} - f_k \Vert_{L^p}
\leq \Vert K_h * f_k - f_k\Vert_{L^p} + C(K) h^{-\theta} \Vert f_k - f_k^{\ep} \Vert_{W^{-\theta,p}},
$$
where $C(k)$ is a constant depending only on the kernel $K$ and on its derivative.
\end{proposition}
We do not specify here the standard machinery to obtain an estimate for $B$ from the estimate for $f_k:$ it consists in dividing $f_k^{\ep,h}$ by $N_k^\ep$ - we do not need any regularisation for the denominator, thanks to the integral and to the choice $\theta\leq 1$ - and then thresholding. See for instance~\cite{DHRR}.

We then find that, for a noise $\ep$ in the space $W^{-\theta,p},$ {\it i.e.} if we have
$$ \Vert f_k - f_k^{\ep} \Vert_{W^{-\theta,p}}\leq \ep,$$
and if we assume $f\in W^{s,p}$ with $\ell \geq s-1,$ the optimal estimate is in the order of $\ep^{s/(s+\theta)},$ and achieved for $h\approx \ep^{s/(s+\theta)}.$ We first notice that if $\theta=0$ (noise in $L^2$), this speed is of order $\ep:$ we do not need any regularisation, and we face a well-posed inverse problem!

We notice that this result is fully coherent with the statistical estimates of Propositions~\ref{prop:estim:age:indiv:gene} and~\ref{rate avant T}: to see it, the correct heuristics consists in  taking $\theta=\f{1}{2}$ (for a heuristics of the regularity of the empirical measure) and $\ep= n^{-1/2}$ (the noise level being given by a central limit theorem), see~\cite{NP} for an illuminating explanation of this comparison.  We then have an estimate in the order of
$$\ep^{s/(s+\theta)}=\ep^{s/(s+1/2)}=n^{-s/(2s+1)},$$
as above. We further develop these heuristics or comparison between stochastic and deterministic noise in Section \ref{subsubsec:estim:size:pop} below.\\

Finally, we note that the proof of Proposition~\ref{prop:estim:age:direct:det} is not more involved for $k=2$ than for $k=1,$ contrarily to the stochastic setting where the selection bias and the dependence between the individuals in the population case make it much more complex.
\subsubsection{Population point data}
\label{subsubsec:estim:age:pop}
Let us imagine that we are given a noisy measurement of the distribution of cells $N_k(a).$ This noise may be modeled by three different settings, increasingly realistic:
\begin{enumerate}
\item deterministic noise model: we model the noise by a measurement $N_k^\ep$ such that
$$\Vert N_k^\ep -N_k\Vert_{W^{-s,p}([0,\infty))} \leq \ep.$$
\item Stochastic sampling noise: we assume that we observe a sample of ages $a_1,\cdots, a_n$ realisations of $A_1,\cdots,A_n$ {\it i.i.d.} random variables of density $N_k.$ We could refine this setting by adding a measurement noise to each $a_i.$
\item Stochastic process: we observe, at a given time $T,$ a sample of ages of cells. This means that we observe $\{a_u\},\;u\in {\cal U}_3$  defined by
$$a_u=T-b_u,\qquad u\in {\cal U}_3=\left\{ u\in {\cal U}, \quad b_u \leq T <b_u+\zeta_u\right\}.$$
\end{enumerate}
For the third model, one first needs to establish asymptotic results as done in~\cite{hoffmann2016nonparametric} given in Proposition \ref{rate en T} above for individual dynamics data.
Having
$$\mathcal E(g) = 2\lambda_2\int_0^\infty g(x)e^{-\lambda_2 x}\exp\big(-\int_0^x B(y)dy\big)dx $$
and ignoring the fact that $\lambda_2$ is unknown, we can anticipate that by picking a suitable test function $g$ as a kernel, the information about $B(x)$ can only be inferred through $\exp(-\int_0^x B(y)dy)$. More precisely, consider the quantity
$$\widehat f_{h,T}(x) = - \mathcal E^T\Big(\lambda_2\big(K_{h}\big)'(x-\cdot)\Big)$$
for a kernel satisfying~\eqref{as:K}.
By Proposition~\ref{rate en T} and integrating by part, we readily see that
\begin{equation} \label{conv ill posed}
\widehat f_{h,T} \mapsto -\mathcal E\Big(\frac{1}{\lambda_2}\big(K_{h}\big)'(x-\cdot)\Big)=\int_0^\infty K_h(x-y)f_{B+\lambda_2}(y)dy
\end{equation}
in probability as $T \rightarrow \infty$, where $f_{B+\lambda_2}$ is the density associated to the division rate $B(x)+\lambda_2$. On the one hand, using following line by line the proofs of \cite{hoffmann2016nonparametric} it is not difficult to show that the rate of convergence in \eqref{conv ill posed} is of order $h^{-3/2}e^{\lambda_2 T/2}$ since we take the derivative of the kernel $K_h$.
 On the other hand, the limit  $\int_0^\infty K_h(x-y)f_{B+\lambda_2}(y)dy$ approximates $f_{B+\lambda_2}(x)$ with an error of order $h^s$ if $B$ is $s-$H\"older. Balancing the two error terms in $h$, we see that we can estimate $f_{B+\lambda_2}(x)$ with an error of (presumably optimal) order $\exp(-\lambda_2\tfrac{s}{2s+3}T)$. Due to the fact that the denominator in representation \eqref{eq: rep fund} can be estimated with parametric error rate $\exp(-\lambda_2 T/2)$  (possibly up to polynomially slow terms in $T$), we end up with the rate of estimation $\exp(-\lambda_2\frac{s}{2s+3}T)$ for $B(x)$ as well, and that can be related to an ill-posed problem of order 1 (see for instance~\cite{tsybakov2003introduction}). This phenomenon, namely the structure of an ill-posed problem of order 1 in restriction to data alive at time $T$, appears in the other settings: for the estimation of a size-division rate from living cells at a given large time in \cite{DPZ,DHRR} or for the estimation of the dislocation measure for a homogeneous fragmentation, see \cite{hoffmann2011statistical}.\\
%

For the first and second noise models,  we use the explicit formula~\eqref{def:eigen:age} to get
\begin{equation}\label{def:B:inverse}
B(a)=-\lambda_k - \f{\p_a N_k (a)}{N_k(a)}
\end{equation}
which is equivalent to~\eqref{conv ill posed}.
As for individual dynamics data, we see that we need to divide by the density, so that we will not be able to estimate $B$ at places where it vanishes. The new fact is that, contrarily to Formula~\eqref{def:B:direct}, the formula depends on the age-derivative of $N_k,$ so that, as shown below in Proposition~\ref{prop:estim:age:inv:det},
the so-called {\it degree of ill-posedness} of the inverse problem is the one of estimating a function from its derivative: as for the third model, more regularisation is needed. We obtain the two following propositions.

\begin{proposition}{\bf{(Deterministic noise)}}\label{prop:estim:age:inv:det}
Under the assumptions of Proposition~\ref{prop:estim:age:direct:det}, defining
$$H_{\ep,h}(a):=-\lambda_k^\ep N_k^\ep (a)- (K_h * \p_a N_k^\ep )(a),\qquad H(a)=B(a)N_k(a)$$
 we have
 $$\begin{array}{c}\Vert H_{\ep} - H\Vert_{L^p} \leq
 C(K,N_k)\left(\Vert K_h*N_k -N_k\Vert_{L^p} \right.
 \\ \\ \left. +\vert \lambda_k^\ep -\lambda_k\vert +
 h^{-\theta -1} \Vert N_k^\ep -N_k\Vert_{W^{-\theta,p}}
 \right)
 \end{array}$$
 with $C(K,N_k)$ depending only on $K,$ $K'$ and $N_k.$
\end{proposition}
The proof is let to the reader; it is exactly the same ingredients as before, namely standard convolution inequalities. We see that for $N_k\in W^{s,p}$ and $\ell \geq s-1,$ and an error $\ep$ in $W^{-\theta,p},$ we have an optimal estimate in the order of $\ep^{\f{s}{s+\theta+1}},$ corresponding, for $\theta=0,$ to an inverse problem of degree of ill-posedness $1.$

\begin{proposition}{\bf{(Stochastic sampling noise)}} Under the assumptions of Proposition~\ref{prop:estim:age:inv:det}, assume that we know $\lambda_k$ from previous observations, and that we observe an {\it i.i.d} sample $a_1,\cdots a_n$ of law $N_k,$ and define the empirical measure
$$N_{n} (da)=\f{1}{n}\sum_{i=1}^n \delta_{a_i}(da)$$
and its regularisation
$$N_{n,h} (a)da=K_h * \biggl(\f{1}{n}\sum\limits_{i=1}^n \delta_{a_i}(da)\biggr)=\f{1}{n}\sum\limits_{i=1}^n K_h(a-a_i)da$$
We define
$$H_{n,h}(a):=-\lambda_k - N_{n,h} (a) -\p_a N_{n,h}(a),\qquad H(a):=B(a)N_k(a),$$
and the bias
$$\mathfrak{b}_h(N_k)(a) = \big|(K_h*N_k)(a)-N_k(a)\big|.$$
We have the following estimate
\begin{align}
\E\big[\big|H_{n,h}(a)-H(a)\big|^2\big]
\leq \mathfrak{b}_h(N_k)(a)^2+ C(N_k,K)\f{1}{nh^3},  \label{eq: esti kernel inv}
\end{align}
where $C(B,K)$ depends only on $N_k$, $K$ and $K'$.
\label{prop:estim:age:stat:det}
\end{proposition}
The proof is close to the proof of Proposition~ \ref{prop:estim:age:indiv:gene}. The inflation in the variance term from the order $(nh)^{-1}$ to $(nh^3)^{-1}$ comes from the fact that we take a derivative $\p_a N_{n,h}(a)$ of the kernel estimator $N_{n,h} (a)$.
As previously seen, the bias is of order $h^s$ if $N_k \in W^{s,p}$ and $\ell\geq s-1;$ this leads to an optimal error ({\it i.e.} the square root of the left-hand side in~\eqref{eq: esti kernel inv}) in the order of $n^{-\f{s}{2s+3}}.$ Taking in Proposition~\ref{prop:estim:age:inv:det} $\theta=1/2$ and $\ep=n^{-1/2},$ we have $\ep^{\f{s}{s+3/2}}=n^{-\f{s}{2s+3}}:$ here again, this is the same optimal speed of convergence.

All these orders of magnitude for the convergence rates remain true for the size-structured model studied below.

\subsection{Estimating a size-dependent division rate}
\label{subsec:estim:size}
We now review the methods and results developed to estimate a size-dependent division rate, as built in Section~\ref{sec:model:PDE} and analysed in Section~\ref{sec:anal:growthfrag}. We follow the same notations as above. We do not treat here the interesting question of estimating the fragmentation kernel $b(y,x)$, and refer for instance to~\cite{doumic:hal-01501811,tournus2021insights,hoang2020nonparametric,hoffmann2011statistical}.

\subsubsection{Individual dynamics data}

Combining stochastic, deterministic and asymptotic approaches allows us to  obtain easily reconstruction formulae.\\

Let us first revisit heuristically the formulae of~\cite{DHKR,robert:hal-00981312}. In general terms, we observe
$$\big\{(\xi_u,\chi_u,\zeta_u),\;\;u \in {\mathcal U}_k\big\},\qquad k=1{\text{ or }}k=2,\qquad{\text{or}}\qquad \{\xi_u^T,\; u \in {\cal V}_T\},$$
with $(\xi_u,\chi_u,\zeta_u)$ respectively the size at birth, size at division and lifetime of the individual $u$ taken in the sample ${\cal U}_k$ defined by~\eqref{def:U1} and~\eqref{def:U2}, $\cal V_T$ defined by~\eqref{def:VT} and $\xi_u^T$ the size of the individual $u \in {\cal V}_T$ alive at time $T.$ ${\cal U}_1$ models the genealogical observation and individual dynamics data, ${\cal U}_2$ the population observation and individual dynamics data, and ${\cal V}_T$ population point data and population observation.

Assuming that the asymptotic behaviour of Section~\ref{sec:anal:growthfrag} has been reached in each of these models, we can interpret the densities with the help of Equation~\eqref{eq:eigenproblem}.
\begin{itemize}
\item $\xi_u,$ $u\in {\cal U}_k,$ has a density distribution given by
$$f_k^b (x)=\f{k\int_x^\infty B(y)\tau(y)b(y,x)N_k(y)dy}{k\int_0^\infty \int_s^\infty B(y)\tau(y)b(y,s)N_k(y)dyds} = \f{\int_x^\infty B(y)\tau(y)b(y,x)N_k(y)dy}{\int_0^\infty  B(y)\tau(y)N_k(y)dy}.$$
This formula is obtained by identifying the last term in~\eqref{eq:eigenproblem} with the newborn proportion, and normalise it to obtain a density.
\item $\chi_u,$ $u\in {\cal U}_k,$ has density distribution given by
$$f_k^d (x)=\f{ B(x)\tau(x)N_k(x)}{\int_0^\infty  B(y)\tau(y)N_k(y)dy}.$$
\item $\xi_u^T$  has for density $N_k.$
\end{itemize}
To estimate $B,$ we can then write
$$B(x)=\f{f_k^d (x)}{\tau (x)N_k(x)}{\int_0^\infty  B(y)\tau(y)N_k(y)dy},$$
and it remains to find formulae for $\tau N_k$ and for ${\int_0^\infty  B(y)\tau(y)N_k(y)dy}$.\\

\noindent {\bf Case $k=1$ (genealogical observation).}
We denote $C=\int_0^\infty  B(y)\tau(y)N_1(y)dy$ and write~\eqref{eq:eigenproblem} as
$$ \p_x (\tau N_1) +  C f_1^d = C  f_1^b \implies \f{\tau(x)N_1(x)}{C}=\int_x^\infty \left(f_1^d(y)-f_1^b(y)\right)dy,$$
so that we obtain the reconstruction formula
\begin{equation}
\label{estim:beta:size:1}
B(x)=\f{f^d_1(x)}{\int_x^\infty \left(f_1^d(y)-f_1^b(y)\right)dy}.\end{equation}

{\bf Case $k=2$ (population observation).}
We then have ${\int_0^\infty  B(y)\tau(y)N_2(y)dy}=\lambda_2,$ and ~\eqref{eq:eigenproblem} may be written as
$$\lambda_2 N_2 + \p_x (\tau N_2) + \lambda_2 f_2^d = 2 \lambda_2 f_2^b,$$
from which we deduce
$$\f{\tau(x)N_2(x)}{\lambda_2}=\int_x^\infty \left(f_2^d(y) -  2f_2^b(y)\right)e^{\lambda_2 \int_x^y \f{ds}{\tau(s)}} dy,$$
leading finally to the reconstruction formula
\begin{equation}
\label{estim:beta:size:2}
B(x)=\f{f_2^d(x)}{\int_x^\infty \left(f_2^d(y) - 2 f_2^b(y)\right)e^{\lambda_2 \int_x^y \f{ds}{\tau(s)}} dy}.
\end{equation}
Using either~\eqref{estim:beta:size:1} or~\eqref{estim:beta:size:2} and replacing
 $f_k^d$ and $f_k^b$ by the empirical distribution obtained from samples $(\xi_i^b,\xi_i^d),$ we can estimate $B$ without any knowledge on the fragmentation kernel $b,$ , as soon as both newborn and dividing cells distributions are observed. For genealogical data ($k=1$), even the growth rate $\tau$ may be unknown. This remark may be generalised to other models, see for instance the formula~(14) of~\cite{DHKR} for a model with $k=1$, the mitosis kernel (for which we have the simplification $f_k^b(x)=2 f_k^d(2x)$) and variable growth rates.
%

\paragraph{Individual dynamics data, stochastic approach, genealogical observation}

In the cases where we model the noise either deterministically or through an {\it i.i.d} sample,
the reconstruction formulae~\eqref{estim:beta:size:1} and~\eqref{estim:beta:size:2} immediately yield estimation results similar to the ones of Propositions~\ref{prop:estim:age:direct:det} and~\ref{prop:estim:age:indiv:gene} stated for the age model. More involved is the case where we do not depart from~\eqref{eq:eigenproblem} but from the stochastic model; it has been studied in~\cite{DHKR} for the case of genealogical observation, easier to study than the population observation case. To our best knowledge, the solution for population observation, with all the difficulties we already mentioned for the age problem (selection bias, censoring, non-ancillarity) plus a more intricate model, remains open.\\

In this setting, we look for a nonparametric estimator of $x \mapsto B(x)$ using the observation scheme $\mathcal U_1$ defined in Section \ref{sec: indiv dyn data}
We thus observe
$$(\xi_u)_{u \in \mathcal U_{1}}\;\;\text{and}\;\;(\zeta_u)_{u \in \mathcal U_{1}}.$$
We have (see Section~\ref{sec:model:branching})
$$
\begin{array}{ll}
\PP(\chi_u \in (x,x+dx) \vert \chi_u\geq x)=B(x)dx=\tau(x)B(x) dt,
\\ \\
 \PP(\chi_u\geq x \vert \xi_{u})=\1_{\{x\geq \xi_{u}\}} \exp\left(-\int_{\xi_{u}}^x B(y)dy\right).
 \end{array}
$$
Hence we infer, taking the equal mitosis kernel so that $2\xi_u=\chi_{u^-}$,
$$\mathbb P\big(\xi_u\in (x',x'+dx')\big|\,\xi_{u^-}=x\big) \\
=2B(2x'){\1}_{\{2x' \geq x\}}\exp\big(-\int_{x}^{2x'} B(y)dy\big)dx'.
$$
We thus obtain a simple and explicit representation for the transition kernel ${\mathcal P}_B\big(x, dx')= {\mathcal P}_B\big(x, x')dx'$ as
$${\mathcal P}_B\big(x, x')
=   2B(2x'){\1}_{\{2x' \geq x\}}\exp\big(-\int_{x}^{2x'} B(y)dy\big).$$
Under appropriate conditions on $B$ set out in details below, there exists a unique invariant probability
$\nu_B(dx) = \nu_B(x)dx$ on $[0,\infty)$ such that the following contraction property holds
\begin{equation} \label{eq:contraction}
\sup_{|g| \leq V}\big|{\mathcal P}^k_Bg(x)-\int_0^\infty g(z)\nu_B(z)dz\big| \leq RV(x)\gamma^k
\end{equation}
(where, for an integer $k\geq 1$, we set ${\mathcal P}_B^k={\mathcal P}_B^{k-1}\circ {\mathcal P}_B$)
for an appropriate Lyapunov function  $V$ and some (explicitly computable) $\gamma < 1$.  The proof of \eqref{eq:contraction} goes along a classical scheme and is detailed in Proposition 4 of \cite{hoffmann2011statistical}, and for $\tau(x)=\kappa x$ \eqref{eq:contraction} holds with
$$V(x)=\exp\big(\tfrac{m}{\kappa\mu}x^\mu\big)$$
for $\mu>0$.
Expand further the equation $\nu_B\mathcal P_B=\nu_B$:
\begin{align*}
& \nu_B(y)
= \int_0^\infty \nu_B(x){\mathcal P}_B\big(x,y\big)dx \\
  & = 2 B(2y) \int_{0}^{2y}\nu_B(x) \exp\big(-\int_{x}^{2y}{B(y')}dy'\big)dx \\
  & = 2 B(2y)\int_0^\infty \int_0^\infty {\1}_{\displaystyle \{x\leq 2y, y' \geq y\}}\nu_B(x)\,{\mathcal P}_B\big(x,y'\big)dy'dx.
\end{align*}
This yields the key representation
$$
\nu_B(y) =  2 B(2y)\mathbb P_{\nu_B}\big(\xi_{u^-} \leq 2y,\;\xi_u \geq y\big).
$$
We conclude
\begin{equation} \label{eq:representationB}
B(y)=\f{1}{2}\frac{\nu_B(y/2)}{\mathbb P_{\nu_B}\big(\xi_u^- \leq y, \xi_u \geq y/2\big)}
\end{equation}
and this yields the estimator
\begin{align*}
\widehat B_{n,h}(y) &
 = \f{1}{2}\,\frac{n^{-1}\sum_{u\in {\mathcal U}_{[n]}}K_{h}(\xi_u-y/2)}{n^{-1}\sum_{u \in {\mathcal U}_{[n]}} {\bf 1}_{\displaystyle \{\xi_{u^-}\leq y, \xi_u \geq y/2\}} \bigvee \varpi_n},
\label{def estimator}
\end{align*}
where the kernel $K_{h}(y)=h^{-1}K\big(h^{-1}y\big)$ is specified with an appropriate bandwidth (and technical threshold $\varpi_n >0$).

We assess the quality of $\widehat B_n$ in squared-loss error over compact intervals ${\mathcal D}$. We need to specify local smoothness properties of $B$ over ${\mathcal D}$, together with general properties that ensure that  the empirical measurements converge with an appropriate speed of convergence. This amounts to impose an appropriate behaviour of $B$ near the origin and infinity.\\
For $\alpha > 0$
and  positive constants
$r,m, \ell, L$,
introduce continuous functions
$B:[0,\infty)\rightarrow [0,\infty)$  such that
\begin{equation} \label{loc control}
\int_{0}^{r/2}x^{-1}B(2x)dx \leq L,\;\;\;\int_{r/2}^{r}x^{-1}B(2x)dx \geq \ell,\qquad
B(x) \geq m\, x^\alpha\;\;\text{for}\;\;\;x\geq r.
\end{equation}
Define
$$\delta:=\frac{1}{1-2^{-\alpha}} \exp\big(- (1-2^{-\alpha})\tfrac{m}{\kappa\alpha} r^\alpha\big)
.$$
Let $\gamma$ denote the spectral radius of the operator ${\mathcal P}_B - 1 \otimes \nu_B$ acting on the Banach space of functions $g: [0,\infty)\rightarrow \R$ such that
$$\sup\{|g(x)|/V(x), x\geq 0\}<\infty.$$
\begin{assumption}\label{the full tree assumption}
We have
$\delta<\tfrac{1}{2}$
and
$\gamma< 1$.
\end{assumption}
It is possible to obtain bounds on $r,\,m,\,\ell,\,L$ so that Assumption \ref{the full tree assumption} holds, by using explicit bounds on $\gamma$ following Hairer and Mattingly \cite{hairer2011yet}, see also Baxendale \cite{baxendale2005renewal}.
We are ready to state the performance of the estimator.


\begin{theorem}{\bf{(Adapted from Theorem~2 from~\cite{DHKR})}} \label{upper bound}
Specify $\widehat B_{n,h}$ with a kernel $K$ satisfying Assumption \ref{as:K} for some $n_0>0$ and
$$h=n^{-1/(2s+1)},\;\; \varpi_n \rightarrow 0.$$
For every  compact interval ${\mathcal D}\subset (0,\infty)$ such that $\inf {\mathcal D} \geq r/2$, there exists a choice of  $m,$ $\ell$, $L$, $\alpha$ such that Assumption~\ref{the full tree assumption} is satisfied. For $B$ $s$-H\"older satisfying~\eqref{loc control},
we have
$$
\E_{\mu}\big[\|\widehat B_n-B\|_{L^2({\mathcal D})}^2\big]^{1/2} \lesssim \varpi_n^{-1}n^{-s/(2s+1)},$$
where $\E_{\mu}[\cdot]$ denotes expectation with respect to any initial distribution $\mu(dx)$ for $\xi_\emptyset$ on $(0,\infty)$ such that $\int_{0}^\infty V(x)^2\mu(dx)<\infty$.
 \end{theorem}
We illustrate this result in Fig.~\ref{fig:DHKR1} for the reconstruction of a rate $xB(x)=x^2$.
\begin{figure}
\begin{center}
\includegraphics[width=10cm,height=5cm]{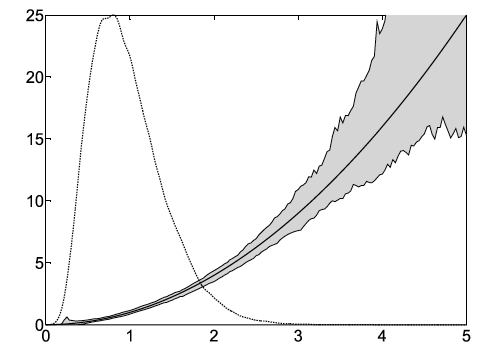}
\end{center}
\caption{\label{fig:DHKR1}Reconstruction for $n=2^{10}$, error band for $95\%$,
$100$ simulations, with a small threshold $\varpi_n=1/n$. We see that the error grows larger as $x$ grows larger, as expected since the density of cells vanishes. Fig.~3 from~\protect\cite{DHKR}.}
\end{figure}
Since $\varpi_n$ is arbitrary, we obtain the classical rate $n^{-s/(2s+1)}$ which is optimal in a minimax sense for density estimation. It is optimal in our context, using for instance the results of \cite{penda2017adaptive}. The knowledge of the smoothness $s$ that is needed for the construction of $\widehat B_n$ is not realistic in practice. An adaptive estimator could be obtained by using a data-driven bandwidth in the estimation of the invariant density $\nu_B(y/2)$ in \eqref{eq:representationB}. The Goldenschluger-Lepski bandwidth selection method seen in this context \cite{DHRR} would presumably yield adaptation.
Finally, let us revisit the representation formula
$$B(y)=\frac{1}{2}\frac{\nu_B(y/2)}{\mathbb P_{\nu_B}\big(\xi_{u^-} \leq y,\;\xi_u \geq y/2\big)}$$
by noticing that we always have $\{\xi_{u^-} \geq y\} \subset \{\xi_u \geq y/2\}$, hence
\begin{align*}
\mathbb P_{\nu_B}\big(\xi_{u^-}\leq y , \xi_u \geq y/2\big)  = &\; \mathbb P_{\nu_B}\big(\xi_u \geq y/2)- \mathbb P_{\nu_B}(\xi_{u^-} \geq y\big) \\
 = & \int_{y/2}^\infty \nu_B(x)dx- \int_y^\infty \nu_B(x)dx\\
  = &\int_{y/2}^y \nu_B(x)dx.
\end{align*}
Finally, for constant growth rate, we obtain
$$B(y)=\frac{1}{2}\frac{\nu_B(y/2)}{\int_{y/2}^y \nu_B(x)dx},$$
where we recognize~\eqref{estim:beta:size:1}.
The ``gain" in the convergence rate $n^{-s/(2s+1)}$ versus the rate $n^{-s/(2s+3)}$ obtained in the proxy model based on the transport-fragmentation equation, as seen in Section~\ref{subsec:estim:age} for the age model,
comes from the fact that we estimate the invariant measure ``at division" versus the invariant measure ``at fixed time" in the proxy model.
In other words, there is more ``nonparametric statistical information" in data extracted from ${\mathcal U}_2$ rather than $\mathcal V_T$, regardless of the tree structure. However $\big|\mathcal V_T \big| \approx \big| {\mathcal U}_2 \big|$
as stems from the supercritical branching process structure.

\subsubsection{Population point data}
\label{subsubsec:estim:size:pop}
We now turn to the case of less rich observation schemes, where we only have access to population point data, which relate to the eigenvector solution $N_k$. For the noise, we consider either a deterministic noise or a sampling noise: as for the individual dynamics case, a complete approach, which would depart from the stochastic branching tree, has not yet been carried out.

Contrarily to the case of individual dynamics data, the shape of the fragmentation kernel is important here, and needs to be previously known, as well as the Malthusian parameter $\lambda_k$ and the growth rate $\tau.$

\

Estimating the division rate $B$ may be formulated by the following inverse problem.

\

\noindent
\boxed{\begin{minipage}{\textwidth}
We know{\it a priori} the growth rate $\tau$, the fragmentation kernel $b(y,x)=\f{1}{y}b_0(\f{x}{y})$ and the Malthusian parameter $\lambda_k$, and measuring either $N_k^\ep$ (deterministic noise) such that
$$\Vert N_k^\ep - N_k\Vert_{W^{-s,p}([0,\infty))} \leq \ep,$$
or a sample $x_1,\cdots x_n$ realisations of $X_1,\cdots X_n$ {\it i.i.d.} random variables of law $N_k$ (sampling noise),
with $(N_k,\lambda_k)$ solution to~\eqref{eq:eigenproblem} {\it i.e.}
\begin{equation}\label{eq:inv1}
\left\{\begin{array}{l}
 \f{\p}{\p x} (\tau N_k(x)) + \lambda_k N(x) =- (\tau B N_k)(x) +  k\int_0^1 (\tau
 B N_k)(\f{x}{z}) \f{b_0(dz)}{z},
 \\ \\ \tau N_k(0)=0,
 \end{array}
 \right.
\end{equation}
How to estimate $B$?
\end{minipage}}

\

We notice the same ingredients as for the age model: 1/ $B$ appears in the equation only multiplied by $\tau N_k,$ so that we cannot estimate it at places where $\tau N_k$ vanishes; 2/ the equation displays a size-derivative of $N_k,$ so that we expect an inverse problem of degree of ill-posedness one as in Section~\ref{subsubsec:estim:age:pop}; 3/ replacing the unknown $B$ by the density of dividing cells $B N_k,$ we transform a nonlinear inverse problem into a linear one.

However, contrarily to the simple age model, we do not have any explicit formula here: a dilation operator needs to be inverted.
We decompose~\eqref{eq:inv1} into
\begin{equation}\label{eq:inv2}
\begin{array}{lll}
L(N_k) = {G}_k(B \tau N_k) & {\text{ with }}&
{G}_k(f)(x):= k\int_0^1 f(\f{x}{z})\f{b_0(dz)}{z}-f(x),\\ \\
&& L(N_k):=  \partial_x\big(\tau N_k\big)+\lambda_k N_k,\end{array}
\end{equation}
and estimate $B$ through the following steps:
\begin{itemize}
\item Solve ${G}_k(f)=L$ for $f,$ $L$ in suitable  spaces: for simplicity of inversion formulae, we thus choose Hilbert spaces of type $L^2([0,\infty),x^pdx)$~\cite{Engl}.
Note that under this shape, the problem to solve $N_k\to f=B \tau N_k$ is now linear.

\item Estimate $L(N_k)$ from the measurement in the chosen space: to do so, we can either depart from $N_k^\ep$ or from a statistical sample, and then divide by (an estimate of) $N_k$ and threshold appropriately.

\end{itemize}
The second step is exactly the same as what has been done in Section~\eqref{subsec:estim:age} for the renewal equation, so that we will not detail this part anymore here, and focus on the first step, which is specific to the growth-fragmentation equation.

\paragraph{Solving a dilation equation}

\

The first step consists in inverting the dilation operator $G_k$ defined by~\eqref{eq:inv2}. We begin to treat the equal mitosis case, first studied  by B. Perthame and J. Zubelli~\cite{PZ,DPZ,DHRR}, and then turn to more general fragmentation kernels~\cite{BDE}.

For the equal mitosis or diagonal kernel $b_0(x)=\delta_{x=\f{1}{2}},$ we have the following result.
\begin{proposition}{\bf{(Adapted from Theorem A.3. of~\cite{DPZ})}}\label{prop:dilation}
Let $G_k$ defined by~\eqref{eq:inv2}, $b_0(x)=\delta_{x=\f{1}{2}},$ $L\in L^2(x^p dx)$ with $p\neq 2k-1$.

There exists a unique solution $f\in L^2(x^p dx)$ to $$G_k(f)=2kf(2\cdot)-f(\cdot)=L,$$ and this solution depends continuously on $\Vert L\Vert_{L^2(x^pdx)}$. Moreover, defining
$$H_k^0:=\sum\limits_{j=1}^\infty (2k)^{-j} L(2^{-j} x),\qquad H_k^\infty:=-\sum\limits_{j=0}^\infty (2k)^{j}L(2^jx),$$
we have $f=H_k^0$ if $p<2k-1$ and $f=H_k^\infty$ if $p>2k-1.$ If $L\in L^q$ then $H_0\in L^q$  for any $1\leq q\leq \infty.$

For $L=0,$ any distribution of the form $ f(\f{\ln x}{x^2})$ with $f\in {\cal D}' ([0,\infty))$ $\ln-2$ periodic is solution.
\end{proposition}
\begin{proof} The proof is based on the Lax-Milgram theorem applied to the bilinear forms
$$\left\{\begin{array}{l}a_p(u,v)= \int \left(-2k u(2x)+u(x)\right)v(x)x^p dx,
\\ \\
b_p(u,v)= \int \left(2k u(2x)-u(x)\right)v(2x)x^p dx,
\end{array}\right.$$
and by Cauchy-Schwarz inequality, $a_p$ and $b_p$ are respectively coercive for $p>p_k$ and $p<p_k$ with $p_1=1$ and $p_2=3.$
\end{proof}

In this result, we first notice that uniqueness depends on the space chosen: there is no reason, generally speaking, to have $H_k^0=H_k^\infty,$ so that departing from an estimate $L(N_k^\ep)$, we have infinitely many choices. We first restrict to the two solutions $H_k^0$ and $H_k^\infty,$ since the others do not vanish fast enough at infinity compared to the one we look for, recall Theorem~\ref{theo:eigenMDPG}. Among these two solutions, $H^0_k$ "behaves better" at infinity, and $H^\infty_k$ at zero, hence in~\cite{BDE} we proposed a combination of both solutions as a best approximation in $L^2((x^p+1)dx)$ for $p>3$, namely, for a given $L,$ we define
\begin{equation}H_k^{\bar x} = H_k^0 \1_{\{x\leq \bar x\}} + H_k^{\infty} \1_{\{x>\bar x\}}.\label{def:H}
\end{equation}
This solution is no more an exact solution of $G_k(f)=L$  unless $H_k^0=H_k^\infty$. However, this property being satisfied for the "true" underlying distribution, it defines a convenient approximation as shown below.

  For general self-similar fragmentation kernels, we use the Mellin transform - an important tool for the study of the equation in many cases, see~\cite{doumic2016time,escobedo2017short}. We recall that the Mellin transform
 $\mathcal M$ is an isometry between $L^2(x^q dx)$ and $L^2(\f{q+1}{2}+i\R)$ defined by
\begin{equation}\mathcal M [f] (s):=\int_{ 0 }^\infty x^{s-1}f(x)dx,\quad \mathcal M^{-1}_{q} [F](x):=\int_{-\infty}^\infty x^{-{\f{q+1}{2}} -iv}F({{\f{q+1}{2}}}+iv)dv\label{def:Mellin}.
\end{equation}
The Mellin transform for the operator $G_k$ is then
$$\mathcal M [\mathcal G(f)](s)=(k\mathcal M [b_0] (s)-1) \mathcal M [f](s),$$
and we see that, if $\vert k\mathcal M [b_0] (s)-1\vert$  is bounded from below by a positive constant on the integration line $\f{q+1}{2}+i\R,$ we can define the inverse
\begin{equation}\label{eq:invMellin}H_k^{q}:=\mathcal M^{-1}_{q}\biggl[\f{\mathcal M [\mathcal G(f)] (s)}{k\mathcal M [b_0](s)-1}\biggr].
\end{equation}
We notice that for $s_k=k$ we have $k{\cal M}[b_0](s_k)-1=0.$ This zero, together with the isometry given by~\eqref{def:Mellin} and the inversion formula~\eqref{eq:invMellin}, gives insight on why $p_k=2k-1$ is the pivot in Proposition~\ref{prop:dilation} (we have $s_k=\f{p_k+1}{2}$), and why $H_k^q\neq H_k^{q'}$ if $q<s_k<q':$ the residue theorem quantifies exactly their difference, see Proposition~4 in~\cite{BDE} for more details.

\paragraph{Estimate with a deterministic noise}

Let us assume here that we measure $N_k$ up to a deterministic noise of level $\ep>0$, $0\leq \theta<1:$
$$\Vert N_k- N_{k}^\ep\Vert_{H^{-\theta}([0,\infty))} \leq \ep.$$
By the general theory of linear inverse problems, let us pick any regularisation method of optimal order~\cite{Engl}, of parameter $h>0,$ and define an approximation $L(N_k^\ep)_h$ such that, for $N_k\in H^m([0,\infty))$, and $q>2k-1,$ we have
$$\Vert  L(N_k^\ep)_h - L(N_k)\Vert_{L^2((1+x^q)dx)}\leq C (\f{\ep}{h} +h^m),$$
with $C$ depending only on the method chosen and on the norm of $N_k$ in $H^m.$
Since we want to estimate $H_k=BN_k$ in $L^2((1+x^q)dx)$ with $q>2k-1$, we define for some $a>0$
\begin{equation}\label{def:Hest}H_{\ep,h}:=\mathcal M^{-1}_{0}\biggl[\f{\mathcal M [L(N_k^\ep)_h] (s)}{k\mathcal M [b_0](s)-1}\biggr] \1_{\{x\leq a\}}
\;{\dst{+}}\;
\mathcal M^{-1}_{q}\biggl[\f{\mathcal M [L(N_\ep)_h] (s)}{k\mathcal M [b_0](s)-1}\biggr] \1_{\{x>a\}}\end{equation}
and we get the following proposition.
\begin{proposition}{\bf{(Adapted from~\cite{BDE}, Theorem~1.1.)}}
For $N\in H^m([0,\infty))$ solution to the eigenequation~\eqref{eq:inv1} we have
$$\Vert N- N_\ep\Vert_{H^{-\theta}([0,\infty))} \leq \ep \implies \Vert H_{\ep,h} - BN_k\Vert_{L^2((1+x^q)dx)} \leq C (\f{\ep}{h^{1+\theta}} +h^m)$$
where $C$ depends only on $ \Vert N \Vert_{H^m}$ and the regularisation method chosen.
\label{prop:estim:size:inv:det}
\end{proposition}
\paragraph{Estimation with a stochastic noise}
Let us now assume that we observe a sample of $n$ cells, of sizes
$x_1,\cdots,x_n$, that are realisations of
$$(X_1,\ldots, X_n),$$
where the $X_i$ are independent, with common density distribution $N(x)dx$ (recall that we have $N \geq 0$ and that we pick the normalisation $\int_0^\infty N=1$). From \eqref{eq:inv2}, we have the formal representation
$$B = \frac{({G}_k)^{-1}\big(L(N_k)\big)}{\tau N_k}.$$
Thus, from data $(X_1,\ldots, X_n)$ we can build a regularisation
$$\widehat N_{h,k} = \frac{1}{nh}\sum_{i = 1}^n K_h(\cdot-X_i)$$
which simply amounts to have a convolution of the empirical measure $n^{-1}\sum_{i = 1}^n \delta_{X_i}(dx)$ with a well behaved kernel $K_h(\cdot) = h^{-1}K(h^{-1}\cdot)$. The resulting $\widehat N_{h,k}$ being smooth (by picking $K$ smooth enough), we may  compute the action of the differential operator $L$ on $\widehat N_h$. As soon as the operator $G_k$ has bounded inverse, or that a nice approximation $({\mathfrak G}_k)^{-1}$ of $({G}_k)^{-1}$ is available, we may form the simple estimator:
\begin{equation} \label{eq: def est DHRR}
\widehat B_{n,h} = \frac{({\mathfrak G}_k)^{-1}\big(L(\widehat N_{h,k})\big)}{\tau \widehat N_{h,k}}.
\end{equation}

In~\cite{DHRR}, we realise this program for the binary fragmentation operator and we propose a method to automatically select the bandwidth $h$ from data.
The following kind of results can be obtained: for a compact set $\mathcal D$, one can construct an approximation $({\mathfrak G}_k)^{-1}$ and select a bandwidth $h$ such that if $B \in H^s$, then
\begin{equation} \label{eq: the rate we find}
\E\big[\|\widehat B_{n,h}-B\|_{L^2(\mathcal D)}\big] \lesssim n^{-s/(2s+3)}.
\end{equation}
 We obtain the rate of convergence that corresponds to ill-posed problem of order 1, and this is consistent with the other approaches based on large population data (here of size $n$) alive at a given fixed, (but large) time.

\subsubsection*{Comparing deterministic and stochastic methods}

The stochastic method rate of convergence $n^{-s/(2s+3)}$ for ill-posed problems of degree 1 in the statistical minimax theory \cite{gine2021mathematical} is to be compared with the deterministic method rate $\epsilon^{s/(s+1)}$, as stems from the classical theory exposed for instance in the classical textbook \cite{Engl}. These are actually the same results, as stems from a classical analysis of comparison between stochastic and deterministic ill-posed inverse problems, see the illuminating paper~\cite{NP}.

Suppose we have an approximate knowledge of $N$ and $\lambda$ up to deterministic errors $\zeta_1 \in L^2$ and $\zeta_2 \in \R$ with noise level $\varepsilon>0$: we observe
\begin{equation} \label{det error model}
N_\varepsilon=N+\varepsilon \zeta,\;\;\|\zeta\|_{L^2}\leq 1,
\end{equation}
and
\begin{equation} \label{calibration}
\lambda_\varepsilon = \lambda + \varepsilon \zeta_2,\;\;|\zeta_2|\leq 1.
\end{equation}
From the formal representation
$$B = \frac{({G}_k)^{-1}\big(L(N_k)\big)}{\tau N_k},$$
The recovery of $L(N_k)$ is ill-posed in the terminology of Wahba \cite{wahba1977practical} since it involves the computation of the derivative of $N$. If ${G}_k$ is bounded with an inverse bounded in $L^2$ and the dependence in $\lambda$ is continuous, the overall inversion problem is ill-posed of degree $a=1$. By classical inverse problem theory for linear cases (here the problem is non linear), this means that if $N \in W^{s,2}$, the optimal recovery rate in $L^2$-error norm should be $\varepsilon^{s/(s+a)} = \varepsilon^{s/(s+1)}$ (see also \cite{PZ,DPZ}).

Suppose now that we replace the deterministic noise $\zeta_1$ by a random Gaussian {\it white noise}: we observe
\begin{equation} \label{stochastic error model}
N_\varepsilon = N+\varepsilon \dot W = A \tfrac{\partial}{\partial x} N(x)dx+\varepsilon \dot{\mathcal W}
\end{equation}
where $\dot{\mathcal W}$ is a Gaussian white noise (a random distribution) and $A\varphi(x)=\int_0^x \varphi(y)dy$ denotes the integration operator (which has degree of ill-posedness $a=1$). This setting is actually statistically (asymptotically) very close to observing $(X_1,\ldots, X_n)$ where the $X_i$ are independent random variables, with common density $N$, at least over compact intervals $\mathcal D$ and as soon as $N$ does not degenerate, as follows from the celebrated result of Nu\ss baum \cite{nussbaum1996asymptotic}.\\

In this setting, we want to recover $\tfrac{\partial}{\partial x} N$. Integrating, we equivalently observe
$$Y_\varepsilon(\cdot) = \int_0^\cdot A \tfrac{\partial}{\partial x} N(x)dx +\varepsilon \mathcal W_\cdot,$$
where $(\mathcal W_x, x\geq 0)$ is a Brownian motion. Applying formally the $1/2$-fractional derivative operator $D^{1/2}$, we recast the observation $Y_\varepsilon$ into an equivalent observation
$$Z_\varepsilon = D^{1/2}\int_0^\cdot A \tfrac{\partial}{\partial x} N(x)dx+\varepsilon D^{1/2}\mathcal W_\cdot,$$
so that $D^{1/2}\mathcal W\cdot$ is in $L^2$. The operator $\varphi \mapsto D^{1/2}\int_0^\cdot A \varphi(x)dx$ maps $L^2$ onto $W^{1+1-1/2,2}$ {\it i.e.} we have an effective degree of ill-posedness $a=3/2$. We should then obtain an optimal rate of the form
$$\varepsilon^{s/(s+3/2)} = \varepsilon^{2s/(2s+3)}=n^{-s/(2s+3)}$$
for the calibration $\varepsilon = n^{-1/2}$ dictated by \eqref{calibration} when we compare our statistical model with the deterministic perturbation. This is exactly the rate we find in \eqref{eq: the rate we find}: the deterministic error model and the statistical error model coincide to that extent.
%
%
%
%
%
%
%
%
%

\subsection{Estimating an increment-dependent division rate}
\label{subsec:estim:adder}
\subsubsection*{Individual dynamics data: a simple renewal problem?}

Despite the intricate character of the adder model, which combines the influences of the age and size, in the case where one is given individual dynamics data, the estimate to do is exactly the same as for the pure age-structured model... at least for the genealogical observation case. In this case, given that we observe $a_1,\cdots, a_n$ increments at division,  they form a renewal process, observed without bias, and we can apply~\eqref{eq:estim:age1} and standard density estimation methods as explained in Section~\ref{subsec:estim:age}.

For the population observation case ($k=2$), things are more intricate: as for the previous models, there exists a selection bias, which is not the same as given by~\eqref{def:B:direct} since the growth rate of the increment is no more constant and moreover depends on size. However, in the case of exponential growth $\tau(x)=\kappa x,$  we have the small miracle already mentioned that $\lambda_2=\kappa$ and, with $C,C_d>0$  normalisation constants, we have
$$N_1(a,x)=C x N_2(a,x) \implies f_1(a,x)=C_d x f_2(a,x),$$
where we denote $f_k(a,x)$ the density of dividing cells of increment $a$ and size $x$.
We can thus take advantage of Formula~\eqref{eq:estim:age1} and write, for $f_{k}^a(a)=\int_0^\infty f_k(a,x)dx$ the marginal density along $a$:
\begin{equation}
\label{def:B:incr2}
B(a)=\f{f_1^a(a)}{\int_a^\infty f_1^a(s)ds}=\f{\int_0^\infty f_2 (a,x) x dx}{\int_a^\infty \int_0^\infty f_2(s,x) x dx ds}.
\end{equation}
Instead of estimating directly the density $f_2^a(a)$ from the sample $(a_1,\cdots,a_n)$, as done for the genealogical case, one has to estimate a weighted density: from a sample $\left((A_1,X_1);\cdots (A_n,X_n)\right)$ of increment and size at division of cells according a population observation, we define the following estimate for the debiased density $f_1:$
\begin{equation}\label{def:f1debiased}
\hat f_1^a (a)da=K_h * \Big(\f{1}{n}\sum\limits_{i=1}^n X_i \delta_{A_i}(da)\Big)=\f{1}{n}\sum\limits_{i=1}^n X_i K_h(a-A_i)da,
\end{equation}
for which we can prove the same rate of convergence as done for instance in Proposition~\ref{prop:estim:age:direct:det}.
\subsubsection*{Population point data: a severely ill-posed problem}
When we have treated the case of population point data for the age-structured equation, in Section~\ref{subsec:estim:age}, we have remarked that it was rather an instructive toy model than a really interesting case, since if we are able to observe ages in a population, this should mean that we are also able to observe newborn among them, hence dividing cells, hence we would be back to the individual dynamics data.

In the case of the size-structured model, studied in Section~\ref{subsubsec:estim:size:pop} on the contrary, the inverse problem setting appears fully relevant in many experimental cases and different applications - fragmenting polymers, {\it in vivo} dividing cells, etc.

For the increment model, as for the renewal model, a "naive" inverse problem, which is to assume that we measure samples distributed along a density $N(a,x)$ solution to~\eqref{eigen:adder} seems irrelevant: observing samples distributed along $N(a,x)$ implies that we also observe newborn, {\it i.e.} $N(0,x)$, and then why not as well dividing cells - which would bring us back to the individual dynamics data, discussed in the paragraph above.  

Much more realistic is the following case, studied in~\cite{doumic2020estimating}: Assume we are given $x_1,\cdots,x_n$ a $n-$ sample of realizations of $X_1,\cdots,X_n$ {\it i.i.d.} random variables of density $N^x_k(x)$ defined by
$$N_k^x(x):=\int_0^\infty N_k(\i,x)d\i,$$
other said, $N_k^x$ is the $x$-marginal, or yet the size-distribution of a sample of cells following an increment-structured dynamics, and having reached their steady behaviour given by $N_k(\i,x)$. How to estimate an {\it increment}-dependent division rate $B(a)$ from a {\it size}-dependent distribution?

Surprisingly, the answer is given by the following proposition, which provides us with a reconstruction formula - and shows that this inverse problem is severely ill-posed.
\begin{proposition}{\bf{(Adapted from Proposition~1 in~\cite{doumic2020estimating})}}
We have the following reconstruction formula, where $\cal F$ and ${\cal F}^{-1}$ denote the Fourier and inverse Fourier transform:
$$B(\i)=\f{f(\i)}{\int_\i^\infty f(s)ds}, \qquad f(\i):={\cal F}^{-1}\biggl(\f{{\cal F} [H_1(\cdot)]}{{\cal F}[H_1(2\cdot)]}\biggr),$$
where $H_1(x)=\kappa x \int_0^\infty B(\i) x^{k-1} N_k(\i,x)d\i$ is the solution of the dilation equation given in Proposition~\ref{prop:dilation}:
\begin{equation}
\label{eq:dilation}
{\cal L}(x)=\f{\p}{\p x} (\kappa x^k N_k)=2 H_1(2x)-H_1(x).
\end{equation}
From this formula, we deduce that if
we observe $X_1,\cdots X_n$ an i.i.d. sample of law $N_k^x(x),$ we propose the following estimator for $B:$
$$
\widehat{B}_{n,h,h'}(\i)
= \frac{\widehat{f}_{n,h}(\i)}{\widehat{S}_{n,h}(\i) }
= \frac{\int_{-1/h'}^{1/h'}  \frac{\widehat{H_{1,n,h}(\cdot)^*}(\xi)}{\widehat{H_{1,n,h}(\cdot)^*}(\xi/2)}    e^{-\mathbf{i}a\xi} d\xi}
{\int_{\i}^{\infty}\int_{-1/h'}^{1/h'}  \frac{\widehat{H_{1,n,h}(\cdot)^*}(\xi)}{\widehat{H_{1,n,h}(\cdot)^*}(\xi/2)}   e^{-\mathbf{i}s\xi} d\xi ds }
$$
where $\widehat H_{k,n}$ denotes an approximate solution to the dilation equation~\eqref{eq:dilation} as seen in Proposition~\ref{prop:estim:size:inv:det} with, for the left-hand side,
$${\cal L}_{n,h}(x):=\f{1}{n}\sum_{i=1}^n\f{\p}{\p x}\left(\kappa x^kK_h(x- X_i)\right)$$

\end{proposition}

\begin{proof} As a sketch of the proof in this simpler case where $\tau(x)=\kappa x$ (see~\cite{doumic2020estimating} for a general growth rate $\tau(x)$) we first write the equation for $N_1$, noticing that $N_1=x^{k-1} N_k$ up to a constant,  and integrate it along $\i$ to find the dilation equation~\eqref{eq:dilation}. We then solve the equation for $C(\i,x)=\kappa x N_1(\i,x)$ along the characteristics, and find
$$\begin{array}{ll}\kappa x  N_1(\i,x)&=\kappa (x-\i) N_1(0,x-\i)e^{-\int_0^\i B(s)ds}.
\end{array}$$
We use this expression in the definition of $H_1$ and find
$$H_1(x)=\int_0^x B(\i) \kappa x N_1(\i,x)d\i
= \int_0^x B(\i)\kappa (x-\i) N_1(0,x-\i)e^{-\int_0^s B(s)ds} d\i$$
using the boundary condition, we also have
$$\kappa x N_1(0,x)=4\kappa x \int_0^{2x} B(\i)N_1(\i,2x) dz=2H_1(2x)
$$
hence
$$H_1(x)=\int_0^x B(\i)e^{-\int_0^\i B(s)ds} 2H_1(2(x-\i))d\i=f_1 *(2H_1(2\cdot)) (x)$$
which appears as a deconvolution problem, where $2H_1(2x)$ plays the role of  "noise".
\end{proof}
We refer to~\cite{doumic2020estimating} for a thorough numerical investigation.

\section{Application to experimental data}
\label{sec:back}

In the previous section, we have developed methods to estimate the division rate in three different models: age-structured, size-structured or size-increment-structured model. In the case where data are rich enough, these methods can be used not only to estimate the division rate but also to select which model is more likely. This model selection is carried out here on the case of individual dynamics data:  in the case of point population data, estimating the division rate is possible, but the comparison between models is not, since the data is not informative enough.
\subsection{Guideline of a protocol}

To test a given (age, size or size-increment - or anything else not included in this chapter) model, we
\begin{itemize}
\item calibrate it as done in Section~\ref{sec:inverse},
\item simulate an age-size or increment-size model, in which our modelling assumption is embedded: see in Section~\ref{subsec:model:gene} the model~\eqref{eq: EDP gene}\eqref{eq: EDP gene bound}. For instance, if we want to compare the adder to the sizer without generalising any assumption on the growth rate or on the fragmentation kernel, we can simulate the following model:
$$
\f{\p}{\p t}  n_k + \f{\p}{\partial \i} (\kappa x n_k)+ \f{\p}{\partial x} \big(\kappa x n_k\big)  = -  \kappa x B(\i,x)n_k(t,\i,x) ,$$
$$n_k(t,\i=0,x)= 4k\int_0^\infty B(\i,2x) n_k(t,\i,2x)d\i$$
  till  its asymptotic steady behaviour $n_k(t,a,x)=e^{\lambda t} N_k(a,x)$ is reached. This simulation step may be carried out either by using the PDE and an adequate numerical scheme, see for instance~\cite{carrillo2014splitting,carrillo2019escalator}, or by Monte-Carlo simulations.
\item Compare quantitatively data and simulations, by defining a convenient distance between measures. This choice has always some intrinsic arbitrariness and is open to debate. We refer to the paper \cite{ramdas2017wasserstein} that discusses Wasserstein distances (and some of their weighted variants), interpreted as connections from univariate methods like the Kolmogorov-Smirnov test, QQ plots and ROC curves, to other multivariate tests. A thorough discussion of statistical tests or distance choices in this context lies beyond the level of generality intended here. We nevertheless refer to \cite{robert:hal-00981312} where an operational protocol is proposed to compare different division rate models.
\end{itemize}

\subsection{Some results}
This protocol has been used in~\cite{robert:hal-00981312} to compare the age-structured and the size-structured equation, compared quantitatively by $L^2$ norms between regularised solutions. We reproduce in Fig.~\ref{fig:BMC2} the sensitivity analysis carried out, which concluded that taking into account the experimental variability in the growth rates (recall Fig.~\ref{fig:distrib:growthrate} Left) or in the fragmentation kernel $b_0$ (recall Fig.~\ref{fig:distrib:growthrate} Right) does not improve significantly the results.

\begin{figure}

\includegraphics[width=\textwidth]{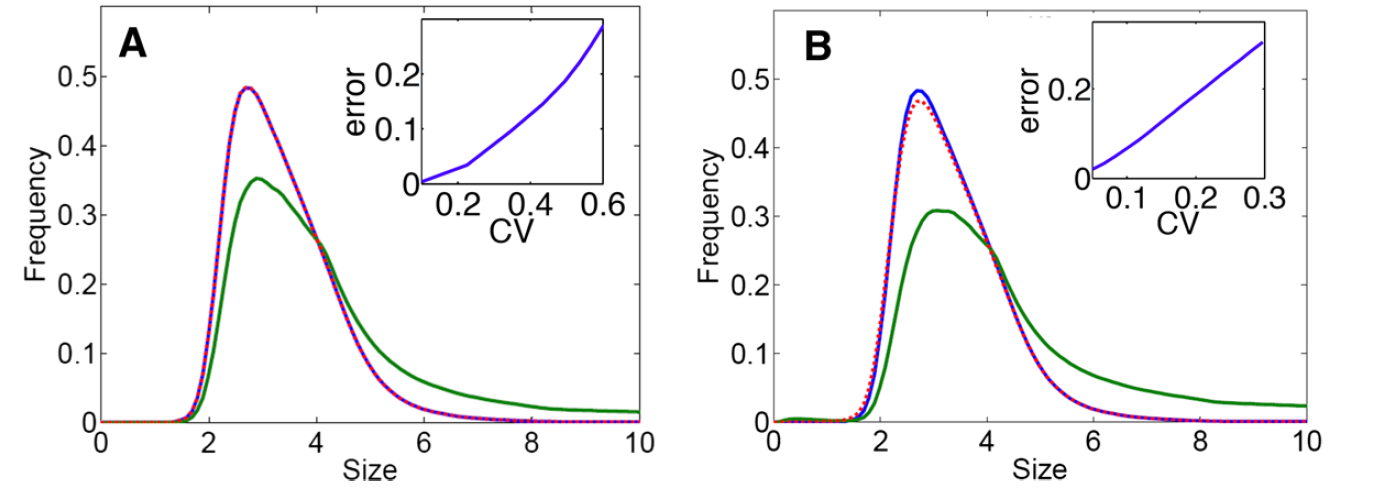}
\caption{Effect of adding variability on the size-distribution of cells. A: adding variability in the growth rate, B: in the fragmentation kernel. In green the size distribution with a coefficient of variation of $60\%$, in dotted pink the simulation without variability, in blue the estimate of the experimental distribution. The insets show the distance between the distribution with no variability and the distance with a variability of given CV: the conclusion is that, for these experiments on bacteria, variability is negligible and the simplified models fit well. Fig.~4 from~\protect\cite{robert:hal-00981312}.\label{fig:BMC2}}
\end{figure}
 This protocol  has then been improved to take into account the adder model, and implemented in the prototype CellDivision plateform \url{https://celldivision.paris.inria.fr/welcome/}, developed by Adeline Fermanian~\cite{doudoufer}. This plateform estimates the three models - age, size or increment structured - when the user provides individual dynamics data of lifetimes, size at birth, size at division and/or increment of size at division, together with a growth law. We show on Fig.~\ref{fig:celldiv} an application on data from~\cite{Wang} also used in~\cite{robert:hal-00981312}, for which it clearly appears that the incremental model fits better.
 \begin{figure}
 \includegraphics[width=\textwidth]{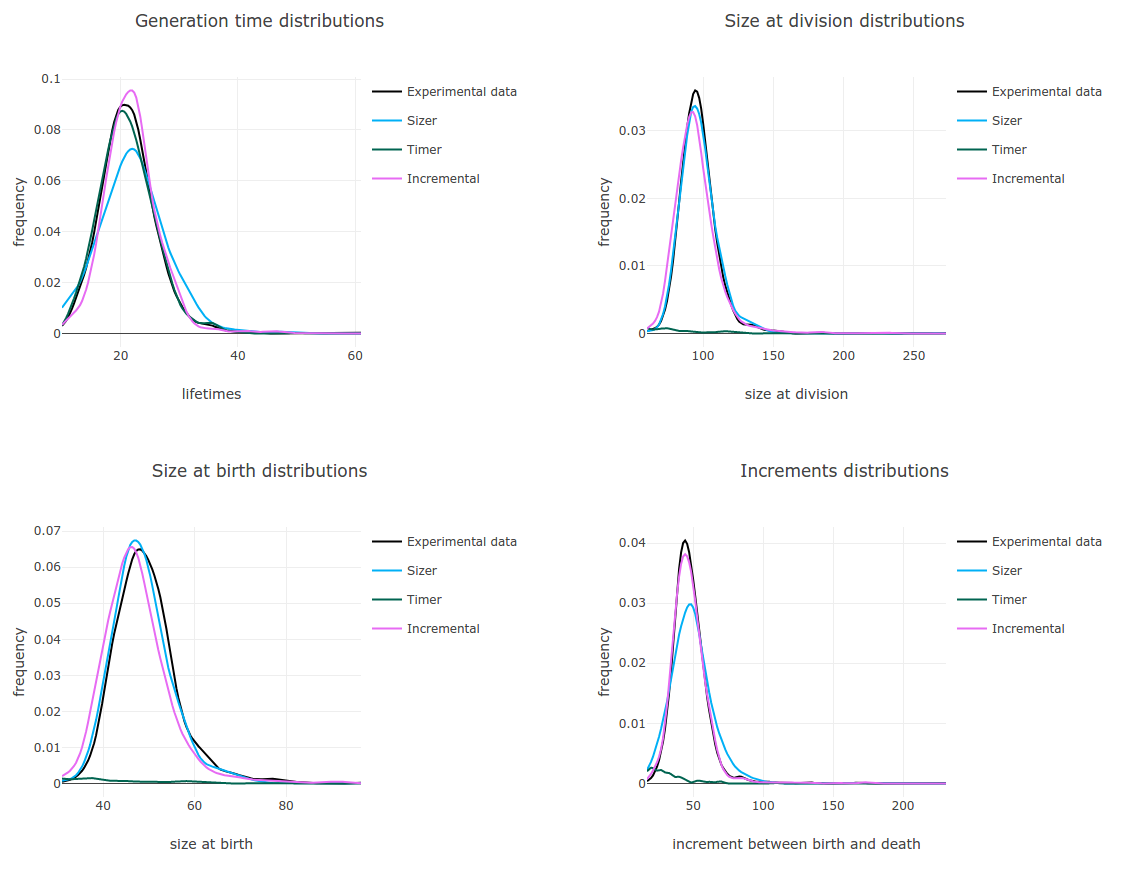}
 \caption{Screenshot of results obtained by the prototype plateform CellDivision, developed by Adeline Fermanian, to fit genealogical data of bacterial division from~\protect\cite{Wang}.\label{fig:celldiv} }
 \end{figure}

Of note, a very efficient way of comparing model and data, used in many biophysical papers such as~\cite{Jun_2015,soifer2016single}, consists in comparing the correlation coefficients between size, age and increment for the various models, since an age model predicts no correlation between size at birth and generation time, and an adder model between size at birth and increment at division. We reproduce in Table~\ref{tab:celldiv} the results obtained on genealogical data from~\cite{Wang}, comparing experimental correlations with the ones obtained from the calibrated  incremental model (using the CellDivision plateform): the match is excellent.

\begin{table}[ht]
\tbl{Correlation coefficients (C.C.)for the data shown in Fig.~\ref{fig:celldiv}.  \\ AD: Age at Division, SB: Size at Birth,  SD: size at Division, ID: Increment at Division. Computed by the CellDivision prototype plateform~\protect\cite{doudoufer}.}
{\begin{tabular}{l|c|c|c|c|c|c|r}
C.C. between: & AD/SB & AD/SD & AD/ID & SB/SD & SB/ID & SD/ID
\\
\hline
Experimental & -0.47 & 0.53 & 0.82 & 0.42 & -0.03 & 0.89
\\
Timer model & -0.02 & 0.04 & 0.08 & 0.98 & 0.93 & 0.98
\\
Sizer model & -0.66 & 0.67 & 0.92 & 0.08 & -0.39 & 0.89
\\
Adder model & -0.48 & 0.51 & 0.86 & 0.49 & -0.01 & 0.87
\end{tabular}
\label{tab:celldiv}}
\end{table}

\section{Perspectives and open questions}

In this chapter, we retraced the methods  developed by our group during the last decade, originating in~\cite{PZ}, to estimate the division rate in linear structured population equation models. They crucially rely on the mathematical analysis of the long-term behaviour of these equations and processes, sketched in Section~\ref{sec:anal}. Many developments still need to be done.

\

These models are well-adapted to model experiments in a steady environment, typically, with unrestricted nutrient and space. How to estimate the division in a non-steady environment? This could be of tremendous importance to understand for instance how cells react to an external stress, or to take into account cell-to-cell interaction.

\

We have focused in the chapter on three main models: age-structured, size-structured, and increment-structured model. The method is always the same, but the mathematical analysis of both the time asymptotics and of the inverse problem is specific. It would be very interesting to adapt it to richer models, for instance taking into account heterogeneity, or the G1/S/G2/M steps in the eucaryotic cell cycle, or yet models structured in DNA content, as the initiation of DNA replication appears to be the true important stage in the life of yeast cells~\cite{soifer2016single} for instance.

\

The inference methods outlined in Section~\ref{sec:inverse} could be enriched by taking into account other types of noise, in particular adding experimental measurement noise linked to image analysis. They could also lead to the design and study of statistical tests - especially interesting when the data is richer, given by individual dynamics data. In this case, we could also investigate the question of model selection: instead of depending on an {\it observed} variable, the division rate could depend on a {\it hidden} variable, more or less in the spirit of Section~\ref{subsec:estim:adder} for population point data~\cite{doumic2020estimating}.

\section*{Acknowledgments}
The content of this note is based on a lecture given at IMS of
National University of Singapore in January 2020. The authors thank
Professor Weizhu Bao and IMS for providing such a nice
opportunity.
This work was partially supported by the ERC Starting Grant SKIPPERAD (Grant number 306321). 


\begin{thebibliography}{118}
\providecommand{\natexlab}[1]{#1}
\providecommand{\url}[1]{\texttt{#1}}
\expandafter\ifx\csname urlstyle\endcsname\relax
  \providecommand{\doi}[1]{doi: #1}\else
  \providecommand{\doi}{doi: \begingroup \urlstyle{rm}\Url}\fi

\bibitem{duvernoy2018asymmetric}
M.-C. Duvernoy, T.~Mora, M.~Ardr{\'e}, V.~Croquette, D.~Bensimon, C.~Quilliet,
  J.-M. Ghigo, M.~Balland, C.~Beloin and S.~Lecuyer, Asymmetric adhesion
  of rod-shaped bacteria controls microcolony morphogenesis, \emph{Nature
  communications}. {\bf 9} \penalty0 (2018), \penalty0 1--10.

\bibitem{dell2018growing}
D.~Dell'Arciprete, M.~Blow, A.~Brown, F.~Farrell, J.~S. Lintuvuori, A.~McVey,
  D.~Marenduzzo, and W.~C. Poon, A growing bacterial colony in two dimensions
  as an active nematic, \emph{Nature communications}. {\bf 9} \penalty0 (2018),
  \penalty0 1--9.

\bibitem{doumic2020purely}
M.~Doumic, S.~Hecht, and D.~Peurichard, A purely mechanical model with
  asymmetric features for early morphogenesis of rod-shaped bacteria
  micro-colony, \emph{arXiv preprint arXiv:2008.04532}.  \penalty0 (2020).

\bibitem{balaban2004bacterial}
N.~Q. Balaban, J.~Merrin, R.~Chait, L.~Kowalik, and S.~Leibler, Bacterial
  persistence as a phenotypic switch, \emph{Science}. {\bf 305} \penalty0
  (2004), \penalty0 1622--1625.

\bibitem{Wang}
P.~Wang, L.~Robert, J.~Pelletier, W.~L. Dang, F.~Taddei, A.~Wright, and S.~Jun,
  Robust growth of {E}scherichia coli, \emph{Curr. Biol.} {\bf 20} \penalty0
  (2010), \penalty0 1099--103.

\bibitem{meunier2021bacterial}
A.~Meunier, F.~Cornet, and M.~Campos, Bacterial cell proliferation: from
  molecules to cells, \emph{FEMS Microbiology Reviews}. {\bf 45} \penalty0
  (2021), \penalty0 fuaa046.

\bibitem{XueRadford2013}
W.-F. Xue and S.~E. Radford, An imaging and systems modeling approach to fibril
  breakage enables prediction of amyloid behavior, \emph{Biophys. Journal}.
  {\bf 105} \penalty0 (2013), \penalty0 2811--2819.

\bibitem{beal2020division}
D.~M. Beal, M.~Tournus, R.~Marchante, T.~J. Purton, D.~P. Smith, M.~F. Tuite,
  M.~Doumic, and W.-F. Xue, The division of amyloid fibrils: Systematic
  comparison of fibril fragmentation stability by linking theory with
  experiments, \emph{Iscience}. {\bf 23} \penalty0 (2020), \penalty0 101512.

\bibitem{tournus2021insights}
M.~Tournus, M.~Escobedo, W.-F. Xue, and M.~Doumic, Insights into the dynamic
  trajectories of protein filament division revealed by numerical investigation
  into the mathematical model of pure fragmentation.  \penalty0  (2021), hal-03131754.

\bibitem{hoffmann2011statistical}
M.~Hoffmann and N.~Krell, Statistical analysis of self-similar conservative
  fragmentation chains, \emph{Bernoulli}. {\bf 17} \penalty0 (2011), \penalty0
  395--423.

\bibitem{basse}
B.~Basse, B.~Baguley, E.~Marshall, W.~Joseph, B.~van Brunt, G.~Wake, and
  D.~J.~N. Wall, A mathematical model for analysis of the cell cycle in cell
  lines derived from human tumors, \emph{J. Math. Biol.} {\bf 47} \penalty0
  (2003), \penalty0 295--312.

\bibitem{PZ}
B.~Perthame and J.~Zubelli, On the inverse problem for a size-structured
  population model, \emph{Inverse Problems}. {\bf 23} \penalty0 (2007),
  \penalty0 1037--1052.

\bibitem{MischlerRicard}
M.~Escobedo, S.~Mischler and M.~Rodriguez~Ricard, On self-similarity and stationary problem for fragmentation and coagulation models, \emph{Annales de l'Institut Henri Poincare (C) Non Linear Analysis}. {\bf 22} \penalty0 (2005),
  \penalty0 99--125.

\bibitem{hoffmann2016nonparametric}
M.~Hoffmann and A.~Olivier, Nonparametric estimation of the division rate of an
  age dependent branching process, \emph{Stochastic Processes and their
  Applications}. {\bf 126} \penalty0 (2016), \penalty0 1433--1471.

\bibitem{van1996weak}
A.~W. Van Der~Vaart and J.~A. Wellner.
\newblock Weak convergence.
\newblock In \emph{Weak convergence and empirical processes}, pp. 16--28.
  Springer,  1996.

\bibitem{fournier2015rate}
N.~Fournier and A.~Guillin, On the rate of convergence in {W}asserstein
  distance of the empirical measure, \emph{Probability Theory and Related
  Fields}. {\bf 162} \penalty0 (2015), \penalty0 707--738.

\bibitem{weed2019sharp}
J.~Weed and F.~Bach, Sharp asymptotic and finite-sample rates of convergence of
  empirical measures in {W}asserstein distance, \emph{Bernoulli}. {\bf 25}
  \penalty0 (2019), \penalty0 2620--2648.

\bibitem{Eric}
E.~Stewart, R.~Madden, G.~Paul, and F.~Taddei, Aging and death in an organism
  that reproduces by morphologically symmetric division, \emph{PLoS Biol.} {\bf
  3} \penalty0 (2005).

\bibitem{robert:hal-00981312}
L.~Robert, M.~Hoffmann, N.~Krell, S.~Aymerich, J.~Robert, and M.~Doumic,
  {Division in Escherichia coli is triggered by a size-sensing rather than a
  timing mechanism}, \emph{{BMC Biology}}. {\bf 12} \penalty0 (2014), \penalty0
  17.
\newblock \doi{10.1186/1741-7007-12-17}.


\bibitem{ArielAmir}
A.~Amir, Cell size regulation in bacteria, \emph{Phys. Rev. Lett.} {\bf 112}
  \penalty0 (2014), \penalty0 208102.


\bibitem{Jun_2015}
S.~Taheri-Araghi, S.~Bradde, J.~T. Sauls, N.~S. Hill, P.~A. Levin, J.~Paulsson,
  M.~Vergassola, and S.~Jun, Cell-size control and homeostasis in bacteria,
  \emph{Current Biology}. {\bf 11679} \penalty0 (2015), \penalty0 1--7.

\bibitem{si2019mechanistic}
F.~Si, G.~Le~Treut, J.~T. Sauls, S.~Vadia, P.~A. Levin, and S.~Jun, Mechanistic
  origin of cell-size control and homeostasis in bacteria, \emph{Current
  Biology}. {\bf 29} \penalty0 (2019), \penalty0 1760--1770.

\bibitem{delyon2018investigation}
B.~Delyon, B.~de~Saporta, N.~Krell, and L.~Robert, Investigation of asymmetry
  in {E}. coli growth rate., \emph{Case Studies in Business, Industry \&
  Government Statistics}. {\bf 7} \penalty0 (2018).

\bibitem{lamperti1967continuous}
J.~Lamperti, Continuous state branching processes, \emph{Bulletin of the
  American Mathematical Society}. {\bf 73} \penalty0 (1967), \penalty0
  382--386.

\bibitem{kendall1966branching}
D.~G. Kendall, Branching processes since 1873, \emph{Journal of the London
  Mathematical Society}. {\bf 1} \penalty0 (1966), \penalty0 385--406.

\bibitem{harris1948branching}
T.~E. Harris, Branching processes, \emph{The Annals of Mathematical
  Statistics}.  \penalty0 (1948), \penalty0 474--494.

\bibitem{sevast1957limit}
B.~A. Sevastyanov, Limit theorems for branching stochastic processes of special
  form, \emph{Theory of Probability \& Its Applications}. {\bf 2} \penalty0
  (1957), \penalty0 321--331.

\bibitem{haccou2005branching}
P.~Haccou,  P.~Jagers,  and V.~Vatutin, \emph{Branching
  processes: variation, growth, and extinction of populations}. Number~5,
  Cambridge university press,  2005.


\bibitem{meleard2009trait}
S.~M{\'e}l{\'e}ard and V.~C. Tran, Trait substitution sequence process and
  canonical equation for age-structured populations, \emph{Journal of
  Mathematical Biology}. {\bf 58} \penalty0 (2009), \penalty0 881--921.

\bibitem{champagnat2006unifying}
N.~Champagnat, R.~Ferri{\`e}re, and S.~M{\'e}l{\'e}ard, Unifying evolutionary
  dynamics: from individual stochastic processes to macroscopic models,
  \emph{Theoretical population biology}. {\bf 69} \penalty0 (2006), \penalty0
  297--321.

\bibitem{champagnat2008individual}
N.~Champagnat, R.~Ferri{\`e}re, and S.~M{\'e}l{\'e}ard, From individual
  stochastic processes to macroscopic models in adaptive evolution,
  \emph{Stochastic Models}. {\bf 24} \penalty0 (2008), \penalty0 2--44.

\bibitem{bansaye2011limit}
V.~Bansaye, J.-F. Delmas, L.~Marsalle, and V.~C. Tran, Limit theorems for
  {M}arkov processes indexed by continuous time {G}alton--{W}atson trees,
  \emph{The Annals of Applied Probability}. {\bf 21} \penalty0 (2011),
  \penalty0 2263--2314.

\bibitem{bansaye2019scaling}
V.~Bansaye, M.-E. Caballero, and S.~M{\'e}l{\'e}ard, Scaling limits of
  population and evolution processes in random environment, \emph{Electronic
  Journal of Probability}. {\bf 24} \penalty0 (2019), 1--38.

\bibitem{BP}
B.~Perthame, \emph{Transport equations in biology}. Frontiers in Mathematics,
  Birkh\"auser Verlag, Basel,  2007.

\bibitem{MMP2}
P.~Michel, S.~Mischler, and B.~Perthame, General relative entropy inequality:
  an illustration on growth models, \emph{J. Math. Pures Appl. (9)}. {\bf 84}
  \penalty0 (2005), \penalty0 1235--1260.

\bibitem{DG}
M.~Doumic and P.~Gabriel, Eigenelements of a general aggregation-fragmentation
  model, \emph{Mathematical Models and Methods in Applied Sciences}. {\bf 20}
  \penalty0 (2009), \penalty0 757.
  
\bibitem{DHKR}
M.~Doumic, M.~Hoffmann, N.~Krell, and L.~Robert, Statistical estimation of a
  growth-fragmentation model observed on a genealogical tree, \emph{Bernoulli}.
  {\bf 21} \penalty0 (2015), \penalty0 1760--1799.

\bibitem{baccelli2020random}
F.~Baccelli, B.~B{\l}aszczyszyn, and M.~Karray.
\newblock Random measures, point processes, and stochastic geometry,  2020.

\bibitem{daley2003introduction}
D.~J. Daley and D.~Vere-Jones, \emph{An introduction to the theory of point
  processes: volume I: elementary theory and methods}. Springer,  2003.

\bibitem{jacod2013limit}
J.~Jacod and A.~Shiryaev, \emph{Limit theorems for stochastic processes}. vol.
  288, Springer Science \& Business Media,  2013.

\bibitem{Tran2008}
{Tran, Viet Chi}, Large population limit and time behaviour of a stochastic
  particle model describing an age-structured population, \emph{ESAIM: PS}.
  {\bf 12} \penalty0 (2008), \penalty0 345--386.


\bibitem{fournier2004microscopic}
N.~Fournier and S.~M{\'e}l{\'e}ard, A microscopic probabilistic description of
  a locally regulated population and macroscopic approximations, \emph{The
  Annals of Applied Probability}. {\bf 14} \penalty0 (2004), \penalty0
  1880--1919.


\bibitem{bertoin2006random}
J.~Bertoin, \emph{Random fragmentation and coagulation processes}. vol. 102,
  Cambridge University Press,  2006.

\bibitem{haas2003loss}
B.~Haas, Loss of mass in deterministic and random fragmentations,
  \emph{Stochastic processes and their applications}. {\bf 106} \penalty0
  (2003), \penalty0 245--277.

\bibitem{Kermack1}
W.~Kermack and A.~McKendrick, A contribution to the mathematical theory of
  epidemics, \emph{Proc. Roy. Society of London, Series A}. {\bf 115} \penalty0
  (1927), \penalty0 700--721.

\bibitem{MetzDiekmann}
J.~A.~J. Metz and O.~Diekmann, eds., \emph{The dynamics of physiologically
  structured populations}. vol.~68, \emph{Lecture Notes in Biomathematics},
  Springer-Verlag, Berlin,  1986.
\newblock ISBN 3-540-16786-2.
\newblock Papers from the colloquium held in Amsterdam, 1983.

\bibitem{harris1963theory}
T.~E. Harris, \emph{The theory of branching processes}. vol.~6, Springer
  Berlin,  1963.

\bibitem{feller2015integral}
W.~Feller.
\newblock On the integral equation of renewal theory.
\newblock In \emph{Selected Papers I}, pp. 567--591. Springer,  2015.

\bibitem{doob1948renewal}
J.~L. Doob, Renewal theory from the point of view of the theory of probability,
  \emph{Transactions of the American Mathematical Society}. {\bf 63} \penalty0
  (1948), \penalty0 422--438.

\bibitem{lotka1948application}
A.~J. Lotka, Application of recurrent series in renewal theory, \emph{The
  Annals of Mathematical Statistics}. {\bf 19} \penalty0 (1948), \penalty0
  190--206.

\bibitem{gabriel2018measure}
P.~Gabriel, Measure solutions to the conservative renewal equation,
  \emph{ESAIM: Proceedings and Surveys}. {\bf 62} \penalty0 (2018), \penalty0
  68--78.

\bibitem{gwiazda2016generalized}
P.~Gwiazda and E.~Wiedemann, Generalized entropy method for the renewal
  equation with measure data, \emph{Communications in Mathematical Sciences}. {\bf 15} 
  \penalty0 (2017), 577--586.

\bibitem{gwiazda2006invariants}
P.~Gwiazda and B.~Perthame, Invariants and exponential rate of convergence to
  steady state in the renewal equation, \emph{Markov Process. Related Fields}.
  {\bf 12} \penalty0 (2006), \penalty0 413--424.

\bibitem{mischler2002stability}
S.~Mischler, B.~Perthame, and L.~Ryzhik, Stability in a nonlinear population
  maturation model, \emph{Mathematical Models and Methods in Applied Sciences}.
  {\bf 12} \penalty0 (2002), \penalty0 1751--1772.

\bibitem{bansaye2019ergodic}
V.~Bansaye, B.~Cloez, and P.~Gabriel, Ergodic behavior of non-conservative
  semigroups via generalized {D}oeblin conditions, \emph{Acta Applicandae
  Mathematicae}.  \penalty0 (2019), \penalty0 1--44.

\bibitem{mischler2016spectral}
S.~Mischler and J.~Scher, Spectral analysis of semigroups and
  growth-fragmentation equations,  \emph{{Annales de l'Institut Henri
  Poincar{\'e} (C) Non Linear Analysis}}. {\bf 33} \penalty0 (2016), \penalty0
  849--898.

\bibitem{Feller}
W.~Feller, \emph{An introduction to probability theory and its applications.
  Vol. II.} John Wiley and Sons Inc.,  1971.

\bibitem{Gw}
P.~Gwiazda and B.~Perthame, Invariants and exponential rate of convergence to
  steady state in the renewal equation., \emph{Markov Processes and Related
  Fields}. {\bf 2} \penalty0 (2006), \penalty0 413--424.

\bibitem{Bansaye2}
V.~Bansaye and V.~C. Tran, Branching {F}eller diffusion for cell division with
  parasite infection, \emph{ALEA, Lat. Am. J. Probab. Math. Stat.} {\bf 8}
  \penalty0 (2011), \penalty0 95--127.

\bibitem{Bansaye}
V.~Bansaye, Proliferating parasites in dividing cells: {K}immel's branching
  model revisited, \emph{Ann. Appl. Probab.} {\bf 18} \penalty0 (2008),
  \penalty0 967--996.

\bibitem{cloez2017limit}
B.~Cloez, Limit theorems for some branching measure-valued processes,
  \emph{Advances in Applied Probability}. {\bf 49} \penalty0 (2017), \penalty0
  549--580.

\bibitem{DiekmannHeijmansThieme1984}
O.~Diekmann, H.~Heijmans, and H.~Thieme, On the stability of the cell size
  distribution, \emph{Journal of Mathematical Biology}. {\bf 19} \penalty0
  (1984), \penalty0 227--248.

\bibitem{diekmann2003steady}
O.~Diekmann, M.~Gyllenberg, and J.~Metz, Steady-state analysis of structured
  population models, \emph{Theoretical population biology}. {\bf 63} \penalty0
  (2003), \penalty0 309--338.

\bibitem{escobedo2017short}
M.~Escobedo, A short remark on a growth--fragmentation equation, \emph{Comptes
  Rendus Mathematique}. {\bf 355} \penalty0 (2017), \penalty0 290--295.

\bibitem{doumic2018explicit}
M.~Doumic and B.~Van~Brunt, Explicit solution and fine asymptotics for a
  critical growth-fragmentation equation, \emph{ESAIM: Proceedings and
  Surveys}. {\bf 62} \penalty0 (2018), \penalty0 30--42.

\bibitem{suebcharoen2011asymmetric}
T.~Suebcharoen, B.~Van~Brunt, and G.~Wake, Asymmetric cell division in a
  size-structured growth model, \emph{Differential and Integral Equations}.
  {\bf 24} \penalty0 (2011), \penalty0 787--799.

\bibitem{M1}
P.~Michel, Existence of a solution to the cell division eigenproblem,
  \emph{Math. Models Methods Appl. Sci.} {\bf 16} \penalty0 (2006), \penalty0
  1125--1153.

\bibitem{MMP1}
P.~Michel, S.~Mischler, and B.~Perthame, General entropy equations for
  structured population models and scattering, \emph{C. R. Math. Acad. Sci.
  Paris}. {\bf 338} \penalty0 (2004), \penalty0 697--702.

\bibitem{caceres2011rate}
M.~J. C{\'a}ceres, J.~A. Canizo, and S.~Mischler, Rate of convergence to an
  asymptotic profile for the self-similar fragmentation and
  growth-fragmentation equations, \emph{Journal de math{\'e}matiques pures et
  appliqu{\'e}es}. {\bf 96} \penalty0 (2011), \penalty0 334--362.

\bibitem{FournierPerthame}
N. Fournier and B. Perthame, 
A non-expanding transport distance for some structured equations, \emph{arXiv preprint}.

\bibitem{canizo2020spectral}
J.~A. Ca{\~n}izo, P.~Gabriel, and H.~Yolda{\c{s}}, Spectral gap for the
  growth-fragmentation equation via {H}arris theorem, \emph{arXiv preprint
  arXiv:2004.08343}.  \penalty0 (2020).

\bibitem{bernard2017asymptotic}
E.~Bernard and P.~Gabriel, Asymptotic behavior of the growth-fragmentation
  equation with bounded fragmentation rate, \emph{Journal of Functional
  Analysis}. {\bf 272} \penalty0 (2017), \penalty0 3455--3485.

\bibitem{balague:hal-00683148}
D.~Balagu{\'e}, J.~Ca{\~n}izo, and P.~Gabriel, {Fine asymptotics of profiles
  and relaxation to equilibrium for growth-fragmentation equations with
  variable drift rates}, \emph{{Kinetic and related models}}. {\bf 6} \penalty0
  (2013), \penalty0 219--243.

\bibitem{bertoin2019feynman}
J.~Bertoin, On a {F}eynman-{K}ac approach to growth-fragmentation semigroups
  and their asymptotic behaviors, \emph{Journal of Functional Analysis}. {\bf
  277} \penalty0 (2019), \penalty0 108270.

\bibitem{bertoin2020strong}
J.~Bertoin and A.~R. Watson, The strong malthusian behavior of
  growth-fragmentation processes, \emph{Annales Henri Lebesgue}. {\bf 3}
  \penalty0 (2020), \penalty0 795--823.

\bibitem{haas2004regularity}
B.~Haas.
\newblock Regularity of formation of dust in self-similar fragmentations.
\newblock In \emph{Annales de l'IHP Probabilit{\'e}s et statistiques}, vol.~40,
  pp. 411--438,  2004.

\bibitem{haas2010asymptotic}
B.~Haas, Asymptotic behavior of solutions of the fragmentation equation
  with shattering: an approach via self-similar {M}arkov processes,
  \emph{Annals of Applied Probability}. {\bf 20} \penalty0 (2010), \penalty0
  382--429.

\bibitem{goldschmidt2010behavior}
C.~Goldschmidt and B.~Haas, Behavior near the extinction time in self-similar
  fragmentations i : the stable case, \emph{Annales de l'I.H.P. Probabilit\'es
  et statistiques}. {\bf 46} \penalty0 (2010), \penalty0 338--368.
\newblock \doi{10.1214/09-AIHP317}.

\bibitem{dyszewski2021sharp}
P.~Dyszewski, N.~Gantert, S.~G. Johnston, J.~Prochno, and D.~Schmid, Sharp
  concentration for the largest and smallest fragment in a $ k $-regular
  self-similar fragmentation, \emph{arXiv preprint arXiv:2102.08935}.
  \penalty0 (2021).

\bibitem{goldschmidt2016behavior}
C.~Goldschmidt, and B.~Haas, Behavior near the extinction time in
  self-similar fragmentations ii: Finite dislocation measures, \emph{Annals of
  Probability}. {\bf 44} \penalty0 (2016), \penalty0 739--805.

\bibitem{banasiak2006shattering}
J.~Banasiak, Shattering and non-uniqueness in fragmentation models; an analytic
  approach, \emph{Physica D: Nonlinear Phenomena}. {\bf 222} \penalty0 (2006),
  \penalty0 63--72.

\bibitem{banasiak2019analytic}
J.~Banasiak, W.~Lamb, and P.~Lauren{\c{c}}ot, \emph{Analytic Methods for
  Coagulation-Fragmentation Models, Volume I}. CRC Press,  2019.

\bibitem{doumic2016time}
M.~Doumic and M.~Escobedo, Time asymptotics for a critical case in
  fragmentation and growth-fragmentation equations, \emph{Kinetic \& Related
  Models}. {\bf 9} \penalty0 (2016), \penalty0 251.

\bibitem{escobedo2020non}
M.~Escobedo.
\newblock On the non existence of non negative solutions to a critical
  growth-fragmentation equation.
\newblock In \emph{Annales de la Facult{\'e} des sciences de Toulouse:
  Math{\'e}matiques}, vol.~29, pp. 177--220,  2020.

\bibitem{bertoin2016probabilistic}
J.~Bertoin and A.~R. Watson, Probabilistic aspects of critical
  growth-fragmentation equations, \emph{Advances in Applied Probability}. {\bf
  48} \penalty0 (2016), \penalty0 37--61.

\bibitem{greiner1988growth}
G.~Greiner and R.~Nagel.
\newblock Growth of cell populations via one-parameter semigroups of positive
  operators.
\newblock In \emph{Mathematics applied to science}, pp. 79--105. Elsevier,
  1988.

\bibitem{bernard2016cyclic}
E.~Bernard, M.~Doumic, and P.~Gabriel, Cyclic asymptotic behaviour of a
  population reproducing by fission into two equal parts, \emph{Kinetic and Related Models}. {\bf 12} \penalty0 (2019), 551.

\bibitem{gabriel2019periodic}
P.~Gabriel and H.~Martin, Periodic asymptotic dynamics of the measure solutions
  to an equal mitosis equation, \emph{arXiv preprint arXiv:1909.08276}.
  \penalty0 (2019).

\bibitem{carrillo2014splitting}
J.~A. Carrillo, P.~Gwiazda, and A.~Ulikowska, Splitting-particle methods for
  structured population models: convergence and applications,
  \emph{Mathematical Models and Methods in Applied Sciences}. {\bf 24}
  \penalty0 (2014), \penalty0 2171--2197.

\bibitem{bansaye2015stochastic}
V.~Bansaye and S.~M{\'e}l{\'e}ard, \emph{Stochastic models for structured
  populations}. vol.~16, Springer,  2015.

\bibitem{gwiazda2010nonlinear}
P.~Gwiazda, T.~Lorenz, and A.~Marciniak-Czochra, A nonlinear structured
  population model: Lipschitz continuity of measure-valued solutions with
  respect to model ingredients, \emph{Journal of Differential Equations}. {\bf
  248} \penalty0 (2010), \penalty0 2703--2735.

\bibitem{gabriel2018steady}
P.~Gabriel and H.~Martin, Steady distribution of the incremental model for
  bacteria proliferation, \emph{arXiv preprint arXiv:1803.04950}.  \penalty0
  (2018).

\bibitem{D}
M.~Doumic, Analysis of a population model structured by the cells molecular
  content, \emph{Math. Model. Nat. Phenom.} {\bf 2} \penalty0 (2007), \penalty0
  121--152.

\bibitem{kang2020nonlinear}
H.~Kang, X.~Huo, and S.~Ruan, Nonlinear physiologically structured population
  models with two internal variables, \emph{Journal of Nonlinear Science}. {\bf
  30} \penalty0 (2020), \penalty0 2847--2884.

\bibitem{Banks1}
H.~Banks, K.~Sutton, W.~Thompson, G.~Bocharov, D.~Roosec, T.~Schenkeld, and
  A.~Meyerhanse, Estimation of cell proliferation dynamics using cfse data,
  \emph{Bull. of Math. Biol.}  \penalty0 (2010).

\bibitem{tsybakov2003introduction}
A.~B. Tsybakov, \emph{Introduction {\`a} l'estimation non param{\'e}trique}.
  vol.~41, Springer Science \& Business Media,  2003.

\bibitem{gine2021mathematical}
E.~Gin{\'e} and R.~Nickl, \emph{Mathematical foundations of
  infinite-dimensional statistical models}. Cambridge University Press,  2021.

\bibitem{hoffmann2018statistical}
M.~Hoffmann and A.~Olivier, Statistical analysis for structured models on
  trees, \emph{Statistical Inference for Piecewise-deterministic Markov
  Processes}.  \penalty0 (2018), \penalty0 1--38.

\bibitem{DHRR}
M.~Doumic, M.~Hoffmann, P.~Reynaud, and V.~Rivoirard, Nonparametric estimation
  of the division rate of a size-structured population, \emph{SIAM J. on Numer.
  Anal.} {\bf 50} \penalty0 (2012), \penalty0 925--950.

\bibitem{NP}
M.~Nussbaum and S.~Pereverzev, The degrees of ill-posedness in stochastic and
  deterministic noise models, \emph{Preprint WIAS 509}.  \penalty0 (1999).

\bibitem{DPZ}
M.~Doumic, B.~Perthame, and J.~Zubelli, Numerical solution of an inverse
  problem in size-structured population dynamics, \emph{Inverse Problems}. {\bf
  25} \penalty0 (2009), \penalty0 045008.

\bibitem{doumic:hal-01501811}
M.~Doumic, M.~Escobedo, and M.~Tournus, {Estimating the division rate and
  kernel in the fragmentation equation}, \emph{{Annales de l'Institut Henri
  Poincar{\'e} (C) Non Linear Analysis}}. {\bf 35}  \penalty0 (2018), 1847--1884.

\bibitem{hoang2020nonparametric}
V.~H. Hoang, T.~M. Pham~Ngoc, V.~Rivoirard, and V.~C. Tran, Nonparametric
  estimation of the fragmentation kernel based on a partial differential
  equation stationary distribution approximation, \emph{Scandinavian Journal of
  Statistics}.  \penalty0 (2020), 1--40.

\bibitem{hairer2011yet}
M.~Hairer and J.~C. Mattingly.
\newblock Yet another look at {H}arris’ ergodic theorem for {M}arkov chains.
\newblock In \emph{Seminar on Stochastic Analysis, Random Fields and
  Applications VI}, pp. 109--117,  2011.

\bibitem{baxendale2005renewal}
P.~H. Baxendale, Renewal theory and computable convergence rates for
  geometrically ergodic {M}arkov chains, \emph{The Annals of Applied
  Probability}. {\bf 15} \penalty0 (2005), \penalty0 700--738.

\bibitem{penda2017adaptive}
S.~V.~B. Penda, M.~Hoffmann, and A.~Olivier, Adaptive estimation for
  bifurcating {M}arkov chains, \emph{Bernoulli}. {\bf 23} \penalty0 (2017),
  \penalty0 3598--3637.

\bibitem{Engl}
H.~Engl, M.~Hanke, and A.~Neubauer, \emph{Regularization of inverse problems}.
  vol. 375, \emph{Mathematics and its Applications}, Springer,  1996.

\bibitem{BDE}
T.~Bourgeron, M.~Doumic, and M.~Escobedo, Estimating the division rate of the
  growth-fragmentation equation with a self-similar kernel, \emph{Inverse
  Problems}. {\bf 30} \penalty0 (2014), \penalty0 025007.


\bibitem{wahba1977practical}
G.~Wahba, Practical approximate solutions to linear operator equations when the
  data are noisy, \emph{SIAM journal on numerical analysis}. {\bf 14} \penalty0
  (1977), \penalty0 651--667.

\bibitem{nussbaum1996asymptotic}
M.~Nussbaum, Asymptotic equivalence of density estimation and {G}aussian white
  noise, \emph{The Annals of Statistics}.  \penalty0 (1996), \penalty0
  2399--2430.

\bibitem{doumic2020estimating}
M.~Doumic, A.~Olivier, and L.~Robert, Estimating the division rate from
  indirect measurements of single cells., \emph{Discrete \& Continuous
  Dynamical Systems-Series B}. {\bf 25} \penalty0 (2020).

\bibitem{carrillo2019escalator}
J.~A. Carrillo, P.~Gwiazda, K.~Kropielnicka, and A.~K. Marciniak-Czochra, The
  escalator boxcar train method for a system of age-structured equations in the
  space of measures, \emph{SIAM Journal on Numerical Analysis}. {\bf 57}
  \penalty0 (2019), \penalty0 1842--1874.

\bibitem{ramdas2017wasserstein}
A.~Ramdas, N.~G. Trillos, and M.~Cuturi, On {W}asserstein two-sample testing
  and related families of nonparametric tests, \emph{Entropy}. {\bf 19}
  \penalty0 (2017), \penalty0 47.

\bibitem{doudoufer}
A.~Fermanian, C.~Doucet, M.~Hoffmann, L.~Robert, and M.~Doumic,
\\ Celldivision: a plateform to select a best-fit cell division cycle
  model.\\
In progress, prototype plateform available at  \\
  https://celldivision.paris.inria.fr/welcome/.

\bibitem{soifer2016single}
I.~Soifer, L.~Robert, and A.~Amir, Single-cell analysis of growth in budding
  yeast and bacteria reveals a common size regulation strategy, \emph{Current
  Biology}. {\bf 26} \penalty0 (2016), \penalty0 356--361.

\end{thebibliography}

\end{document}